\documentclass[a4paper,11pt,reqno]{customart}

\usepackage[utf8x]{inputenc}
\usepackage[margin=1in]{geometry}
\usepackage{verbatim}
\usepackage{color}
\usepackage[table]{xcolor}
\usepackage{listings}
\usepackage{amssymb,amsfonts,amsthm,amsmath}
\usepackage{url}
\usepackage{enumerate}
\usepackage{todonotes}
\usepackage[all]{xy}
\usepackage{multirow}
\usepackage{changepage}
\usepackage{attachfile}
\usepackage{stmaryrd}
\usepackage{footnote}
\makesavenoteenv{tabular}
\usepackage{hyperref}

\makeatletter
\newcommand\footnoteref[1]{\protected@xdef\@thefnmark{\ref{#1}}\@footnotemark}
\makeatother

\hypersetup{
	pdfborder={0 0 0}
}

\definecolor{grey}{rgb}{0.95,0.95,0.95}
\definecolor{green}{rgb}{0.2,0.6,0.4}

\renewcommand{\O}{\textbf{0}}

\newcommand{\imp}{\rightarrow}

\newcommand{\biimp}{\leftrightarrow}
\newcommand{\dbf}{\mathbf{d}}
\newcommand{\Ab}{\mathbb{A}}

\newcommand{\Pb}{\mathbb{P}}
\newcommand{\Tb}{\mathbb{T}}
\newcommand{\TPb}{\mathbb{U}}

\newcommand{\Ccal}{\mathcal{C}}
\newcommand{\Dcal}{\mathcal{D}}
\newcommand{\Fcal}{\mathcal{F}}
\newcommand{\Gcal}{\mathcal{G}}

\newcommand{\uh}{{\upharpoonright}}

\renewcommand{\setminus}{\smallsetminus}

\newcommand{\set}{\mathrm{set}}
\newcommand{\parts}{\mathrm{parts}}

\newcommand{\dom}{\mathrm{dom}}

\newcommand{\s}[1]{\ensuremath{\sf{#1}}}

\DeclareMathOperator{\rca}{\s{RCA}_0}
\DeclareMathOperator{\aca}{\s{ACA}_0}
\DeclareMathOperator{\wkl}{\s{WKL}_0}

\DeclareMathOperator{\bst}{\s{B}\Sigma^0_2}

\DeclareMathOperator{\rt}{\s{RT}}
\DeclareMathOperator{\srt}{\s{SRT}}

\DeclareMathOperator{\rrt}{\s{RRT}}

\DeclareMathOperator{\ads}{\s{ADS}}
\DeclareMathOperator{\sads}{\s{SADS}}

\DeclareMathOperator{\coh}{\s{COH}}

\DeclareMathOperator{\sts}{\s{STS}}

\DeclareMathOperator{\emo}{\s{EM}}
\DeclareMathOperator{\semo}{\s{SEM}}

\DeclareMathOperator{\amt}{\s{AMT}}

\newcommand{\Apx}{\operatorname{Apx}}
\newcommand{\Ext}{\operatorname{Ext}}

\newtheoremstyle{custom}
  {10pt}
  {10pt}
  {\normalfont}
  {}
  {\bfseries}
  {}
  { }
  {}

\theoremstyle{custom}

\usepackage{xcolor}	
\usepackage{soul}

\newtheorem{theorem}{Theorem}[section]
\newtheorem{lemma}[theorem]{Lemma}
\newtheorem{definition}[theorem]{Definition}

\newtheorem{corollary}[theorem]{Corollary}

\title{Controlling iterated jumps of solutions to combinatorial problems}
\author{
  Ludovic Patey
}

\begin{document}

\begin{abstract}
Among the Ramsey-type hierarchies, namely, Ramsey's theorem,
the free set, the thin set and the rainbow Ramsey theorem,
only Ramsey's theorem is known to collapse in reverse mathematics.
A promising approach to show the strictness of the hierarchies
would be to prove that every computable instance at level~$n$
has a low$_n$ solution. In particular, this requires
effective control of iterations of the Turing jump.

In this paper, we design some variants of Mathias forcing
to construct solutions to cohesiveness, the Erd\H{o}s-Moser theorem
and stable Ramsey's theorem for pairs, while controlling their iterated jumps.
For this, we define forcing relations which, unlike Mathias forcing,
have the same definitional complexity as the formulas they force.
This analysis enables us to answer two questions of Wei Wang,
namely, whether cohesiveness and the Erd\H{o}s-Moser theorem
admit preservation of the arithmetic hierarchy,
and can be seen as a step towards the resolution
of the strictness of the Ramsey-type hierarchies. 
\end{abstract}

\maketitle

\section{Introduction}

Effective forcing is a very powerful tool in the computational analysis
of mathematical statements. In this framework, lowness is achieved by
deciding formulas during the forcing argument, while ensuring that the whole construction
remains effective. Thus, the definitional strength of the forcing relation
is very sensitive in effective forcing.
We present a new forcing argument enabling one to control iterated jumps
of solutions to Ramsey-type theorems. Our main motivation is reverse mathematics.

\subsection{Reverse mathematics}

Reverse mathematics is a vast mathematical program
whose goal is to classify ordinary theorems in terms of their provability strength.
It uses the framework of subsystems of second order arithmetic,
which is sufficiently rich to express in a natural way many theorems.
The base system, $\rca$ standing for Recursive Comprehension Axiom,
contains the basic first order Peano arithmetic together with the~$\Delta^0_1$
comprehension scheme and the~$\Sigma^0_1$ induction scheme.
Thanks to the equivalence between~$\Delta^0_1$-definable sets
and computable sets, $\rca$ can be considered as capturing
``computable mathematics''. The proof-theoretic
analysis of the theorems in reverse mathematics is therefore
closely related to their computability-theoretic content. See Simpson~\cite{Simpson2009Subsystems}
for a formal introduction to reverse mathematics.

Early reverse mathematics results support two main empirical observations:
First, many ordinary (i.e.\ non set-theoretic) theorems require very weak set existence axioms.
Second, most of those theorems are in fact \emph{equivalent} to one of four
main subsystems, which together with $\rca$ are known as the ``Big Five''.
However, among the theorems studied in reverse mathematics, a notable class of theorems
fails to support those observations, namely, Ramsey-type theorems.
This article focuses on consequences of Ramsey's theorem
below the arithmetic comprehension axiom ($\aca$). 
See Hirschfeldt~\cite{Hirschfeldt2015Slicing} for a gentle introduction to
the reverse mathematics below~$\aca$.

\subsection{Controlling iterated jumps}

Among the hierarchies of combinatorial principles, namely, Ramsey's theorem~\cite{Jockusch1972Ramseys,Seetapun1995strength,Cholak2001strength}, 
the rainbow Ramsey theorem~\cite{Csima2009strength,Wang2014Cohesive,Patey2015Somewhere}, 
and the free set and thin set theorems~\cite{Cholak2001Free,Wang2014Some} --
only Ramsey's theorem is known to collapse within the framework of reverse mathematics.
The above mentioned hierarchies satisfy the lower bounds of Jockusch~\cite{Jockusch1972Ramseys}, 
that is, there exists a computable instance at every level~$n \geq 2$ with no $\Sigma^0_n$ solution.
Thus, a possible strategy for proving that a hierarchy is strict consists
of showing the existence, for every computable instance at level~$n$, of a low${}_n$ solution.

The solutions to combinatorial principles are often built by Mathias forcing, whose forcing
relation is known to be of higher definitional strength than the formula it forces~\cite{Cholak2014Generics}.
Therefore there is a need for new notions of forcing with a better-behaving forcing relation.
In this paper, we design three notions of forcing to construct solutions
to cohesiveness, the Erd\H{o}s-Moser theorem and stable Ramsey's theorem for pairs, respectively.
We define a forcing relation with the expected properties, and which formalises
the first and the second jump control of Cholak, Jockusch and Slaman~\cite{Cholak2001strength}.
This can be seen as a step toward the resolution the strictness of the Ramsey-type hierarchies.
We take advantage of this new analysis of Ramsey-type statements
to prove two conjectures of Wang about the preservation
of the arithmetic hierarchy. 

\subsection{Preservation of the arithmetic hierarchy}

The notion of preservation of the arithmetic hierarchy has
been introduced by Wang in~\cite{Wang2014Definability}, in the context of a new analysis
of principles in reverse mathematics in terms of their definitional strength.

\begin{definition}[Preservation of definitions]\ 
\begin{itemize}
	\item[1.]
A set $Y$ \emph{preserves $\Xi$-definitions} (relative to $X$) for $\Xi$ among $\Delta^0_{n+1}, \Pi^0_n, \Sigma^0_n$ where $n > 0$, if every properly $\Xi$ (relative to $X$) set is properly $\Xi$ relative to $Y$ ($X \oplus Y$). $Y$ \emph{preserves the arithmetic hierarchy} (relative to $X$) if $Y$ preserves $\Xi$-definitions (relative to $X$) for all $\Xi$ among $\Delta^0_{n+1}, \Pi^0_n, \Sigma^0_n$ where $n > 0$.

	\item[2.]
Suppose that $\Phi = (\forall X) (\exists Y) \varphi(X,Y)$ and $\varphi$ is arithmetic. $\Phi$ \emph{admits preservation of $\Xi$-definitions} if for each $Z$ and $X \leq_T Z$ there exists $Y$ such that $Y$ preserves $\Xi$-definitions relative to $Z$ and $\varphi(X,Y)$ holds. $\Phi$ \emph{admits preservation of the arithmetic hierarchy} if for each $Z$ and $X \leq_T Z$ there exists $Y$ such that $Y$ preserves the arithmetic hierarchy relative to $Z$ and $\varphi(X,Y)$ holds.
\end{itemize}
\end{definition}

The preservation of the arithmetic hierarchy seems closely related
to the problem of controlling iterated jumps of solutions to combinatorial problems.
Indeed, a proof of such a preservation
usually consists of noticing that the forcing relation has the same strength
as the formula it forces, and then deriving a diagonalization from it. 
See Lemma~\ref{lem:diagonalization} for a case-in-point.
Wang proved in~\cite{Wang2014Definability} that weak König's lemma ($\wkl$),
the rainbow Ramsey theorem for pairs ($\rrt^2_2$) and the atomic model theorem ($\amt$)
admit preservation of the arithmetic hierarchy. He conjectured
that this is also the case for cohesiveness and the Erd\H{o}s-Moser theorem.
We prove the two conjectures via the following concatenation of Theorems~\ref{thm:coh-preservation-arithmetic-hierarchy}
and~\ref{thm:em-preserves-arithmetic}, where~$\coh$ stands for cohesiveness and~$\emo$ for the Erd\H{o}s-Moser theorem.

\begin{theorem}
$\coh$ and~$\emo$ admit preservation of the arithmetic hierarchy.
\end{theorem}

\subsection{Definitions and notation}

Fix an integer $k \in \omega$.
A \emph{string} (over $k$) is an ordered tuple of integers $a_0, \dots, a_{n-1}$
(such that $a_i < k$ for every $i < n$). The empty string is written $\epsilon$. A \emph{sequence}  (over $k$)
is an infinite listing of integers $a_0, a_1, \dots$ (such that $a_i < k$ for every $i \in \omega$).
Given $s \in \omega$,
$k^s$ is the set of strings of length $s$ over~$k$ and
$k^{<s}$ is the set of strings of length $<s$ over~$k$. As well,
$k^{<\omega}$ is the set of finite strings over~$k$
and $k^{\omega}$ is the set of sequences (i.e. infinite strings)
over~$k$. 
Given a string $\sigma \in k^{<\omega}$, we use $|\sigma|$ to denote its length.
Given two strings $\sigma, \tau \in k^{<\omega}$, $\sigma$ is a \emph{prefix}
of $\tau$ (written $\sigma \preceq \tau$) if there exists a string $\rho \in k^{<\omega}$
such that $\sigma \rho = \tau$. Given a sequence $X$, we write $\sigma \prec X$ if
$\sigma = X \uh n$ for some $n \in \omega$.
A \emph{binary string} (resp. real) is a \emph{string} (resp. sequence) over $2$.
We may identify a real with a set of integers by considering that the real is its characteristic function.

A tree $T \subseteq k^{<\omega}$ is a set downward-closed under the prefix relation.
A \emph{binary} tree is a set $T \subseteq 2^{<\omega}$.
A set $P \subseteq \omega$ is a \emph{path} through~$T$ if for every $\sigma \prec P$,
$\sigma \in T$. A string $\sigma \in k^{<\omega}$ is a \emph{stem} of a tree $T$
if every $\tau \in T$ is comparable with~$\sigma$.
Given a tree $T$ and a string $\sigma \in T$,
we denote by $T^{[\sigma]}$ the subtree $\{\tau \in T : \tau \preceq \sigma \vee \tau \succeq \sigma\}$.

Given two sets $A$ and $B$, we denote by $A < B$ the formula
$(\forall x \in A)(\forall y \in B)[x < y]$.
We write $A \subseteq^{*} B$ to mean that $A - B$ is finite, that is, 
$(\exists n)(\forall a \in A)(a \not \in B \imp a < n)$.
A \emph{Mathias condition} is a pair $(F, X)$
where $F$ is a finite set, $X$ is an infinite set
and $F < X$. A condition $(F_1, X_1)$ \emph{extends } $(F, X)$ (written $(F_1, X_1) \leq (F, X)$)
if $F \subseteq F_1$, $X_1 \subseteq X$ and $F_1 \setminus F \subset X$.
A set $G$ \emph{satisfies} a Mathias condition $(F, X)$
if $F \subset G$ and $G \setminus F \subseteq X$.

\section{Cohesiveness preserves the arithmetic hierarchy}

Cohesiveness plays a central role in reverse mathematics.
It appears naturally in the standard proof of Ramsey's theorem, as a preliminary
step to reduce an instance of Ramsey's theorem over $(n+1)$-tuples into a non-effective instance over $n$-tuples.

\begin{definition}[Cohesiveness]
An infinite set $C$ is $\vec{R}$-cohesive for a sequence of sets $R_0, R_1, \dots$
if for each $i \in \omega$, $C \subseteq^{*} R_i$ or $C \subseteq^{*} \overline{R_i}$.
A set $C$ is \emph{cohesive} (resp. \emph{p-cohesive}, \emph{r-cohesive}) if it is $\vec{R}$-cohesive where
$\vec{R}$ is the sequence of all the c.e. sets (resp. primitive recursive sets, computable sets).
$\coh$ is the statement ``Every uniform sequence of sets $\vec{R}$
admits an infinite $\vec{R}$-cohesive set.''
\end{definition}

Mileti~\cite{Mileti2004Partition} and Jockusch and Lempp [unpublished]
proved that $\coh$ is a consequence of Ramsey's theorem for pairs over $\rca$.
The computational power of $\coh$ is relatively well understood.
A Turing degree $\dbf$ \emph{bounds} $\coh$ if every computable sequence of sets~$R_0, R_1, \dots$,
has an $\vec{R}$-cohesive set bounded by $\dbf$.
Jockusch and Stephan characterized in~\cite{Jockusch1993cohesive}
the degrees bounding $\coh$ as the degrees whose jump is PA relative to $\emptyset'$.
The author~\cite{Patey2015weakness} extended this characterization to an instance-wise correspondance
between cohesiveness and the statement ``For every~$\Delta^0_2$ tree~$T$, there is a set whose jump computes a path through~$T$''.
Wang~\cite{Wang2014Definability} conjectured that $\coh$ admits preservation of the arithmetic hierarchy.
We prove his conjecture by using a new forcing argument.

\begin{theorem}\label{thm:coh-preservation-arithmetic-hierarchy}
$\coh$ admits preservation of the arithmetic hierarchy.
\end{theorem}

Before proving Theorem~\ref{thm:coh-preservation-arithmetic-hierarchy},
we state an immediate corollary.

\begin{corollary}
There exists a cohesive set preserving the arithmetic hierarchy.
\end{corollary}
\begin{proof}
Jockusch~\cite{Jockusch1972Degreesa} proved that every PA degree computes a sequence of sets
containing, among others, all the computable sets.
Wang proved in~\cite{Wang2014Definability} that $\wkl$ preserves the arithmetic hierarchy. 
Therefore there exists a uniform sequence of sets $\vec{R}$ containing all the computable sets 
and preserving the arithmetic hierarchy. By Theorem~\ref{thm:coh-preservation-arithmetic-hierarchy}
relativized to $\vec{R}$, there exists an infinite $\vec{R}$-cohesive set $C$ preserving the arithmetic hierarchy relative to~$\vec{R}$.
In particular $C$ is r-cohesive and preserves the arithmetic hierarchy. By~\cite{Jockusch1993cohesive},
the degrees of r-cohesive and cohesive sets coincide. Therefore $C$ computes a cohesive set which preserves the arithmetic hierarchy.
\end{proof}

Given a uniformly computable sequence of sets~$R_0, R_1, \dots$,
the construction of an $\vec{R}$-cohesive set is usually done
with computable Mathias forcing, that is, using conditions~$(F, X)$
in which~$X$ is computable.
The construction starts with~$(\emptyset, \omega)$ and interleaves two kinds of steps.
Given some condition~$(F,X)$,
\begin{itemize}
	\item[(S1)] the \emph{extension} step consists of taking an element $x$ from $X$ and adding it to~$F$,
	therefore forming the extension $(F \cup \{x\}, X \setminus [0,x])$;
	\item[(S2)] the \emph{cohesiveness} step consists of deciding which one of $X \cap R_i$
	and $X \cap \overline{R}_i$ is infinite, and taking the chosen one as the new reservoir.
\end{itemize} 

Cholak, Dzhafarov, Hirst and Slaman~\cite{Cholak2014Generics} studied the definitional complexity 
of the forcing relation for computable Mathias forcing.
They proved that it has good definitional properties for the first jump, but not for iterated jumps.
Indeed, given a computable Mathias condition~$c = (F, X)$ and a $\Sigma^0_1$ formula~$(\exists x)\varphi(G, x)$,
one can $\emptyset'$-effectively decide whether there is an extension~$d$ forcing~$(\exists x)\varphi(G, x)$ by asking the following question:
\begin{quote}
Is there an extension~$d = (E, Y) \leq c$ and some~$n \in \omega$ such that~$\varphi(E, n)$ holds?
\end{quote}
If there is such an extension, then we can choose it to be a \emph{finite extension}, that is, such that~$Y =^* X$.
Therefore, the question is $\Sigma^{0,X}_1$.
Consider now a $\Pi^0_2$ formula~$(\forall x)(\exists y)\varphi(G, x, y)$. The question becomes
\begin{quote}
For every extension~$d \leq c$ and every~$m \in \omega$, is there some extension~$e = (E, Y) \leq d$
and some~$n \in \omega$ such that~$\varphi(E, m, n)$ holds?
\end{quote}
In this case, the extension~$d$ is not usually a finite extension and therefore the question cannot
be presented in a $\Pi^0_2$ way. In particular, the formula ``$Y$ is an infinite subset of~$X$'' is definitionally complex. 
In general, deciding iterated jumps of a generic set requires to be able to talk about the future
of a given condition, and in particular to describe by simple means the formula ``$d$ is a valid condition''
and the formula ``$d$ is an extension of~$c$''.

Thankfully, in the case of cohesiveness, we do not need the full generality of 
the computable Mathias forcing. Indeed, the reservoirs
have a very special shape. After the first application of stage~(S2), 
the set $X$ is, up to finite changes, of the form $\omega \cap R_0$
or $\omega \cap \overline{R_0}$. After the second application of (S2), it is in one of the following forms: $\omega \cap R_0 \cap R_1$,
$\omega \cap R_0 \cap \overline{R}_1$, $\omega \cap \overline{R}_0 \cap R_1$,
$\omega \cap \overline{R}_0 \cap \overline{R}_1$, and so on.
More generally, after $n$ applications of (S2), a condition~$c = (F, X)$ is characterized
by a pair~$(F, \sigma)$ where $\sigma$ is a string of length~$n$ representing the choices made during (S2).
Given a string~$\sigma \in 2^{<\omega}$, let~$R_\sigma = \bigcap_{\sigma(i) = 0} \overline{R}_i \bigcap_{\sigma(i) = 1} R_i$.
In particular, $R_\varepsilon = \omega$, where~$\varepsilon$ is the empty string.

Even within this restricted partial order, the decision of the~$\Pi^0_2$ formula remains too complicated 
sinces it requires deciding if $R_\sigma$ is infinite.
However, notice that the~$\sigma$'s such that~$R_\sigma$ is infinite are exactly the initial
segments of the $\Pi^{0, \emptyset'}_1$ class~$\Ccal(\vec{R})$ defined as the collection of the reals~$X$
such that~$R_\sigma$ has more than~$|\sigma|$ elements for every~$\sigma \prec X$.
We can therefore use a compactness argument at the second level to decrease the definitional strength of the forcing relation,
as Wang~\cite{Wang2014Definability} did for weak K\"onig's lemma.

\subsection{The forcing notion}

We let~$\Tb$ denote the collection of all the infinite $\emptyset'$-primitive recursive trees~$T$
such that~$[T] \subseteq \Ccal(\vec{R})$. By $\emptyset'$-primitive recursive, we mean the class of functions
Add a comment to this line
obtained by adding the characteristic function of $\emptyset'$ to the basic primitive recursive functions,
and closing under the standard primitive recursive operations.
Note that~$\Tb$ is a computable set. 
Given two finite sets~$E, F$ and some string~$\sigma \in 2^{<\omega}$,
we write~$E \leq_\sigma F$ to say that $F \subseteq E \subseteq F \cup (R_\sigma \cap [\max F, \infty))$.
In other words, $E \leq_\sigma F$ if and only if $(E, R_\sigma)$ is a valid Mathias extension
of $(F, R_\sigma)$, where~$R_\sigma$ might be finite.
We are now ready to defined our partial order.

\begin{definition}
Let $\Pb$ be the partial order whose conditions are tuples $(F, \sigma, T)$
where~$F \subseteq \omega$ is a finite set, $\sigma \in 2^{<\omega}$ and~$T \in \Tb$ with stem~$\sigma$.
A condition~$d = (E, \tau, S)$ \emph{extends} $c = (F, \sigma, T)$ (written~$d \leq c$)
if~$E \leq_\sigma F$, $\tau \succeq \sigma$ and $S \subseteq T$.
\end{definition}

Given a condition~$c = (F, \sigma, T)$, the string~$\sigma$ imposes a finite restriction on the possible extensions of the set~$F$.
The condition $c$ intuitively denotes the Mathias condition $(F, R_\sigma \cap (\max F, \infty))$
with some additional constraints on the extensions of~$\sigma$ represented by the tree~$T$.
Accordingly, set~$G$ \emph{satisfies} $(F, \sigma, T)$ if it satisfies
the induced Mathias condition, that is, if~$F \subseteq G \subseteq F \cup (R_\sigma \cap (\max F, \infty))$. 
We let~$\Ext(c)$ be the collection of all the extensions of~$c$.

Note that although we did not explicitely require $R_\sigma$ to be infinite,
this property holds for every condition~$(F, \sigma, T) \in \Pb$.
Indeed, since $[T] \subseteq \Ccal(\vec{R})$, then $R_\tau$ is infinite
for every extensible node~$\tau \in T$. Since $\sigma$ is a stem of~$T$,
it is extensible and therefore $R_\sigma$ is infinite.

\subsection{Preconditions and forcing $\Sigma^0_1$ ($\Pi^0_1$) formulas}

When forcing complex formulas, we need to be able to consider all possible extensions of some condition~$c$.
Checking that some $d = (E, \tau, S)$ is a valid condition extending $c$ 
requires to decide whether the $\emptyset'$-p.r.\ tree $S$ is infinite,
which is a $\Pi^0_2$ question.
At some point, we will need to decide a $\Sigma^0_1$ formula without having enough computational power
to check that the tree part is infinite (see clause~(ii) of Definition~\ref{def:forcing-condition}).
As the tree part of a condition is not accurate for such formulas,
we may define the corresponding forcing relation over 
a weaker notion of condition where the tree is not required to be infinite.

\begin{definition}[Precondition]
A \emph{precondition} is a condition~$(F, \sigma, T)$ without the assumption that $T$ is infinite.
\end{definition}

In particular, $R_\sigma$ may be a finite set. The notion of condition extension can be generalized
to the preconditions. The set of all preconditions is computable, contrary to the set~$\Pb$.
Given a precondition~$c = (F, \sigma, T)$, we denote by~$\Ext_1(c)$ the set of all preconditions~$(E, \tau, S)$
extending~$c$ such that~$\tau = \sigma$ and~$T = S$. Here, $T = S$ in a strong sense,
that is, the Turing indices of~$T$ and~$S$ are the same. This fact is used in clause~a) of Lemma~\ref{lem:extension-complexity}.
We let~$\Ab$ denote the collection of all the finite sets of integers. 
The set $\Ab$ can be thought of as representing the set of finite approximations of the generic set~$G$. 
We also fix a uniformly computable enumeration $\Ab_0 \subseteq \Ab_1 \subseteq \dots$
of finite subsets of~$\Ab$ such that~$\bigcup_s \Ab_s = \Ab$.
We denote by~$\Apx(c)$ the set $\{ E \in \Ab : (E, \sigma, T) \in \Ext_1(c) \}$.
In particular, $\Apx(c)$ is collection of all finite sets~$E$ satisfying~$c$, that is, 
$\Apx(c) = \{ E \in \Ab : E \leq_\sigma F \}$.
Last, we let~$\Apx_s(c) = \Apx(c) \cap \Ab_s$. We start by proving a few trivial statements.

\begin{lemma}\label{lem:basic-statements}
Fix a precondition $c = (F, \sigma, T)$.
\begin{itemize}
	\item[1)] If $c$ is a condition then $\Ext_1(c) \subseteq \Ext(c)$.
	\item[2)] If $c$ is a condition then $\Apx(c) = \{ E : (E, \tau, S) \in \Ext(c) \}$.
	\item[3)] If $d$ is a precondition extending $c$ then 
	$\Apx(d) \subseteq \Apx(c)$ and~$\Apx_s(d) \subseteq \Apx_s(c)$.
\end{itemize}
\end{lemma}
\begin{proof}\ 
\begin{itemize}
	\item[1)]
By definition, if $c$ is a condition, then $T$ is infinite.
If $d \in \Ext_1(c)$ then $d = (E, \sigma, T)$ for some $E \in \Apx(c)$.
As $d$ is a precondition and $T$ is infinite, $d$ is a condition.

	\item[2)] By definition,
	$\Apx(c) = \{ E : (E, \sigma, T) \in \Ext_1(c) \} \subseteq \{ E : (E, \tau, S) \in \Ext(c) \}$.
	In the other direction, fix an extension $(E, \tau, S) \in \Ext(c)$.
	By definition of an extension, $E \leq_\tau F$, so $E \leq_\sigma F$.
	Therefore $(E, \sigma, T) \in \Ext_1(c)$ and by definition of $\Apx(c)$, $E \in \Apx(c)$.

	\item[3)]
	Fix some $(E, \tau, S) \in \Ext_1(d)$. As $d$ extends $c$, $\tau \succeq \sigma$.
	By definition of an extension, $E \leq_\tau F$, so $E \leq_\sigma F$,
	hence $(E, \sigma, T) \in \Ext_1(c)$.
	Therefore $\Apx(d) = \{ E : (E, \tau, S) \in \Ext_1(d) \} \subseteq
\{ E : (E, \sigma, T) \in \Ext_1(c) \} = \Apx(c)$. For any~$s \in \omega$,
	$\Apx_s(d) = \Apx(d) \cap \Ab_s \subseteq \Apx(c) \cap \Ab_s = \Apx_s(c)$.
\end{itemize}
\end{proof}

Note that although the extension relation has been generalized to preconditions, 
$\Ext(c)$ is defined to be the set of all the \emph{conditions} extending $c$. 
In particular, if $c$ is a precondition which is not a condition, $\Ext(c) = \emptyset$,
whereas at least $c \in \Ext_1(c)$.
This is why clause 1 of Lemma~\ref{lem:basic-statements} gives the useful information
that whenever~$c$ is a true condition, so are the members of $\Ext_1(c)$.

\begin{definition}\label{def:forcing-precondition}
Fix a precondition~$c = (F, \sigma, T)$ and a $\Sigma^0_0$ formula $\varphi(G, x)$.
\begin{itemize}
	\item[(i)] $c \Vdash (\exists x)\varphi(G, x)$ iff $\varphi(F, w)$ holds for some~$w \in \omega$
	\item[(ii)] $c \Vdash (\forall x)\varphi(G, x)$ iff $\varphi(E, w)$ holds for every~$w \in \omega$
	and every set~$E \in \Apx(c)$.
\end{itemize}
\end{definition}

As explained, $\sigma$ restricts the possible extensions of the set~$F$ (see clause 3 of Lemma~\ref{lem:basic-statements}), 
so this forcing notion is stable by condition extension. The tree $T$ itself restricts the possible extensions of $\sigma$,
but has no effect in deciding a $\Sigma^0_1$ formula (Lemma~\ref{lem:tree-no-effect-first-level}).

The following trivial lemma expresses the fact that the tree part of a precondition
has no effect in the forcing relation for a $\Sigma^0_1$ or $\Pi^0_1$ formula.

\begin{lemma}\label{lem:tree-no-effect-first-level}
Fix two preconditions $c = (F, \sigma, T)$ and $d = (F, \sigma, S)$, and some $\Sigma^0_1$ or $\Pi^0_1$
formula~$\varphi(G)$.
$$
c \Vdash \varphi(G) \hspace{10pt} \mbox{ if and only if } \hspace{10pt} d \Vdash \varphi(G)
$$
\end{lemma}
\begin{proof}
Simply notice that the tree part of the condition does not occur in the definition
of the forcing relation, and that $\Apx(c) = \Apx(d)$.
\end{proof}

As one may expect, the forcing relation
for a precondition is closed under extension.

\begin{lemma}\label{lem:forcing-extension-precondition}
Fix a precondition $c$ and a $\Sigma^0_1$ or $\Pi^0_1$ formula $\varphi(G)$.
If $c \Vdash \varphi(G)$ then for every precondition $d \leq c$, $d \Vdash \varphi(G)$.
\end{lemma}
\begin{proof}
Fix a precondition $c = (F, \sigma, T)$ such that $c \Vdash \varphi(G)$ and an extension $d = (E, \tau, S) \leq c$.
\begin{itemize}
  \item
If $\varphi \in \Sigma^0_1$ then $\varphi(G)$ can be expressed
as $(\exists x)\psi(G, x)$ where $\psi \in \Sigma^0_0$.
As $c \Vdash \varphi(G)$, then by clause (i) of Definition~\ref{def:forcing-precondition},
there exists a $w \in \omega$ such that $\psi(F, w)$ holds. 
By definition of $d \leq c$, $E \leq_\sigma F$, so $\psi(E, w)$ holds,
hence $d \Vdash \varphi(G)$.

\item
If $\varphi \in \Pi^0_1$ then $\varphi(G)$ can be expressed
as $(\forall x)\psi(G, x)$ where $\psi \in \Sigma^0_0$.
As $c \Vdash \varphi(G)$, then by clause (ii) of Definition~\ref{def:forcing-precondition},
for every $w \in \omega$ and every $H \in \Apx(c)$, $\varphi(H, w)$ holds.
By clause 3 of Lemma~\ref{lem:basic-statements}, $\Apx(d) \subseteq \Apx(c)$ so $d \Vdash \varphi(G)$.
\end{itemize}
\end{proof}

\subsection{Forcing higher formulas}

We are now able to define the forcing relation for any arithmetic formula.
The forcing relation for arbitrary arithmetic formulas is induced by the forcing 
relation for~$\Sigma^0_1$ formulas. However, the definitional strength of the resulting
relation is too high with respect to the formula it forces. We therefore
design a custom relation with better definitional properties,
and which still preserve the expected properties of a forcing relation,
that is, the density of the set of conditions forcing a formula or its negation,
and the preservation of the forced formulas under condition extension.

\begin{definition}\label{def:forcing-condition}
Let~$c = (F, \sigma, T)$ be a condition and~$\varphi(G)$ be an arithmetic formula.
\begin{itemize}
	\item[(i)] If $\varphi(G) = (\exists x)\psi(G, x)$ where~$\psi \in \Pi^0_{n+1}$ then~$c \Vdash \varphi(G)$
	iff there is a $w < |\sigma|$ such that~$c \Vdash \psi(G, w)$
	\item[(ii)] If $\varphi(G) = (\forall x)\psi(G, x)$ where~$\psi \in \Sigma^0_1$ then $c \Vdash \varphi(G)$
	iff for every~$\tau \in T$, every $E \in \Apx_{|\tau|}(c)$
	and every~$w < |\tau|$, $(E, \tau, T^{[\tau]}) \not \Vdash \neg \psi(G, w)$
	\item[(iii)] If~$\varphi(G) = \neg \psi(G, x)$ where~$\psi \in \Sigma^0_{n+3}$ then $c \Vdash \varphi(G)$
	iff $d \not \Vdash \psi(G)$ for every~$d \leq c$.
\end{itemize}
\end{definition}

Note that in clause (ii) of Definition~\ref{def:forcing-condition},
there may be some $\tau \in T$ such that $T^{[\tau]}$ is finite, hence $(E, \tau, T^{[\tau]})$
is not necessarily a condition. This is where we use the generalization of forcing of $\Sigma^0_1$ and $\Pi^0_1$
formulas to preconditions. We now prove that this relation enjoys the main properties
of a forcing relation.

\begin{lemma}\label{lem:forcing-extension}
Fix a condition $c$ and an arithmetic formula $\varphi(G)$.
If $c \Vdash \varphi(G)$ then for every condition $d \leq c$, $d \Vdash \varphi(G)$.
\end{lemma}
\begin{proof}
We prove by induction over the complexity of the formula $\varphi(G)$
that for every condition $c$, if $c \Vdash \varphi(G)$
then for every condition $d \leq c$, $d \Vdash \varphi(G)$.
Fix a condition $c = (F, \sigma, T)$ such that $c \Vdash \varphi(G)$ and an extension $d = (E, \tau, S)$.
\begin{itemize}
  \item If $\varphi \in \Sigma^0_1 \cup \Pi^0_1$ then it follows from Lemma~\ref{lem:forcing-extension-precondition}.
  
	\item If $\varphi \in \Sigma^0_{n+2}$ then $\varphi(G)$ can be expressed as $(\exists x)\psi(G, x)$ 
  where $\psi \in \Pi^0_{n+1}$.
  By clause~(i) of Definition~\ref{def:forcing-condition}, there exists a $w \in \omega$
  such that $c \Vdash \psi(G, w)$. By induction hypothesis, $d \Vdash \psi(G, w)$
  so by clause~(i) of Definition~\ref{def:forcing-condition}, $d \Vdash \varphi(G)$.

  \item If $\varphi \in \Pi^0_2$ then $\varphi(G)$ can be expressed as $(\forall x)\psi(G, x)$ where $\psi \in \Sigma^0_1$.
	By clause~(ii) of Definition~\ref{def:forcing-condition}, for every $\rho \in T$, every $w < |\rho|$,
	and every $H \in \Apx_{|\rho|}(c)$, $(H, \rho, T^{[\rho]}) \not \Vdash \neg \psi(G, w)$.
  As $S \subseteq T$ and $\Apx(d) \subseteq \Apx(c)$, for every $\rho \in S$, every $w < |\rho|$,
  and every $H \in \Apx_{|\rho|}(d)$, $(H, \rho, T^{[\rho]}) \not \Vdash \neg \psi(G, w)$.
  By Lemma~\ref{lem:tree-no-effect-first-level}, $(H, \rho, S^{[\rho]}) \not \Vdash \neg \psi(G, w)$
  hence by clause~(ii) of Definition~\ref{def:forcing-condition}, $d \Vdash \varphi(G)$.

  \item If $\varphi \in \Pi^0_{n+3}$ then $\varphi(G)$ can be expressed as $\neg \psi(G)$ where $\psi \in \Sigma^0_{n+3}$.
  By clause~(iii) of Definition~\ref{def:forcing-condition}, for every $e \in \Ext(c)$, $e \not \Vdash \psi(G)$.
  As $\Ext(d) \subseteq \Ext(c)$, for every $e \in \Ext(d)$, $e \not \Vdash \psi(G)$,
  so by clause~(iii) of Definition~\ref{def:forcing-condition}, $d \Vdash \varphi(G)$.
\end{itemize}
\end{proof}

\begin{lemma}\label{lem:forcing-dense}
For every arithmetic formula $\varphi$, the following set is dense
$$
\{c \in \Pb : c \Vdash \varphi(G) \mbox{ or } c \Vdash \neg \varphi(G) \}
$$
\end{lemma}
\begin{proof}
We prove by induction over $n > 0$ that if $\varphi$ is a $\Sigma^0_n$ ($\Pi^0_n$)
formula then the following set is dense
$$
\{c \in \Pb : c \Vdash \varphi(G) \mbox{ or } c \Vdash \neg \varphi(G) \}
$$
It suffices to prove it for the case where $\varphi$ is a $\Sigma^0_n$ formula,
as the case where $\varphi$ is a $\Pi^0_n$ formula is symmetric. Fix a condition $c = (F, \sigma, T)$.
\begin{itemize}
	\item In case $n = 1$, the formula $\varphi$ is of the form $(\exists x)\psi(G, x)$ where $\psi \in \Sigma^0_0$.
	Suppose there exist a $w \in \omega$ and a set $E \in \Apx(c)$ such that $\psi(E, w)$ holds.
	The precondition $d = (E, \sigma, T)$ is a condition extending $c$ by clause~1 of Lemma~\ref{lem:basic-statements}
	and by definition of $\Apx(c)$.
  Moreover $d \Vdash (\exists x)\psi(G, x)$ by clause~(i) of Definition~\ref{def:forcing-precondition} hence $d \Vdash \varphi(G)$.
	Suppose now that for every $w \in \omega$ and every $E \in \Apx(c)$, $\psi(E, w)$ does not hold.
	By clause~(ii) of Definition~\ref{def:forcing-precondition}, 
	$c \Vdash (\forall x)\neg \psi(G, x)$, hence $c \Vdash \neg \varphi(G)$.

	\item In case $n = 2$, the formula $\varphi$ is of the form $(\exists x)\psi(G, x)$ where
	$\psi \in \Pi^0_1$. Let
	$$
	S = \{ \tau \in T : (\forall w < |\tau|)(\forall E \in \Apx_{|\tau|}(c))(E, \tau, T^{[\tau]}) \not \Vdash \psi(G, w) \}
	$$
	The set $S$ is obviously $\emptyset'$-p.r. We prove that it is a subtree of $T$.
	Suppose that $\tau \in S$ and $\rho \preceq \tau$. 
	Fix a $w < |\rho|$ and $E \in \Apx_{|\rho|}(c)$. In particular $w < |\tau|$ and $E \in \Apx_{|\tau|}(c)$
	so $(E, \tau, T^{[\tau]}) \not \Vdash \psi(G, w)$. 
	Note that~$(E, \tau,T^{[\tau]})$ is a precondition extending~$(E, \rho, T^{[\rho]})$,
	so by the contrapositive of Lemma~\ref{lem:forcing-extension-precondition}, $(E, \rho, T^{[\rho]}) \not \Vdash \psi(G, w)$.
	Therefore $\rho \in S$. Hence $S$ is a tree, and as $S \subseteq T$, it is a subtree of $T$.
	
	If $S$ is infinite, then $d = (F, \sigma, S)$
	is an extension of $c$ such that for every $\tau \in S$, every $w < |\tau|$
	and every $E \in \Apx_{|\tau|}(c)$, $(E, \tau, T^{[\tau]}) \not \Vdash \psi(G, w)$.
	By Lemma~\ref{lem:tree-no-effect-first-level}, for every $E \in \Apx_{|\tau|}(c)$,
	$(E, \tau, S^{[\tau]}) \not \Vdash \psi(G, w)$ and by clause 3 of Lemma~\ref{lem:basic-statements}, 
	$\Apx_{|\tau|}(d) \subseteq \Apx_{|\tau|}(c)$.
	Therefore, by clause~(ii) of Definition~\ref{def:forcing-condition}, $d \Vdash (\forall x)\neg \psi(G, x)$
	so $d \Vdash \neg \varphi(G)$. If $S$ is finite, then pick some $\tau \in T \setminus S$
	such that $T^{[\tau]}$ is infinite. By choice of $\tau \in T \setminus S$, there exist a $w < |\tau|$
	and an $E \in \Apx_{|\tau|}(c)$ such that $(E, \tau, T^{[\tau]}) \Vdash \psi(G, w)$.
	$d = (E, \tau, T^{[\tau]})$ is a valid condition extending $c$
	and by clause~(i) of Definition~\ref{def:forcing-condition} $d \Vdash \varphi(G)$.

	\item In case $n > 2$, density follows from clause~(iii) of Definition~\ref{def:forcing-condition}.
\end{itemize}
\end{proof}

Any sufficiently generic filter~$\Fcal$ induces a unique generic real~$G$
defined by
$$
G = \bigcup \{ F \in \Ab : (F, \sigma, T) \in \Fcal \}
$$
The following lemma informally asserts that the forcing relation is \emph{sound} and \emph{complete}.
Sound because whenever a property is forced at some point, then this property actually holds over the generic real~$G$.
The forcing is also complete in that every property which holds over~$G$ is forced at some point
whenever the filter is sufficiently generic.

\begin{lemma}\label{lem:coh-holds-filter}
Suppose that $\Fcal$ is a sufficiently generic filter and let~$G$ be the corresponding generic real.
Then for each arithmetic formula $\varphi(G)$,
$\varphi(G)$ holds iff $c \Vdash \varphi(G)$ for some $c \in \Fcal$. 
\end{lemma}
\begin{proof}
We prove by induction over the complexity of the arithmetic formula $\varphi(G)$ that 
$\varphi(G)$ holds iff $c \Vdash \varphi(G)$ for some $c \in \Fcal$.
Note that thanks to Lemma~\ref{lem:forcing-dense}, it suffices to prove that if $c \Vdash \varphi(G)$ for some
$c \in \Fcal$ then $\varphi(G)$ holds. Indeed, conversely if $\varphi(G)$ holds,
then by genericity of $G$ either $c \Vdash \varphi(G)$ or $c \Vdash \neg \varphi(G)$ for some~$c \in \Fcal$,
but if $c \Vdash \neg \varphi(G)$ then $\neg \varphi(G)$ holds, contradicting 
the hypothesis. So $c \Vdash \varphi(G)$.

We proceed by case analysis on the formula~$\varphi$.
Note that in the above argument, the converse of the~$\Sigma$ case is proved assuming
the~$\Pi$ case. However, in our proof, we use the converse of the~$\Sigma^0_{n+3}$
case to prove the $\Pi^0_{n+3}$ case. We need therefore to prove the converse of the~$\Sigma^0_{n+3}$
case without Lemma~\ref{lem:forcing-dense}.
Fix a condition $c = (F, \sigma, T) \in \Fcal$ such that $c \Vdash \varphi(G)$. 
\begin{itemize}
	\item If $\varphi \in \Sigma^0_1$ then $\varphi(G)$ can be expressed as $(\exists x)\psi(G, x)$ where $\psi \in \Sigma^0_0$.
	By clause~(i) of Definition~\ref{def:forcing-precondition}, there exists a $w \in \omega$ such that
	$\psi(F, w)$ holds. As $F \subseteq G$ and~$G \setminus F \subseteq (\max F, \infty)$, then by continuity $\psi(G, w)$ holds, hence $\varphi(G)$ holds.
	
	\item If $\varphi \in \Pi^0_1$ then $\varphi(G)$ can be expressed as $(\forall x)\psi(G, x)$ where $\psi \in \Sigma^0_0$.
	By clause~(ii) of Definition~\ref{def:forcing-precondition}, for every $w \in \omega$
	and every $E \in \Apx(c)$, $\psi(E, w)$ holds. As $\{E \subset_{fin} G : E \supseteq F \} \subseteq \Apx(c)$,
	then for every $w \in \omega$, $\psi(G, w)$ holds, so $\varphi(G)$ holds.
 
	\item If $\varphi \in \Sigma^0_{n+2}$ then $\varphi(G)$ can be expressed as $(\exists x)\psi(G, x)$ 
  where $\psi \in \Pi^0_{n+1}$.
  By clause~(i) of Definition~\ref{def:forcing-condition}, there exists a $w \in \omega$
  such that $c \Vdash \psi(G, w)$. By induction hypothesis, $\psi(G, w)$ holds, hence $\varphi(G)$ holds.
  
  Conversely, suppose that $\varphi(G)$ holds. Then there exists a $w \in \omega$ such that $\psi(G, w)$ holds,
  so by induction hypothesis $c \Vdash \psi(G, w)$ for some $c \in \Fcal$,
  so by clause~(i) of Definition~\ref{def:forcing-condition}, $c \Vdash \varphi(G)$.

	\item If $\varphi \in \Pi^0_2$ then $\varphi(G)$ can be expressed as $(\forall x)\psi(G, x)$ where $\psi \in \Sigma^0_1$.
	By clause~(ii) of Definition~\ref{def:forcing-condition}, for every $\tau \in T$, every $w < |\tau|$,
	and every $E \in \Apx_{|\tau|}(c)$, $(E, \tau, T^{[\tau]}) \not \Vdash \neg \psi(G, w)$. 
  Suppose by way of contradiction that $\psi(G, w)$ does not hold for some $w \in \omega$.
  Then by induction hypothesis, there exists a $d \in \Fcal$ such that $d \Vdash \neg \psi(G, w)$.
  Let $e = (E, \tau, S) \in \Fcal$ be such that $e \Vdash \neg \psi(G, w)$, $|\tau| > w$
  and $e$ extends both $c$ and $d$. The condition $e$ exists by Lemma~\ref{lem:forcing-extension-precondition}.
  We can furthermore require that $E \in \Apx_{|\tau|}(c)$,
  so $e \not \Vdash \neg \psi(G, w)$ and $e \Vdash \neg \psi(G, w)$. Contradiction.
  Hence for every $w \in \omega$, $\psi(G, w)$ holds, so $\varphi(G)$ holds.

	\item If $\varphi \in \Pi^0_{n+3}$ then $\varphi(G)$ can be expressed as $\neg \psi(G)$ where $\psi \in \Sigma^0_{n+3}$.
  By clause~(iii) of Definition~\ref{def:forcing-condition}, for every $d \in \Ext(c)$, $d \not \Vdash \psi(G)$.
	By Lemma~\ref{lem:forcing-extension}, $d \not \Vdash \psi(G)$ for every~$d \in \Fcal$, and by 
  a previous case, $\psi(G)$ does not hold, so $\varphi(G)$ holds.
\end{itemize}
\end{proof}

We now prove that the forcing relation enjoys the desired definitional properties,
that is, the complexity of the forcing relation is the same as the complexity
of the formula forced. We start by analysing the complexity of some components
of this notion of forcing.

\begin{lemma}\label{lem:extension-complexity}\ 
\begin{itemize}
	\item[a)] For every precondition $c$, $\Apx(c)$ and $\Ext_1(c)$ are $\Delta^0_1$ uniformly in~$c$.
	\item[b)] For every condition $c$, $\Ext(c)$ is $\Pi^0_2$ uniformly in~$c$.
\end{itemize}
\end{lemma}
\begin{proof}\ 
\begin{itemize}
	\item[a)] Fix a precondition $c = (F, \sigma, T)$.
	A set $E \in \Apx(c)$ iff the following $\Delta^0_1$ predicate holds:
	$$
	(F \subseteq E) \wedge (\forall x \in E \setminus F)[x > \max F \wedge x \in R_\sigma]
	$$
	Moreover, $(E, \tau, S) \in \Ext_1(c)$ iff the $\Delta^0_1$ predicate 
	$E \in \Apx(c) \wedge \tau = \sigma \wedge S = T$ holds.
	As already mentioned, the equality~$S = T$ is translated into ``the indices
	of~$S$ and~$T$ coincide'' which is a $\Sigma^0_0$ statement.

	\item[b)] Fix a condition $c = (F, \sigma, T)$.
	By clause 2) of Lemma~\ref{lem:basic-statements},
	$(E, \tau, S) \in \Ext(c)$ iff the following $\Pi^0_2$ formula holds
	$$
	\begin{array}{ll}
	E \in \Apx(c) \wedge \sigma \preceq \tau \\
	\wedge (\forall \rho \in S)(\forall \xi)[\xi \preceq \rho \imp \xi \in S] & \mbox{ ($S$ is a tree)}\\
	\wedge (\forall n)(\exists \rho \in 2^n)\rho \in S) & \mbox{ ($S$ is infinite) }\\
	\wedge (\forall \rho \in S)(\sigma \prec \rho \vee \rho \preceq \sigma) & \mbox{ ($S$ has stem $\sigma$)}\\
	\wedge (\forall \rho \in S)(\rho \in T) & \mbox{ ($S$ is a subset of $T$) }\\
	\end{array}
	$$
	
\end{itemize}
\end{proof}

\begin{lemma}\label{lem:complexity-forcing}
Fix an arithmetic formula $\varphi(G)$.
\begin{itemize}
	\item[a)] Given a precondition $c$, if $\varphi(G)$ is a $\Sigma^0_1$ ($\Pi^0_1$) formula 
	then so is the predicate $c \Vdash \varphi(G)$.
	\item[b)] Given a condition $c$, if $\varphi(G)$ is a $\Sigma^0_{n+2}$ ($\Pi^0_{n+2}$) formula 
	then so is the predicate $c \Vdash \varphi(G)$.
\end{itemize}
\end{lemma}
\begin{proof}
We prove our lemma by induction over the complexity of the formula $\varphi(G)$.
Fix a (pre)condition $c = (F, \sigma, T)$.
\begin{itemize}
	\item If $\varphi(G) \in \Sigma^0_1$ then it can be expressed as $(\exists x)\psi(G, x)$ where $\psi \in \Sigma^0_0$.
	By clause~(i) of Definition~\ref{def:forcing-precondition}, $c \Vdash \varphi(G)$ if and only if 
	the formula $(\exists w \in \omega)\psi(F, w)$ holds. This is a $\Sigma^0_1$ predicate.
	
	\item If $\varphi(G) \in \Pi^0_1$ then it can be expressed as $(\forall x)\psi(G, x)$ where $\psi \in \Sigma^0_0$.
	By clause~(ii) of Definition~\ref{def:forcing-precondition}, $c \Vdash \varphi(G)$ if and only if 
	the formula $(\forall w \in \omega)(\forall E \in \Apx(c))\psi(E, w)$ holds. 
	By clause~a) of Lemma~\ref{lem:extension-complexity}, this is a $\Pi^0_1$ predicate.

	\item If $\varphi(G) \in \Sigma^0_{n+2}$ then it can be expressed as $(\exists x)\psi(G, x)$ where $\psi \in \Pi^0_{n+1}$.
	By clause~(i) of Definition~\ref{def:forcing-condition}, $c \Vdash \varphi(G)$ if and only if 
	the formula $(\exists w < |\sigma|)c \Vdash \psi(G, w)$ holds. This is a $\Sigma^0_{n+2}$ predicate
	by induction hypothesis.

	\item If $\varphi(G) \in \Pi^0_2$ then it can be expressed as $(\forall x)\psi(G, x)$ where $\psi \in \Sigma^0_1$.
	By clause~(ii) of Definition~\ref{def:forcing-condition}, $c \Vdash \varphi(G)$ if and only if 
	the formula $(\forall \tau \in T)(\forall w < |\tau|)(\forall E \in \Apx_{|\tau|}(c))
	(E, \tau, T^{[\tau]}) \not \Vdash \neg \psi(G, w)$ holds. 
	By induction hypothesis, $(E, \tau, T^{[\tau]}) \not \Vdash \neg \psi(G, w)$ is a $\Sigma^0_1$ predicate,
	hence by clause~a) of Lemma~\ref{lem:extension-complexity}, $c \Vdash \varphi(G)$ is a $\Pi^0_2$ predicate.

	\item If $\varphi(G) \in \Pi^0_{n+3}$ then it can be expressed as $\neg \psi(G)$ where $\psi \in \Sigma^0_{n+3}$. 
	By clause~(iii) of Definition~\ref{def:forcing-condition}, $c \Vdash \varphi(G)$ if and only if 
	the formula $(\forall d)(d \not \in \Ext(c) \vee d \not \Vdash \psi(G))$ holds.
	By induction hypothesis, $d \not \Vdash \psi(G)$ is a $\Pi^0_{n+3}$ predicate.
	Hence by clause~b) of Lemma~\ref{lem:extension-complexity},
	$c \Vdash \varphi(G)$ is a $\Pi^0_{n+3}$ predicate.
\end{itemize}
\end{proof}

\subsection{Preserving the arithmetic hierarchy}

The following lemma asserts that every sufficiently generic real
for this notion of forcing preserves the arithmetic hierarchy.
The argument deeply relies on the fact that this notion of forcing
admits a forcing relation with good definitional properties.

\begin{lemma}\label{lem:diagonalization}
If $A \not \in \Sigma^0_{n+1}$ and $\varphi(G, x)$ is $\Sigma^0_{n+1}$, 
then the set of $c \in \Pb$ satisfying the following property is dense:
$$
[(\exists w \in A)c \Vdash \neg \varphi(G, w)] \vee [(\exists w \not \in A)c \Vdash \varphi(G, w)]
$$
\end{lemma}
\begin{proof}
Fix a condition $c = (F, \sigma, T)$.
\begin{itemize}
	\item In case $n = 0$, $\varphi(G, w)$ can be expressed as $(\exists x)\psi(G, w, x)$ where $\psi \in \Sigma^0_0$.
	Let $U = \{ w \in \omega : (\exists E \in \Apx(c))(\exists u)\psi(E, w, u) \}$.
	By clause~a) of Lemma~\ref{lem:extension-complexity}, $U \in \Sigma^0_1$, thus $U \neq A$.
	Fix $w \in U \Delta A$. If $w \in U \setminus A$ then by definition of~$U$,
	there exist an $E \in \Apx(c)$ and a $u \in \omega$ such that $\psi(E, w, u)$ holds.
	By definition of $\Apx(c)$ and clause~1) of Lemma~\ref{lem:basic-statements}, $d = (E, \sigma, T)$
	is a condition extending $c$. By clause~(i) of Definition~\ref{def:forcing-precondition},
	$d \Vdash \varphi(G, w)$.
	If $w \in A \setminus U$, then for every $E \in \Apx(c)$ and every $u \in \omega$, 
	$\psi(E, w, u)$ does not hold, so by clause~(ii) of Definition~\ref{def:forcing-precondition},
	$c \Vdash (\forall x)\neg \psi(G, w, x)$, hence $c \Vdash \neg \varphi(G, w)$.

	\item In case $n = 1$, $\varphi(G, w)$ can be expressed as $(\exists x)\psi(G, w, x)$ where $\psi \in \Pi^0_1$.
	Let $U = \{ w \in \omega : (\exists s)(\forall \tau \in 2^s \cap T)
	(\exists u < s)(\exists E \in \Apx_s(c))(E, \tau, T^{[\tau]}) \Vdash \psi(G, w, u) \}$.
	By Lemma~\ref{lem:complexity-forcing} and clause~a) of Lemma~\ref{lem:extension-complexity},
	$U \in \Sigma^0_2$, thus $U \neq A$. Fix $w \in U \Delta A$.
	If $w \in U \setminus A$ then by definition of~$U$, there exist an $s \in \omega$,
	a $\tau \in 2^s \cap T$, a $u < s$ and an $E \in \Apx_s(c)$ such that $T^{[\tau]}$ is infinite
	and $(E, \tau, T^{[\tau]}) \Vdash \psi(G, w, u)$. Thus $d = (E, \tau, T^{[\tau]})$ is a condition
	extending $c$ and by clause~(i) of Definition~\ref{def:forcing-condition}, $d \Vdash \varphi(G, w)$.
	If $w \in A \setminus U$, then let 
	$S = \{ \tau \in T : (\forall u < |\tau|)(\forall E \in \Apx_{|\tau|}(c) (E, \tau, T^{[\tau]}) \not \Vdash \psi(G, w, u) \}$.
	As proven in Lemma~\ref{lem:forcing-dense}, $S$ is a $\emptyset'$-p.r. subtree of~$T$
	and by $w \not \in U$, $S$ is infinite. Thus $d = (F, \sigma, S)$ is a condition extending $c$.
	By clause~3) of Lemma~\ref{lem:basic-statements}, $\Apx(d) \subseteq \Apx(c)$, 
	so for every $\tau \in S$, every $u < |\tau|$, and every $E \in \Apx_{|\tau|}(d)$,
	$(E, \tau, T^{[\tau]}) \not \Vdash \psi(G, w, u)$. By Lemma~\ref{lem:tree-no-effect-first-level},
	$(E, \tau, S^{[\tau]}) \not \Vdash \psi(G, w, u)$, so by clause~(ii) of Definition~\ref{def:forcing-condition},
	$d \Vdash (\forall x)\neg \psi(G, w, u)$ hence $d \Vdash \neg \varphi(G, w)$.

	\item In case $n > 1$, let $U = \{ w \in \omega : (\exists d \in \Ext(c)) d \Vdash \varphi(G, w) \}$.
	By clause~b) of Lemma~\ref{lem:extension-complexity} and Lemma~\ref{lem:complexity-forcing},
	$U \in \Sigma^0_n$, thus $U \neq A$.
	Fix $w \in U \Delta A$. If $w \in U \setminus A$ then by definition of~$U$,
	there exists a condition $d$ extending $c$ such that $d \Vdash \varphi(G, w)$.
	If $w \in A \setminus U$, then for every $d \in \Ext(c) d \not \Vdash \varphi(G, w)$
	so by clause~(iii) of Definition~\ref{def:forcing-condition}, $c \Vdash \neg \varphi(G, w)$. 
\end{itemize}
\end{proof}

We are now ready to prove Theorem~\ref{thm:coh-preservation-arithmetic-hierarchy}.

\begin{proof}[Proof of Theorem~\ref{thm:coh-preservation-arithmetic-hierarchy}]
Let~$C$ be a set and~$R_0, R_1, \dots$ be a uniformly $C$-computable sequence of sets.
Let~$T_0$ be a $C'$-primitive recursive tree such that~$[T_0] \subseteq \Ccal(\vec{R})$.
Let~$\Fcal$ be a sufficiently generic filter containing~$c_0 = (\emptyset, \epsilon, T_0)$.
and let $G$ be the corresponding generic real. By genericity, the set~$G$ is
an infinite $\vec{R}$-cohesive set.
By Lemma~\ref{lem:diagonalization} and Lemma~\ref{lem:complexity-forcing},
$G$ preserves non-$\Sigma^0_{n+1}$ definitions relative to~$C$ for every~$n \in \omega$.
Therefore, by Proposition 2.2 of~\cite{Wang2014Definability}, $G$
preserves the arithmetic hierarchy relative to~$C$.
\end{proof}

\section{The Erd\H{o}s Moser theorem preserves the arithmetic hierarchy}

We now extend the previous result to the Erd\H{o}s-Moser theorem.
The Erd\H{o}s-Moser theorem is a statement coming from graph theory.
It can be used with the ascending descending principle~($\ads$) to provide an alternative proof of
Ramsey's theorem for pairs ($\rt^2_2$). Indeed, every coloring~$f : [\omega]^2 \to 2$
can be seen as a tournament~$R$ such that~$R(x,y)$ holds if~$x < y$ and~$f(x,y) = 1$, or~$x > y$ and~$f(y, x) = 0$.
Every infinite transitive subtournament induces a linear order whose infinite ascending or descending
sequences are homogeneous for~$f$.

\begin{definition}[Erd\H{o}s-Moser theorem]
A tournament $T$ on a domain $D \subseteq \omega$ is an irreflexive binary relation on~$D$ such that for all $x,y \in D$ with $x \not= y$, exactly one of $T(x,y)$ or $T(y,x)$ holds. A tournament $T$ is \emph{transitive} if the corresponding relation~$T$ is transitive in the usual sense. A tournament $T$ is \emph{stable} if $(\forall x \in D)(\exists n)[(\forall s > n) T(x,s) \vee (\forall s > n) T(s, x)]$.
$\emo$ is the statement ``Every infinite tournament $T$ has an infinite transitive subtournament.''
$\semo$ is the restriction of $\emo$ to stable tournaments.
\end{definition}

Bovykin and Weiermann proved in \cite{Bovykin2005strength} that $\emo + \ads$
is equivalent to $\rt^2_2$ over $\rca$, and $\semo + \sads$ is equivalent to~$\srt^2_2$ over~$\rca$.
Lerman et al.~\cite{Lerman2013Separating} proceeded to a combinatorial and effective analysis of the Erd\H{o}s-Moser
theorem, and proved in particular that there is an $\omega$-model of $\emo$ which is not a model of~$\srt^2_2$.
The author simplified their proof in~\cite{Patey2015Iterative} and showed in~\cite{Patey2015Somewhere} 
that $\rca \vdash \emo \imp [\sts^2 \vee \coh]$,
where $\sts^2$ stands for the stable thin set theorem for pairs. 
In particular, since Wang~\cite{Wang2014Definability} proved that $\sts^2$
does not admit preservation of the arithmetic hierarchy, 
Theorem~\ref{thm:coh-preservation-arithmetic-hierarchy} follows from Theorem~\ref{thm:em-preserves-arithmetic}.
From a definitional point of view, Wang proved in~\cite{Wang2014Definability}
that $\emo$ admits preservation of $\Delta^0_2$ definitions
and preservation of definitions beyond the $\Delta^0_2$ level. He conjectured that
$\emo$ admits preservation of the arithmetic hierarchy. The balance of this section proves his conjecture.

\begin{theorem}\label{thm:em-preserves-arithmetic}
$\emo$ admits preservation of the arithmetic hierarchy.
\end{theorem}

Again, the core of the proof consists of finding a good forcing notion
whose generics will preserve the arithmetic hierarchy.
For simplicity, we will restrict ourselves to stable tournaments
even though it is clear that the forcing notion can be adapted to arbitrary tournaments.
The proof of Theorem~\ref{thm:em-preserves-arithmetic} will be obtained by composing the proof that cohesiveness
and the stable Erd\H{o}s-Moser theorem admit preservation of the arithmetic hierarchy.

The following notion of \emph{minimal interval}
plays a fundamental role in the analysis of $\emo$.
See~\cite{Lerman2013Separating} for a background analysis of $\emo$.

\begin{definition}[Minimal interval]
Let $T$ be an infinite tournament and $a, b \in T$
be such that $T(a,b)$ holds. The \emph{interval} $(a,b)$ is the
set of all $x \in T$ such that $T(a,x)$ and $T(x,b)$ hold.
Let $F \subseteq T$ be a finite transitive subtournament of $T$.
For $a, b \in F$ such that $T(a,b)$ holds, we say that $(a,b)$
is a \emph{minimal interval of $F$} if there is no $c \in F \cap (a,b)$,
i.e., no $c \in F$ such that $T(a,c)$ and $T(c,b)$ both hold.
\end{definition}

We must introduce an preliminary variant of Mathias forcing which is more suited to 
the Erd\H{o}s-Moser theorem.

\subsection{Erd\H{o}s Moser forcing}

The following notion of Erd\H{o}s-Moser forcing was implicitly first
used by Lerman, Solomon and Towsner~\cite{Lerman2013Separating} to separate the Erd\H{o}s-Moser
theorem from stable Ramsey's theorem for pairs. The author formalized this notion 
of forcing in~\cite{Patey2015Degrees} to construct a low${}_2$ degree bounding the Erd\H{o}s-Moser theorem.

\begin{definition}
An \emph{Erd\H{o}s Moser condition} (EM condition) for an infinite tournament $R$
is a Mathias condition $(F, X)$ where
\begin{itemize}
	\item[(a)] $F \cup \{x\}$ is $R$-transitive for each $x \in X$
	\item[(b)] $X$ is included in a minimal $R$-interval of $F$.
\end{itemize}
\end{definition}

The Erd\H{o}s-Moser extension is the usual Mathias extension.
EM conditions have good properties for tournaments as shown by the following lemmas.
Given a tournament $R$ and two sets $E$ and $F$,
we denote by $E \to_R F$ the formula $(\forall x \in E)(\forall y \in F) R(x,y) \mbox{ holds}$.

\begin{lemma}[Patey~\cite{Patey2015Degrees}]\label{lem:emo-cond-beats}
Fix an EM condition $(F, X)$ for a tournament $R$.
For every $x \in F$, $\{x\} \to_R X$ or $X \to_R \{x\}$.
\end{lemma}

\begin{lemma}[Patey~\cite{Patey2015Degrees}]\label{lem:emo-cond-valid}
Fix an EM condition $c = (F, X)$ for a tournament $R$, 
an infinite subset $Y \subseteq X$ and a finite $R$-transitive set $F_1 \subset X$ such that
$F_1 < Y$ and $[F_1 \to_R Y \vee Y \to_R F_1]$.
Then $d = (F \cup F_1, Y)$ is a valid extension of $c$.
\end{lemma}

\subsection{Partition trees}

Given a string $\sigma \in k^{<\omega}$,
we denote by $\set_\nu(\sigma)$ the set $\{ x < |\sigma| : \sigma(x) = \nu \}$
where $\nu < k$. The notion can be extended to sequences $P \in k^{\omega}$
where $\set_\nu(P) = \{ x \in \omega : P(x) = \nu \}$.

\begin{definition}[Partition tree]
A \emph{$k$-partition tree of $[t, \infty)$} for some $k, t \in \omega$
is a tuple $(k, t, T)$ such that $T$ is a subtree of $k^{<\omega}$.
A \emph{partition tree} is a $k$-partition tree of $[t, \infty)$ for some $k, t \in \omega$.
\end{definition}

To simplify our notation, we may use the same letter $T$ to denote both a partition tree $(k, t, T)$
and the actual tree $T \subseteq k^{<\omega}$. We then write $\dom(T)$ for $[t, \infty)$
and $\parts(T)$ for~$k$. Given a p.r.\ partition tree~$T$,
we write~$\#T$ for its Turing index, and may refer to it as its \emph{code}.

\begin{definition}[Refinement]
Given a function $f : \ell \to k$, a string $\sigma \in \ell^{<\omega}$ \emph{$f$-refines} a string $\tau \in k^{<\omega}$
if $|\sigma| = |\tau|$ and for every $\nu < \ell$, $\set_\nu(\sigma) \subseteq \set_{f(\nu)}(\tau)$.
A p.r.\ $\ell$-partition tree $S$ of $[u,\infty)$ \emph{$f$-refines} a p.r.\ $k$-partition tree $T$ 
of $[t, \infty)$ (written $S \leq_f T$)
if $\#S \geq \#T$, $\ell \geq k$, $u \geq t$ and for every $\sigma \in S$, $\sigma$ $f$-refines some $\tau \in T$. 
\end{definition}

The partition trees will act as the reservoirs in the forcing conditions defined in the next section.
Consequently, refining a partition tree restricts the reservoir, as desired when extending a condition.
The collection of partition trees is equipped
with a partial order $\leq$ such that $(\ell, u, S) \leq (k, t, T)$ if there exists a function $f : \ell \to k$ such that $S \leq_f T$.
Given a $k$-partition tree of $[t, \infty)$ $T$,
we say that part $\nu$ of $T$ is \emph{acceptable} if there exists a path $P$ through $T$
such that $\set_\nu(P)$ is infinite.
Moreover, we say that part $\nu$ of $T$ is \emph{empty} if 
$(\forall \sigma \in T)[dom(T) \cap \set_\nu(\sigma) = \emptyset]$.
Note that each partition tree has at least one acceptable part
since for every path~$P$ through~$T$, $\set_\nu(P)$ is infinite for some~$\nu < k$.
It can also be the case that part~$\nu$ of~$T$ is non-empty,
while for every path $P$ through~$T$, $\set_\nu(P) \cap \dom(T) = \emptyset$.
However, in this case, we can choose the infinite computable subtree~$S = \{\sigma \in T : \set_\nu(\sigma) \cap \dom(T) = \emptyset\}$
of~$T$ which has the same collection of infinite paths and such that part~$\nu$ of~$S$ is empty.

Given a $k$-partition tree $T$, a finite set $F \subseteq \omega$ and a part $\nu < k$,
define
$$
T^{[\nu, F]} = \{ \sigma \in T :  F \subseteq \set_\nu(\sigma) \vee |\sigma| < \max F \}
$$
The set $T^{[\nu, F]}$ is a (possibly finite) subtree of~$T$ which id-refines $T$
and such that~$F \subseteq \set_\nu(P)$ for every path~$P$ through $T^{[\nu, F]}$.

We denote by $\TPb$ the set of all ordered pairs $(\nu, T)$
such that $T$ is an infinite, primitive recursive $k$-partition tree of $[t, \infty)$
for some $t, k \in \omega$ and $\nu < k$. The set $\TPb$ is equipped with a partial ordering $\leq$
such that $(\mu, S) \leq (\nu, T)$ if $S$ $f$-refines $T$ and $f(\mu) = \nu$ for some~$f$.
In this case we say that \emph{part $\mu$ of $S$ refines part $\nu$ of~$T$}.
Note that the domain of $\TPb$ and the relation $\leq$ are co-c.e.
We denote by $\TPb[T]$ the set of all $(\nu, S) \leq (\mu, T)$ for some $(\mu, T) \in \TPb$.

\begin{definition}[Promise for a partition tree]\label{def:em-promise}
Fix a p.r. $k$-partition tree of $[t,\infty)$ $T$.
A class $\Ccal \subseteq \TPb[T]$ is a \emph{promise for $T$} if
\begin{itemize}
	\item[a)] $\Ccal$ is upward-closed under the $\leq$ relation restricted to $\TPb[T]$
	\item[b)] for every infinite p.r. partition tree $S \leq T$,
	$(\mu, S) \in \Ccal$ for some non-empty part $\mu$ of~$S$.
\end{itemize}
\end{definition}

A promise for $T$ can be seen as a two-dimensional tree
with at first level the acyclic digraph of refinement of partition trees.
Given an infinite path in this digraph, the parts of the members of this path
form an infinite, finitely branching tree.
The following lemma holds for every $\emptyset'$-computable promise.
However, we shall work later with conditions containing $\emptyset'$-primitive recursive
promises in order to lower the definitional complexity of being a valid condition
and to be able to prove Lemma~\ref{lem:em-complexity-forcing}. 
We therefore focus on $\emptyset'$-p.r.\ promises.

\begin{lemma}\label{lem:em-refinement-complexity}
Let $T$ and $S$ be p.r.\  partition trees such that $S \leq_f T$ for some function $f : \parts(S) \to \parts(T)$
and let $\Ccal$ be a $\emptyset'$-p.r.\ promise for $T$.
\begin{itemize}
	\item[a)] The predicate ``$T$ is an infinite $k$-partition tree of $[t, \infty)$'' is $\Pi^0_1$ uniformly in $T$, $k$ and~$t$.
	\item[b)] The relations ``$S$ $f$-refines $T$''
	and ``part $\nu$ of $S$ $f$-refines part $\mu$ of $T$''
	are $\Pi^0_1$ uniformly in $S$, $T$ and $f$.
	\item[c)] The predicate ``$\Ccal$ is a promise for $T$'' is $\Pi^0_2$ uniformly in an index for~$\Ccal$ and~$T$.
\end{itemize}
\end{lemma}
\begin{proof}\ 
\begin{itemize}
	\item[a)] $T$ is an infinite $k$-partition tree of $[t, \infty)$ if and only if the $\Pi^0_1$ formula
$[(\forall \sigma \in T)(\forall \tau \preceq \sigma) \tau \in T \cap k^{<\infty}]
 \wedge [(\forall n)(\exists \tau \in k^n) \tau \in T]$  holds.

	\item[b)] Suppose that $T$ is a $k$-partition tree of $[t,\infty)$ and $S$ is an $\ell$-partition tree of $[u,\infty)$.
$S$ $f$-refines $T$ if and only if the $\Pi^0_1$ formula holds:
$$
u \geq t \wedge \ell \geq k \wedge [(\forall \sigma \in S)(\exists \tau \in k^{|\sigma|} \cap T)
(\forall \nu < u)set_\nu(\sigma) \subseteq \set_{f(\nu)}(\tau)]
$$
Part $\nu$ of $S$ $f$-refines part $\mu$ of $T$ if and only if $\mu = f(\nu)$ and $S$ $f$-refines $T$.

	\item[c)] Given $k, t \in \omega$, let $PartTree(k, t)$ denote 
	the $\Pi^0_1$ set of all the infinite p.r.\ $k$-partition trees of $[t, \infty)$.
	Given a $k$-partition tree $S$ and a part $\nu$ of $S$,
	let $Empty(S, \nu)$ denote the $\Pi^0_1$ formula ``part $\nu$ of $S$ is empty'',
	that is the formula $(\forall \sigma \in S) \set_\nu(\sigma) \cap \dom(S) = \emptyset$.

	$\Ccal$ is a promise for $T$ if and only if the following $\Pi^0_2$ formula holds:
	$$
	\begin{array}{l}
	(\forall \ell, u)(\forall S \in PartTree(\ell,u))[S \leq T \imp (\exists \nu < \ell) \neg Empty(S, \nu) \wedge (\nu, S) \in \Ccal)] \\
	\wedge (\forall \ell', u')(\forall V \in PartTree(\ell', u'))(\forall g : \ell \to \ell')[ S \leq_g V \leq T \imp \\
	(\forall \nu < \ell)((\nu, S) \in \Ccal \imp (g(\nu), V) \in \Ccal)]
	\end{array}
	$$
\end{itemize}
\end{proof}

Given a promise $\Ccal$ for $T$ and some infinite p.r. partition tree $S$
refining $T$, we denote by $\Ccal[S]$ the set of all $(\nu, S') \in \Ccal$
below some $(\mu, S) \in \Ccal$, that is, $\Ccal[S] = \Ccal \cap \TPb[S]$.
Note that by clause b) of Lemma~\ref{lem:em-refinement-complexity},
if $\Ccal$ is a $\emptyset'$-p.r. promise for $T$ then $\Ccal[S]$
is a $\emptyset'$-p.r. promise for~$S$.

Establishing a distinction between the acceptable parts and the non-acceptable ones requires
a lot of definitional power. However, we prove that we can always find an
extension where the distinction is $\Delta^0_2$.
We say that an infinite p.r. partition tree $T$
\emph{witnesses its acceptable parts} if its parts are either acceptable or empty.

\begin{lemma}\label{lem:em-promise-extension-witnessing-acceptable}
For every infinite p.r.\ $k$-partition tree $T$ of $[t, \infty)$,
there exists an infinite p.r.\ $k$-partition tree $S$ of $[u, \infty)$
refining $T$ with the identity function and such that $S$ witnesses its acceptable parts.
\end{lemma}
\begin{proof}
Given a partition tree $T$, we let $I(T)$ be the set of its empty parts.
Let~$T$ be a fixed infinite p.r.\ $k$-partition tree of $[t, \infty)$.
It suffices to prove that if $\nu$ is a non-empty and non-acceptable part of~$T$,
then there exists an infinite p.r.\ $k$-partition tree $S$
refining $T$ with the identity function, such that $\nu \in I(S) \setminus I(T)$.
As $I(T) \subseteq I(S)$ and $|I(S)| \leq k$, it suffices to iterate the process
at most $k$ times to obtain a refinement witnessing its acceptable parts.

So fix a non-empty and non-acceptable part $\nu$ of~$T$.
By definition of being non-acceptable, there exists a path $P$ through $T$
and an integer $u > \max(t, \set_\nu(P))$.
Let $S = \{ \sigma \in T : \set_\nu(\sigma) \cap [u, \infty) = \emptyset \}$.
The set $S$ is a p.r. $k$-partition tree of $[u, \infty)$ refining
$T$ with the identity function and such that part $\nu$ of $S$ is empty.
Moreover, $S$ is infinite since~$P \in [S]$. 
\end{proof}

The following lemma strengthens clause b) of Definition~\ref{def:em-promise}.

\begin{lemma}\label{lem:promise-keeps-acceptable-parts}
Let $T$ be a p.r. partition tree and $\Ccal$ be a promise for $T$.
For every infinite p.r. partition tree $S \leq T$, 
$(\mu, S) \in \Ccal$ for some acceptable part $\mu$ of~$S$.
\end{lemma}
\begin{proof}
Fix an infinite p.r. $\ell$-partition tree $S \leq T$.
By Lemma~\ref{lem:em-promise-extension-witnessing-acceptable},
there exists an infinite p.r. $\ell$-partition tree $S' \leq_{id} S$
witnessing its acceptable parts. As $\Ccal$ is a promise for $T$
and $S' \leq T$, there exists a non-empty (hence acceptable) part $\nu$ of $S'$ such that $(\nu, S') \in \Ccal$.
As $\Ccal$ is upward-closed, $(\nu, S) \in \Ccal$.
\end{proof}

\subsection{Forcing conditions}

We now describe the forcing notion for the Erd\H{o}s-Moser theorem.
Recall that an EM condition for an infinite tournament~$R$
is a Mathias condition $(F, X)$ where
$F \cup \{x\}$ is $R$-transitive for each $x \in X$
and $X$ is included in a minimal $R$-interval of $F$.

\begin{definition}
We denote by~$\Pb$ the forcing notion whose conditions are tuples $(\vec{F}, T, \Ccal)$ where
\begin{itemize}
	\item[(a)] $T$ is an infinite p.r.\ partition tree
	\item[(b)] $\Ccal$ is a $\emptyset'$-p.r. promise for $T$
	\item[(c)] $(F_\nu, \dom(T))$ is an EM condition for $R$ and each $\nu < \parts(T)$
\end{itemize}
A condition $d = (\vec{E}, S, \Dcal)$ \emph{extends} $c = (\vec{F}, T, \Ccal)$
(written $d \leq c$) if there exists a function $f : \ell \to k$ such that
$\Dcal \subseteq \Ccal$ and the followings hold:
\begin{itemize}
	\item[(i)] $(E_\nu, \dom(S))$ EM extends $(F_{f(\nu)}, \dom(T))$ for each $\nu < \parts(S)$ 
	\item[(ii)] $S$ $f$-refines $\bigcap_{\nu < \parts(S)} T^{[f(\nu), E_\nu]}$
\end{itemize}
\end{definition}

We may think of a condition $c = (\vec{F}, T, \Ccal)$ as a collection of EM conditions
$(F_\nu, H_\nu)$ for~$R$, where $H_\nu = \dom(T) \cap \set_\nu(P)$ for some path $P$ through $T$.
$H_\nu$ must be infinite for at least one of the parts $\nu < \parts(T)$.
At a higher level, $\Dcal$ restricts the possible subtrees $S$
and parts $\mu$ refining some part of $T$ in the condition~$c$.
Given a condition $c = (\vec{F}, T, \Ccal)$, we write $\parts(c)$ for $\parts(T)$.

\begin{lemma}\label{lem:em-forcing-infinite}
For every condition $c = (\vec{F}, T, \Ccal)$ and every $n \in \omega$, there exists an extension $d = (\vec{E}, S, \Dcal)$
such that $|E_\nu| \geq n$ on each acceptable part $\nu$ of~$S$.
\end{lemma}
\begin{proof}
It suffices to prove that for every condition $c = (\vec{F}, T, \Ccal)$
and every acceptable part $\nu$ of $T$, 
there exists an extension $d = (\vec{E}, S, \Dcal)$ such that $S \leq_{id} T$ 
and $|E_\nu| \geq n$. Iterating the process at most $\parts(T)$ times completes the proof.
Fix an acceptable part $\nu$ of $T$ and a path $P$ trough $T$
such that $\set_\nu(P)$ is infinite. Let $F'$ be an $R$-transitive subset of $\set_\nu(P) \cap \dom(T)$
of size $n$. Such a set exists by the classical Erd\H{o}s-Moser theorem. Let $\vec{E}$ be defined by $E_\mu = F_\mu$
if $\mu \neq \nu$ and $E_\nu = F_\nu \cup F'$ otherwise. As the tournament~$R$ is stable,
there exists some~$u \geq t$ such that $(E_\nu, [u, \infty))$ is an EM condition
and therefore EM extends $(F_\nu, \dom(T))$.
Let $S$ be the p.r.\ partition tree $T^{[\nu, E_\nu]}$ of $[u, \infty)$.
The condition $(\vec{E}, S, \Ccal[S])$ is the desired extension.
\end{proof}

Given a condition~$c \in \Pb$, we denote by $\Ext(c)$ the set of all its extensions. 

\subsection{The forcing relation}

The forcing relation at the first level, namely, for $\Sigma^0_1$ and $\Pi^0_1$ formulas,
is parameterized by some part of the tree of the considered condition. Thanks to the forcing relation
we will define, we can build an infinite decreasing sequence of conditions which decide $\Sigma^0_1$ and~$\Pi^0_1$ formulas
effectively in~$\emptyset'$. This sequence yields a $\emptyset'$-computably bounded $\emptyset'$-computable tree of (possibly empty)
parts. Therefore, any PA degree relative to~$\emptyset'$ is sufficient to control the first jump of an infinite
transitive subtournament of a stable infinite computable tournament.

We cannot do better since Kreuzer proved in~\cite{Kreuzer2012Primitive} the existence
of an infinite, stable, computable tournament with no low infinite transitive subtournament.
If we ignore the promise part of a condition,
the careful reader will recognize the construction of Cholak, Jockusch and Slaman~\cite{Cholak2001strength} of a low${}_2$ infinite
subset of a $\Delta^0_2$ set or its complement by the first jump control. 
The difference, which at first seems only notational,
is in fact one of the key features of this notion of forcing. Indeed, forcing iterated jumps requires
a definitionally weak description of the set of extensions of a condition,
and it requires much less computational power to describe a primitive recursive
tree than an infinite reservoir of a Mathias condition.

\begin{definition}\label{def:em-forcing-precondition}
Fix a condition $c = (\vec{F}, T, \Ccal)$, a $\Sigma^0_0$ formula $\varphi(G, x)$
and a part $\nu < \parts(T)$.
\begin{itemize}
	\item[1.] $c \Vdash_\nu (\exists x)\varphi(G, x)$ iff there exists a $w \in \omega$ such that $\varphi(F_\nu, w)$ holds.
	\item[2.] $c \Vdash_\nu (\forall x)\varphi(G, x)$ iff 
	for every $\sigma \in T$, every $w < |\sigma|$ and every $R$-transitive set $F' \subseteq \dom(T) \cap \set_\nu(\sigma)$,
	$\varphi(F_\nu \cup F', w)$ holds.
\end{itemize}
\end{definition}

We start by proving some basic properties of the forcing relation over~$\Sigma^0_1$
and $\Pi^0_1$ formulas. As one may expect, the forcing relation at first level
is closed under the refinement relation.

\begin{lemma}\label{lem:em-forcing-extension-level1}
Fix a condition $c = (\vec{F}, T, \Ccal)$ and a $\Sigma^0_1$ ($\Pi^0_1$) formula $\varphi(G)$.
If $c \Vdash_\nu \varphi(G)$ for some $\nu < \parts(T)$, then for every $d = (\vec{E}, S, \Dcal) \leq c$
and every part $\mu$ of $S$ refining part $\nu$ of $T$, $d \Vdash_{\mu} \varphi(G)$.
\end{lemma}
\begin{proof}We have two cases.
\begin{itemize}
	\item If $\varphi \in \Sigma^0_1$ then it can be expressed as $(\exists x)\psi(G, x)$
	where $\psi \in \Sigma^0_0$. By clause~1 of Definition~\ref{def:em-forcing-precondition},
	there exists a $w \in \omega$ such that
	$\psi(F_\nu, w)$ holds. By property (i) of the definition of an extension, $E_\mu \supseteq F_\nu$
	and $(E_\mu \setminus F_\nu) \subset \dom(T)$, therefore $\psi(E_\mu, w)$ holds by continuity,
	so by clause~1 of Definition~\ref{def:em-forcing-precondition}, $d \Vdash_\mu (\exists x)\psi(G, x)$.

	\item If $\varphi \in \Pi^0_1$ then it can be expressed as $(\forall x)\psi(G, x)$
	where $\psi \in \Sigma^0_0$. Fix a $\tau \in S$, a $w < |\tau|$ and an $R$-transitive set
	$F' \subseteq \dom(S) \cap \set_\mu(\tau)$.
	It suffices to prove that $\varphi(E_\mu \cup F')$ holds to conclude
	that $d \Vdash_\mu (\forall x)\psi(G, x)$ by clause~2 of Definition~\ref{def:em-forcing-precondition}.
 	By property~(ii) of the definition of an extension,
	there exists a $\sigma \in T^{[\nu, E_\mu]}$
	such that $|\sigma| = |\tau|$ and $\set_\mu(\tau) \subseteq \set_\nu(\sigma)$. 
	As $\dom(S) \subseteq \dom(T)$, $F' \subseteq \dom(T) \cap \set_\nu(\sigma)$.
	As $\sigma \in T^{[\nu, E_\mu]}$, $E_\mu \subseteq \set_\nu(\sigma)$
	and by property~(i) of the definition of an extension, $E_\mu \subseteq \dom(T)$.
	So $E_\mu \cup F' \subseteq \dom(T) \cap \set_\nu(\sigma)$.
	As $w < |\tau| = |\sigma|$ and~$E_\mu \cup F'$ is an $R$-transitive subset of~$\dom(T) \cap \set_\nu(\sigma)$,
	then by clause~2 of Definition~\ref{def:em-forcing-precondition} applied to $c \Vdash_\nu (\forall x)\psi(G, x)$, 
	$\varphi(F_\nu \cup (E_\mu \setminus F_\nu) \cup F', w)$ holds, hence $\varphi(E_\mu \cup F')$ holds.
\end{itemize}
\end{proof}

Before defining the forcing relation at higher levels, we prove a density lemma
for $\Sigma^0_1$ and $\Pi^0_1$ formulas. It enables us
in particular to reprove that every degree PA relative to $\emptyset'$
computes the jump of an infinite $R$-transitive set.

\begin{lemma}\label{lem:em-forcing-dense-level1}
For every $\Sigma^0_1$ ($\Pi^0_1$) formula $\varphi$, the following set is dense
$$
\{ c = (\vec{F}, T, \Ccal) \in \Pb : (\forall \nu < \parts(T))[ c \Vdash_\nu \varphi(G) \vee c \Vdash_\nu \neg \varphi(G)] \}
$$
\end{lemma}
\begin{proof}
It suffices to prove the statement for the case where $\varphi$ is a $\Sigma^0_1$ formula,
as the case where $\varphi$ is a $\Pi^0_1$ formula is symmetric. Fix a condition $c = (\vec{F}, T, \Ccal)$
and let $I(c)$ be the set of the parts $\nu < \parts(T)$ such that $c \not \Vdash_\nu \varphi(G)$
and $c \not \Vdash_\nu \neg \varphi(G)$. If $I(c) = \emptyset$ then we are done, so suppose $I(c) \neq \emptyset$
and fix some $\nu \in I(c)$. We will construct an extension $d$ of~$c$ such that $I(d) \subseteq I(c) \setminus \{\nu\}$.
Iterating the operation completes the proof.

The formula $\varphi$ is of the form $(\exists x)\psi(G, x)$ where $\psi \in \Sigma^0_0$.
Define $f : k+1 \to k$ as $f(\mu) = \mu$ if $\mu < k$ and $f(k) = \nu$ otherwise.
Let $S$ be the set of all $\sigma \in (k+1)^{<\omega}$ which $f$-refine some $\tau \in T \cap k^{|\sigma|}$
and such that for every $w < |\sigma|$, every part $\mu \in \{\nu, k\}$
and every finite $R$-transitive set $F' \subseteq \dom(T) \cap \set_\mu(\sigma)$, 
$\varphi(F_\nu \cup F', w)$ does not hold.

Note that $S$ is a p.r. partition tree of $[t, \infty)$ refining $T$ with witness function~$f$. 
Suppose that $S$ is infinite. 
Let $\vec{E}$ be defined by $E_\mu = F_\mu$ if $\mu < k$ and $E_k = F_\nu$
and consider the extension $d = (\vec{E}, S, \Ccal[S])$.
We claim that~$\nu, k \not \in I(d)$. Fix a part $\mu \in \{\nu, k\}$ of $S$. By definition of $S$,
for every $\sigma \in S$, every $w < |\sigma|$
and every $R$-transitive set $F' \subseteq \dom(S) \cap \set_\mu(\sigma)$,
$\varphi(E_\mu \cup F', w)$ does not hold.
Therefore, by clause~2 of Definition~\ref{def:em-forcing-precondition},
$d \Vdash_\mu (\forall x)\neg \psi(G, x)$, hence $d \Vdash_\mu \neg \varphi(G)$.
Note that $I(d) \subseteq I(c) \setminus \{\nu\}$.

Suppose now that $S$ is finite. Fix a threshold $\ell \in \omega$ such that $(\forall \sigma \in S)|\sigma| < \ell$
and a $\tau \in T \cap k^\ell$ such that $T^{[\tau]}$ is infinite.
Consider the 2-partition $E_0 \sqcup E_1$ of $\set_\nu(\tau) \cap \dom(T)$ defined by
$E_0 = \{ i \geq t : \tau(i) = \nu \wedge (\exists n)(\forall s > n) R(i, s) \mbox{ holds}\}$
and $E_1 = \{ i \geq t : \tau(i) = \nu \wedge (\exists n)(\forall s > n) R(s, i) \mbox{ holds}\}$.
This is a 2-partition since the tournament~$R$ is stable.
As there exists no $\sigma \in S$ which $f$-refines $\tau$,
there exists a $w < \ell$ and an $R$-transitive set $F' \subseteq E_0$ or $F' \subseteq E_1$
such that $\varphi(F_\nu \cup F', w)$ holds. By choice of the partition,
there exists a $t' > t$ such that $F' \to_R [t', \infty)$ or $[t', \infty) \to_R F'$.
By Lemma~\ref{lem:emo-cond-valid}, $(F_\nu \cup F', [t', \infty))$ is a valid EM extension for $(F_\nu, [t, \infty))$.
As $T^{[\tau]}$ is infinite, $T^{[\nu, F']}$ is also infinite.
Let $\vec{E}$ be defined by $E_\mu = F_\mu$ if $\mu \neq \nu$ and $E_\mu = F_\nu \cup F'$ otherwise.
Let $S$ be the $k$-partition tree $(k, t', T^{[\nu, F']})$.
The condition $d = (\vec{E}, S, \Ccal[S])$ is a valid extension of $c$.
By clause~1 of Definition~\ref{def:em-forcing-precondition}, $d \Vdash_\mu \varphi(G)$.
Therefore $I(d) \subseteq I(c) \setminus \{\nu\}$.
\end{proof}

As in the previous notion of forcing, the following trivial lemma expresses the fact that the promise part of a condition
has no effect in the forcing relation for a $\Sigma^0_1$ or $\Pi^0_1$ formula.

\begin{lemma}\label{lem:em-promise-no-effect-first-level}
Fix two conditions $c = (\vec{F}, T, \Ccal)$ and $d = (\vec{F}, T, \Dcal)$, and a $\Sigma^0_1$ ($\Pi^0_1$)
formula. For every part $\nu$ of $T$,
$c \Vdash_\nu \varphi(G)$ if and only if $d \Vdash_\nu \varphi(G)$.
\end{lemma}
\begin{proof}
If $\varphi \in \Sigma^0_1$ then $\varphi(G)$ can be expressed as $(\exists x)\psi(G, x)$ 
where $\psi \in \Sigma^0_0$.
By clause~1 of Definition~\ref{def:em-forcing-precondition},
$c \Vdash_\nu \varphi(G)$ iff
there exists a $w \in \omega$ such that $\psi(F_\nu, w)$ holds,
iff $d \Vdash_\nu \varphi(G)$.
Similarily, if $\varphi \in \Pi^0_1$ then $\varphi(G)$ can be expressed as $(\forall x)\psi(G, x)$
where $\psi \in \Sigma^0_0$.
By clause~2 of Definition~\ref{def:em-forcing-precondition},
$c \Vdash_\nu \varphi(G)$ iff
for every $\sigma \in T$, every $w < |\sigma|$ and 
every $R$-transitive set $F' \subseteq \dom(T) \cap \set_\nu(\sigma)$, $\varphi(F_\nu \cup F', w)$ holds,
iff $d \Vdash_\nu \varphi(G)$.
\end{proof}

We are now ready to define the forcing relation for an arbitrary arithmetic formula.
Again, the natural forcing relation induced by the forcing of~$\Sigma^0_0$ formulas
is too complex, so we design a more effective relation which still
enjoys the main properties of a forcing relation.

\begin{definition}\label{def:em-forcing-condition}
Fix a condition $c = (\vec{F}, T, \Ccal)$ and an arithmetic formula $\varphi(G)$.
\begin{itemize}
	\item[1.] If $\varphi(G) = (\exists x)\psi(G, x)$ where $\psi \in \Pi^0_1$ then
	$c \Vdash \varphi(G)$ iff for every part $\nu < \parts(T)$ such that
	$(\nu, T) \in \Ccal$ there exists a $w < \dom(T)$ such that $c \Vdash_\nu \psi(G, w)$
	\item[2.] If $\varphi(G) = (\forall x)\psi(G, x)$ where $\psi \in \Sigma^0_1$ then
	$c \Vdash \varphi(G)$ iff for every infinite p.r.\ $k'$-partition tree $S$,
	every function~$f : \parts(S) \to \parts(T)$,
	every $w$ and $\vec{E}$ smaller than $\#S$ such that the followings hold
	\begin{itemize}
		\item[i)] $(E_\nu, \dom(S))$ EM extends $(F_{f(\nu)}, \dom(T))$ for each $\nu < \parts(S)$
		\item[ii)] $S$ $f$-refines $\bigcap_{\nu < \parts(S)} T^{[f(\nu), E_\nu]}$
	\end{itemize}
	for every $(\mu, S) \in \Ccal$, $(\vec{E}, S, \Ccal[S]) \not \Vdash_\mu \neg \psi(G, w)$
	\item[3.] If $\varphi(G) = (\exists x)\psi(G, x)$ where $\psi \in \Pi^0_{n+2}$ then
	$c \Vdash \varphi(G)$ iff there exists a $w \in \omega$ such that $c \Vdash \psi(G, w)$
	\item[4.] If $\varphi(G) = \neg \psi(G,x)$ where $\psi \in \Sigma^0_{n+3}$
	then $c \Vdash \varphi(G)$ iff $d \not \Vdash \psi(G)$ for every $d \in \Ext(c)$.
\end{itemize}
\end{definition}

Notice that, unlike the forcing relation for $\Sigma^0_1$ and $\Pi^0_1$ formulas,
the relation over higher formuals does not depend on the part of the relation. 
The careful reader will have recognized the combinatorics of the second jump control 
introduced by Cholak, Jockusch and Slaman in~\cite{Cholak2001strength}.
We now prove the main properties of this forcing relation.

\begin{lemma}\label{lem:em-forcing-extension}
Fix a condition $c$ and a $\Sigma^0_{n+2}$ ($\Pi^0_{n+2}$) formula $\varphi(G)$.
If $c \Vdash \varphi(G)$ then for every $d \leq c$, $d \Vdash \varphi(G)$.
\end{lemma}
\begin{proof}
We prove the statement by induction over the complexity of the formula $\varphi(G)$.
Fix a condition $c = (\vec{F}, T, \Ccal)$ such that $c \Vdash \varphi(G)$ and an extension $d = (\vec{E}, S, \Dcal)$ of~$c$.
\begin{itemize}
	\item If $\varphi \in \Sigma^0_2$ then $\varphi(G)$ can be expressed as $(\exists x)\psi(G, x)$ 
  where $\psi \in \Pi^0_1$. By clause~1 of Definition~\ref{def:em-forcing-condition},
	for every part $\nu$ of $T$ such that $(\nu, T) \in \Ccal$, there exists a $w < \dom(T)$
	such that $c \Vdash_\nu \psi(G, w)$. Fix a part $\mu$ of $S$ such that $(\mu, S) \in \Dcal$.
	As $\Dcal \subseteq \Ccal$, $(\mu, S) \in \Ccal$. By upward-closure of $\Ccal$,
	part $\mu$ of $S$ refines some part $\nu$ of $\Ccal$ such that $(\nu, T) \in \Ccal$.
	Therefore by Lemma~\ref{lem:em-forcing-extension-level1}, $d \Vdash_\mu \psi(G, w)$,
	with $w < \dom(T) \leq \dom(S)$. Applying again clause~1 of Definition~\ref{def:em-forcing-condition},
	we deduce that $d \Vdash (\forall x)\psi(G, x)$, hence $d \Vdash \varphi(G)$.

  \item If $\varphi \in \Pi^0_2$ then $\varphi(G)$ can be expressed as $(\forall x)\psi(G, x)$ where $\psi \in \Sigma^0_1$.
	Suppose by way of contradiction that $d \not \Vdash (\forall x)\psi(G, x)$.
	Let $f : \parts(S) \to \parts(T)$ witness the refinement $S \leq T$.
	By clause~2 of Definition~\ref{def:em-forcing-condition}, there exists an infinite p.r.\ 
	$k'$-partition tree $S'$, a function~$g : \parts(S') \to \parts(S)$, a $w \in \omega$, and $\vec{H}$
	smaller than the code of $S'$ such that
	\begin{itemize}
		\item[i)] $(H_\nu, \dom(S'))$ EM extends $(E_{g(\nu)}, \dom(S))$ for each $\nu < \parts(S')$
		\item[ii)] $S'$ $g$-refines $\bigcap_{\nu < \parts(S')} S^{[g(\nu), H_\nu]}$
		\item[iii)] there exists a $(\mu, S') \in \Dcal$ such that $(\vec{H}, S', \Dcal[S']) \Vdash_\mu \neg \psi(G, w)$.
	\end{itemize}
	To deduce by clause~2 of Definition~\ref{def:em-forcing-condition} that
	$c \not \Vdash (\forall x)\psi(G, x)$ and derive a contradiction, it suffices to prove
	that the same properties hold with respect to $T$.
	\begin{itemize}
		\item[i)] By property (i) of the definition of an extension,
		$(E_{g(\nu)}, \dom(S))$ EM extends $(F_{f(g(\nu))}, \dom(T))$
		and $(H_\nu, \dom(S')$ EM extends $(E_{g(\nu)}, \dom(S))$,
		then $(H_\nu, \dom(S'))$ EM extends $(F_{f(g(\nu))}, \dom(T))$.
		\item[ii)] By property (ii) of the definition of an extension,
		$S$ $f$-refines $\bigcap_{\nu < \parts(S')} T^{[f(g(\nu)), E_{g(\nu)}]}$
		and $S'$ $g$-refines $\bigcap_{\nu < \parts(S')} S^{[g(\nu), H_\nu]}$,
		then $S'$ $(g \circ f)$-refines $\bigcap_{\nu < \parts(S')} T^{[(g \circ f)(\nu), H_\nu]}$.
		\item[iii)] As $\Dcal \subseteq \Ccal$, there exists a part $(\mu, S') \in \Ccal$
		such that $(\vec{H}, S', \Dcal[S']) \Vdash_\mu \neg \psi(G, w)$. 
		By Lemma~\ref{lem:em-promise-no-effect-first-level}, $(\vec{H}, S', \Ccal[S']) \Vdash_\mu \neg \psi(G, w)$.
	\end{itemize}

	\item If $\varphi \in \Sigma^0_{n+3}$ then $\varphi(G)$ can be expressed as $(\exists x)\psi(G, x)$ 
  where $\psi \in \Pi^0_{n+2}$.
  By clause~3 of Definition~\ref{def:em-forcing-condition}, there exists a $w \in \omega$
  such that $c \Vdash \psi(G, w)$. By induction hypothesis, $d \Vdash \psi(G, w)$
  so by clause~3 of Definition~\ref{def:em-forcing-condition}, $d \Vdash \varphi(G)$.

  \item If $\varphi \in \Pi^0_{n+3}$ then $\varphi(G)$ can be expressed as $\neg \psi(G)$ where $\psi \in \Sigma^0_{n+3}$.
  By clause~4 of Definition~\ref{def:em-forcing-condition}, for every $e \in \Ext(c)$, $e \not \Vdash \psi(G)$.
  As $\Ext(d) \subseteq \Ext(c)$, for every $e \in \Ext(d)$, $e \not \Vdash \psi(G)$,
  so by clause~4 of Definition~\ref{def:em-forcing-condition}, $d \Vdash \varphi(G)$.
\end{itemize}
\end{proof}

\begin{lemma}\label{lem:em-forcing-dense}
For every $\Sigma^0_{n+2}$ ($\Pi^0_{n+2}$) formula $\varphi$, the following set is dense
$$
\{c \in \Pb : c \Vdash \varphi(G) \mbox{ or } c \Vdash \neg \varphi(G) \}
$$
\end{lemma}
\begin{proof}
We prove the statement by induction over~$n$.
It suffices to treat the case where $\varphi$ is a $\Sigma^0_{n+2}$ formula,
as the case where $\varphi$ is a $\Pi^0_{n+2}$ formula is symmetric. Fix a condition $c = (\vec{F}, T, \Ccal)$.
\begin{itemize}
	\item In case $n = 0$, the formula $\varphi$ is of the form $(\exists x)\psi(G, x)$ where
	$\psi \in \Pi^0_1$.
	Suppose there exist an infinite p.r. $k'$-partition tree $S$
	for some $k' \in \omega$, a function~$f : \parts(S) \to \parts(T)$ and a $k'$-tuple of finite sets $\vec{E}$ such that
	\begin{itemize}
		\item[i)] $(E_\nu, [\ell, \infty))$ EM extends $(F_{f(\nu)}, \dom(T))$ for each $\nu < \parts(S)$.
		\item[ii)] $S$ $f$-refines $\bigcap_{\nu < \parts(S)} T^{[f(\nu), E_\nu]}$
		\item[iii)] for each non-empty part $\nu$ of $S$ such that $(\nu, S) \in \Ccal$, 
		$(\vec{E}, S, \Ccal[S]) \Vdash_\nu \psi(G, w)$ for some $w < \#S$
	\end{itemize}
	We can choose $\dom(S)$  so that $(E_\nu, \dom(S))$ EM extends $(F_{f(\nu)}, \dom(T))$ for each $\nu < \parts(S)$.
	Properties i-ii) remain trivially true.
	By Lemma~\ref{lem:em-forcing-extension-level1} and Lemma~\ref{lem:em-promise-no-effect-first-level}, property iii) remains true too.
	Let $\Dcal = \Ccal[S] \setminus \{(\nu, S') \in \Ccal : \mbox{ part } \nu \mbox{ of } S' \mbox{ is empty} \}$. 
	As $\Ccal$ is an $\emptyset'$-p.r. promise for $T$,
	$\Ccal[S]$ is an $\emptyset'$-p.r. promise for $S$.
	As $\Dcal$ is obtained from $\Ccal[S]$ by removing only empty parts, $\Dcal$ is also an $\emptyset'$-p.r. promise for $S$.
	By clause~1 of Definition~\ref{def:em-forcing-condition},
	$d = (\vec{E}, S, \Dcal) \Vdash (\exists x)\psi(G, x)$ hence $d \Vdash \varphi(G)$.

	We may choose a coding of the p.r. trees such that
	the code of $S$ is sufficiently large to witness $\ell$ and $\vec{E}$.
	So suppose now that for every infinite p.r. $k'$-partition tree $S$,
	every function~$f : \parts(S) \to \parts(T)$ and $\vec{E}$ smaller than the code of $S$ such that properties i-ii) hold,
	there exists a non-empty part $\nu$ of $S$ such that $(\nu, S) \in \Ccal$
	and $(\vec{E}, S, \Ccal) \not \Vdash_\nu \psi(G, w)$ for every $w < \ell$.
	Let $\Dcal$ be the collection of all such $(\nu, S)$. The set $\Dcal$ is $\emptyset'$-p.r.\ 
	since by Lemma~\ref{lem:em-complexity-forcing}, both $(\vec{E}, S, \Ccal) \not \Vdash_\nu \psi(G, w)$
	and ``part $\nu$ of~$S$ is non-empty'' are $\Sigma^0_1$.
	By Lemma~\ref{lem:em-forcing-extension-level1} and since we require that~$\#S \geq \#T$
	in the definition of~$S \leq T$, $\Dcal$ is upward-closed under the refinement relation,
	hence is a promise for~$T$. By clause~2
	of Definition~\ref{def:em-forcing-condition}, $d = (\vec{F}, T, \Dcal) \Vdash (\forall x) \neg \psi(G, x)$,
	hence $d \Vdash \neg \varphi(G)$.

	\item In case $n > 0$, density follows from clause~4 of Definition~\ref{def:em-forcing-condition}.
\end{itemize}
\end{proof}

By Lemma~\ref{lem:promise-keeps-acceptable-parts}, 
given any filter~$\Fcal = \{c_0, c_1, \dots \}$ with $c_s = (\vec{F}_s, T_s, \Ccal_s)$, 
the set of the acceptable parts~$\nu$ of~$T_s$ such that~$(\nu, T_s) \in \Ccal_s$ forms
an infinite, directed acyclic graph~$\Gcal(\Fcal)$. Whenever~$\Fcal$ is sufficiently generic,
the graph~$\Gcal(\Fcal)$ has a unique infinite path~$P$.
The path~$P$ induces an infinite set~$G = \bigcup_s F_{P(s), s}$. 
We call~$P$ the \emph{generic path} and $G$ the \emph{generic real}.

\begin{lemma}\label{lem:em-generic-level1}
Suppose that $\Fcal$ is sufficiently generic and let~$P$ and~$G$ be the generic path and the generic real, respectively.
For any $\Sigma^0_1$ ($\Pi^0_1$) formula $\varphi(G)$, 
$\varphi(G)$ holds iff $c_s \Vdash_{P(s)} \varphi(G)$ for some $c_s \in \Fcal$.
\end{lemma}
\begin{proof}
Fix a condition $c_s = (\vec{F}, T, \Ccal) \in \Fcal$ such that $c \Vdash_{P(s)} \varphi(G)$,
and let~$\nu = P(s)$. 
\begin{itemize}
	\item If $\varphi \in \Sigma^0_1$ then $\varphi(G)$ can be expressed as $(\exists x)\psi(G, x)$ where $\psi \in \Sigma^0_0$.
	By clause~1 of Definition~\ref{def:em-forcing-precondition}, there exists a $w \in \omega$ such that
	$\psi(F_\nu, w)$ holds. As $\nu = P(s)$, $F_\nu = F_{P(s)} \subseteq G$
	and~$G \setminus F_\nu \subseteq (\max F_\nu, \infty)$, so $\psi(G, w)$ holds by continuity, hence $\varphi(G)$ holds.
	
	\item If $\varphi \in \Pi^0_1$ then $\varphi(G)$ can be expressed as $(\forall x)\psi(G, x)$ where $\psi \in \Sigma^0_0$.
	By clause~2 of Definition~\ref{def:em-forcing-precondition}, for every $\sigma \in T$, every $w < |\sigma|$
	and every $R$-transitive set $F' \subseteq \dom(T) \cap \set_\nu(\sigma)$, $\psi(F_\nu \cup F', w)$ holds. 
	For every $F' \subseteq G \setminus F_\nu$, and $w \in \omega$ there exists a $\sigma \in T$
	such that $w < |\sigma|$ and $F' \subseteq \dom(T) \cap \set_\nu(\sigma)$. Hence $\psi(F_\nu \cup F', w)$ holds.
	Therefore, for every $w \in \omega$, $\psi(G, w)$ holds, so $\varphi(G)$ holds.
\end{itemize}
The other direction holds by Lemma~\ref{lem:em-forcing-dense-level1}.
\end{proof}

\begin{lemma}\label{lem:em-generic-higher-levels}
Suppose that $\Fcal$ is sufficiently generic and let~$P$ and~$G$ be the generic path and the generic real, respectively.
For any $\Sigma^0_{n+2}$ ($\Pi^0_{n+2}$) formula $\varphi(G)$, 
$\varphi(G)$ holds iff $c_s \Vdash \varphi(G)$ for some $c_s \in \Fcal$.
\end{lemma}
\begin{proof}
Assuming the reversal, we first show that if $\varphi(G)$ holds, then $c_s \Vdash \varphi(G)$ for some
$c_s \in \Fcal$. Indeed, by Lemma~\ref{lem:em-forcing-dense} and by genericity of $\Fcal$ either $c_s \Vdash \varphi(G)$ or $c_s \Vdash \neg \varphi(G)$, but if $c \Vdash \neg \varphi(G)$ then $\neg \varphi(G)$ holds, contradicting 
the hypothesis. So $c_s \Vdash \varphi(G)$.
We now prove the forward implication by induction over the complexity of the formula $\varphi(G)$.
Fix a condition $c_s = (\vec{F}, T, \Ccal) \in \Fcal$ such that $c_s \Vdash \varphi(G)$. 
We proceed by case analysis on $\varphi$. 
\begin{itemize}
	\item If $\varphi \in \Sigma^0_2$ then $\varphi(G)$ can be expressed as $(\exists x)\psi(G, x)$ 
  where $\psi \in \Pi^0_1$.
  By clause~1 of Definition~\ref{def:em-forcing-condition}, for every part $\nu$ of $T$
	such that $(\nu, T) \in \Ccal$, there exists a $w < \dom(T)$
  such that $c_s \Vdash_\nu \psi(G, w)$. In particular $(P(s), T) \in \Ccal$, so $c_s \Vdash_{P(s)} \psi(G, w)$.
	By Lemma~\ref{lem:em-generic-level1}, $\psi(G, w)$ holds, hence $\varphi(G)$ holds.

	\item If $\varphi \in \Pi^0_2$ then $\varphi(G)$ can be expressed as $(\forall x)\psi(G, x)$ where $\psi \in \Sigma^0_1$.
	By clause~2 of Definition~\ref{def:em-forcing-condition}, for every infinite $k'$-partition
	tree $S$, every function~$f : \parts(S) \to \parts(T)$, 
	every $w$ and $\vec{E}$ smaller than the code of $S$ such that the followings hold
	\begin{itemize}
		\item[i)] $(E_\nu, \dom(S))$ EM extends $(F_{f(\nu)}, \dom(T))$ for each $\nu < \parts(S)$
		\item[ii)] $S$ $f$-refines $\bigcap_{\nu < \parts(S)} T^{[f(\nu), E_\nu]}$
	\end{itemize}
	for every $(\mu, S) \in \Ccal$, $(\vec{E}, S, \Ccal[S]) \not \Vdash_\mu \neg \psi(G, w)$.
	Suppose by way of contradiction that $\psi(G, w)$ does not hold for some $w \in \omega$.
  Then by Lemma~\ref{lem:em-generic-level1}, there exists a $d_t \in \Fcal$
	such that $d_t \Vdash_{P(t)} \neg \psi(G, w)$.
	Since~$\Fcal$ is a filter, there is a condition $e_r = (\vec{E}, S, \Dcal) \in \Fcal$
	extending both~$c_s$ and~$d_t$.
	Let~$\mu = P(r)$. By choice of~$P$, $(\mu, S) \in \Ccal$, so
	by clause ii), $(\vec{E}, S, \Ccal[S]) \not \Vdash_\mu \psi(G, w)$,
	hence by Lemma~\ref{lem:em-promise-no-effect-first-level}, $e_r \not \Vdash_\mu \neg \psi(G, w)$.
	However, since part~$\mu$ of~$S$ refines part~$P(t)$ of~$d_t$,
	then by Lemma~\ref{lem:em-forcing-extension-level1}, $e_r \Vdash_\mu \neg \psi(G, w)$. Contradiction.
	Hence for every $w \in \omega$, $\psi(G, w)$ holds, so $\varphi(G)$ holds.

	\item If $\varphi \in \Sigma^0_{n+3}$ then $\varphi(G)$ can be expressed as $(\exists x)\psi(G, x)$ 
  where $\psi \in \Pi^0_{n+2}$.
  By clause~3 of Definition~\ref{def:em-forcing-condition}, there exists a $w \in \omega$
  such that $c_s \Vdash \psi(G, w)$. By induction hypothesis, $\psi(G, w)$ holds, hence $\varphi(G)$ holds.

  Conversely, if $\varphi(G)$ holds, then there exists a $w \in \omega$ such that $\psi(G, w)$ holds,
  so by induction hypothesis $c_s \Vdash \psi(G, w)$ for some $c_s \in \Fcal$,
  so by clause~3 of Definition~\ref{def:em-forcing-condition}, $c_s \Vdash \varphi(G)$.
	The proof of the reversal is not redundant with the first paragraph of the proof since
	it is used in the next case at the same rank.

	\item If $\varphi \in \Pi^0_{n+3}$ then $\varphi(G)$ can be expressed as $\neg \psi(G)$ where $\psi \in \Sigma^0_{n+3}$.
  By clause~4 of Definition~\ref{def:em-forcing-condition}, for every $d \in \Ext(c_s)$, $d \not \Vdash \psi(G)$.
	By Lemma~\ref{lem:em-forcing-extension}, $d \not \Vdash \psi(G)$ for every~$d \in \Fcal$
 	and by the previous case, $\psi(G)$ does not hold, so $\varphi(G)$ holds.
\end{itemize}
\end{proof}

We now prove that the forcing relation has good definitional properties
as we did with the notion of forcing for cohesiveness.

\begin{lemma}\label{lem:em-extension-complexity}
For every condition $c$, $\Ext(c)$ is $\Pi^0_2$ uniformly in~$c$.
\end{lemma}
\begin{proof}
Recall from Lemma~\ref{lem:em-refinement-complexity} that
given $k, t \in \omega$, $PartTree(k, t)$ denotes 
the $\Pi^0_1$ set of all the infinite p.r.\ $k$-partition trees of $[t, \infty)$,
and given a $k$-partition tree $S$ and a part $\nu$ of $S$,
the predicate $Empty(S, \nu)$ denotes the $\Pi^0_1$ formula ``part $\nu$ of $S$ is empty'',
that is, the formula $(\forall \sigma \in S)[\set_\nu(\sigma) \cap \dom(S) = \emptyset]$.
If $T$ is p.r. then so is $T^{[\nu, H]}$ for some finite set $H$.

Fix a condition $c = (\vec{F}, (k, t, T), \Ccal)$.
By definition, $(\vec{H}, (k', t', S), \Dcal) \in \Ext(c)$ iff the following formula holds:
$$
\begin{array}{l@{\hskip 0.5in}r}
(\exists f : k' \to k)\\
(\forall \nu < k')(H_\nu, [t', \infty)) \mbox{ EM extends } (F_{f(\nu)}, [t, \infty)) & (\Pi^0_1)\\
\wedge S \in PartTree(k', t') \wedge S \leq_f \bigwedge_{\nu < k'} T^{[f(\nu), H_{\nu}]} & (\Pi^0_1)\\
\wedge \Dcal \mbox{ is a promise for } S \wedge \Dcal \subseteq \Ccal & (\Pi^0_2)\\
\end{array}
$$
By Lemma~\ref{lem:em-refinement-complexity}
and the fact that $\bigwedge_{\nu < k'} T^{[f(\nu), H_{\nu}]}$
is p.r. uniformly in $T$, $f$, $\vec{H}$ and $k'$,
the above formula is $\Pi^0_2$.
\end{proof}

\begin{lemma}\label{lem:em-complexity-forcing}
Fix an arithmetic formula $\varphi(G)$, a condition $c = (\vec{F}, T, \Ccal)$
and a part $\nu$ of $T$.
\begin{itemize}
	\item[a)] If $\varphi(G)$ is a $\Sigma^0_1$ ($\Pi^0_1$) formula 
	then so is the predicate $c \Vdash_\nu \varphi(G)$.
	\item[b)] If $\varphi(G)$ is a $\Sigma^0_{n+2}$ ($\Pi^0_{n+2}$) formula 
	then so is the predicate $c \Vdash \varphi(G)$.
\end{itemize}
\end{lemma}
\begin{proof}
We prove our lemma by induction over the complexity of the formula $\varphi(G)$.
\begin{itemize}
	\item If $\varphi(G) \in \Sigma^0_1$ then it can be expressed as $(\exists x)\psi(G, x)$ where $\psi \in \Sigma^0_0$.
	By clause~1 of Definition~\ref{def:em-forcing-precondition}, $c \Vdash_\nu \varphi(G)$ if and only if 
	the formula $(\exists w \in \omega)\psi(F_\nu, w)$ holds. This is a $\Sigma^0_1$ predicate.
	
	\item If $\varphi(G) \in \Pi^0_1$ then it can be expressed as $(\forall x)\psi(G, x)$ where $\psi \in \Sigma^0_0$.
	By clause~2 of Definition~\ref{def:em-forcing-precondition}, $c \Vdash_\nu \varphi(G)$ if and only if 
	the formula $(\forall \sigma \in T)(\forall w < |\sigma|)(\forall F' \subseteq \dom(T) \cap \set_\nu(\sigma))
	[F'\ R\mbox{-transitive} \imp \psi(F_\nu \cup F', w)]$ holds. This is a $\Pi^0_1$ predicate.

	\item If $\varphi(G) \in \Sigma^0_2$ then it can be expressed as $(\exists x)\psi(G, x)$ where $\psi \in \Pi^0_1$.
	By clause~1 of Definition~\ref{def:em-forcing-condition}, $c \Vdash \varphi(G)$ if and only if 
	the formula $(\forall \nu < \parts(T))(\exists w < \dom(T))[(\nu, T) \in \Ccal \imp c \Vdash_\nu \psi(G, w)]$ holds.
	This is a $\Sigma^0_2$ predicate by induction hypothesis and the fact that $\Ccal$ is $\emptyset'$-computable.

	\item If $\varphi(G) \in \Pi^0_2$ then it can be expressed as $(\forall x)\psi(G, x)$ where $\psi \in \Sigma^0_1$.
	By clause~2 of Definition~\ref{def:em-forcing-condition}, $c \Vdash \varphi(G)$ if and only if 
	for every infinite $k'$-partition tree $S$, every function~$f : \parts(S) \to \parts(T)$,
	every $w$ and $\vec{E}$ smaller than the code of $S$ such that the followings hold
	\begin{itemize}
		\item[i)] $(E_\nu, \dom(S))$ EM extends $(F_{f(\nu)}, \dom(T))$ for each $\nu < \parts(S)$
		\item[ii)] $S$ $f$-refines $\bigcap_{\nu < \parts(S)} T^{[f(\nu), E_\nu]}$
	\end{itemize}
	for every $(\mu, S) \in \Ccal$, $(\vec{E}, S, \Ccal[S]) \not \Vdash_\mu \neg \psi(G, w)$.
	By Lemma~\ref{lem:em-refinement-complexity}, Properties i-ii) are $\Delta^0_2$.
	Moreover, the predicate $(\mu, S) \in \Ccal$ is $\Delta^0_2$.
	By induction hypothesis, $(\vec{E}, S, \Ccal) \not \Vdash_\mu \neg \psi(G, w)$ is $\Sigma^0_1$.
	Therefore $c \Vdash \varphi(G)$ is a $\Pi^0_2$ predicate.

	\item If $\varphi(G) \in \Sigma^0_{n+3}$ then it can be expressed as $(\exists x)\psi(G, x)$ where $\psi \in \Pi^0_{n+2}$.
	By clause~3 of Definition~\ref{def:em-forcing-condition}, $c \Vdash \varphi(G)$ if and only if 
	the formula $(\exists w \in \omega)c \Vdash \psi(G, w)$ holds. This is a $\Sigma^0_{n+3}$ predicate
	by induction hypothesis.

	\item If $\varphi(G) \in \Pi^0_{n+3}$ then it can be expressed as $\neg \psi(G)$ where $\psi \in \Sigma^0_{n+3}$. 
	By clause~4 of Definition~\ref{def:em-forcing-condition}, $c \Vdash \varphi(G)$ if and only if 
	the formula $(\forall d)(d \not \in \Ext(c) \vee d \not \Vdash \psi(G))$ holds.
	By induction hypothesis, $d \not \Vdash \psi(G)$ is a $\Pi^0_{n+3}$ predicate.
	By Lemma~\ref{lem:em-extension-complexity}, the set $\Ext(c)$ is $\Pi^0_2$-computable uniformly in $c$,
	thus $c \Vdash \varphi(G)$ is a $\Pi^0_{n+3}$ predicate.
\end{itemize}
\end{proof}

\subsection{Preserving the arithmetic hierarchy}

We now prove the core lemmas showing that every sufficiently generic real
preserves the arithmetic hierarchy. The proof is split into two lemmas
since the forcing relation for $\Sigma^0_1$ and~$\Pi^0_1$ formulas
depends on the part of the condition, and therefore has to be treated separately.

\begin{lemma}\label{lem:em-diagonalization-level1}
If $A \not \in \Sigma^0_1$ and $\varphi(G, x)$ is $\Sigma^0_1$,
then the set of $c = (\vec{F}, T, \Ccal) \in \Pb$ satisfying the following property is dense:
$$
(\forall \nu < \parts(T))[(\exists w \in A)c_s \Vdash_\nu \neg \varphi(G, w)] \vee [(\exists w \not \in A)c_s \Vdash_\nu \varphi(G, w)]
$$
\end{lemma}
\begin{proof}
The formula $\varphi(G, w)$ can be expressed as $(\exists x)\psi(G, w, x)$ where $\psi \in \Sigma^0_0$.
Given a condition $c = (\vec{F}, T, \Ccal)$, let $I(c)$ be the set of the parts $\nu$ of $T$
such that for every $w \in A$, $c \not \Vdash_\nu \neg \varphi(G, w)$
and for every $w \in \overline{A}$, $c \not \Vdash_\nu \varphi(G, w)$.
If $I(c) = \emptyset$ then we are done, so suppose $I(c) \neq \emptyset$ and
fix some $\nu \in I(c)$. We will construct an extension $d$ such that 
$I(d) \subseteq I(c) \setminus \{\nu\}$. Iterating the operation completes the proof.

Say that $T$ is a $k$-partition tree of $[t, \infty)$ for some $k, t \in \omega$.
Define $f : k+1 \to k$ as $f(\mu) = \mu$ if $\mu < k$ and $f(k) = \nu$ otherwise.
Given an integer $w \in \omega$, let $S_w$ be the set of all $\sigma \in (k+1)^{<\omega}$
which $f$-refine some $\tau \in T \cap k^{|\sigma|}$ and such that for every $u < |\sigma|$,
every part $\mu \in \{\nu, k\}$ and every finite $R$-transitive set $F' \subseteq \dom(T) \cap \set_\mu(\sigma)$,
$\varphi(F_\nu \cup F', w, u)$ does not hold.

The set $S_w$ is a p.r.\ (uniformly in $w$) partition tree of $[t, \infty)$ refining $T$ with witness function~$f$.
Let $U = \{ w \in \omega : S_w \mbox{ is finite } \}$. $U \in \Sigma^0_1$, thus $U \neq A$.
Fix some $w \in U \Delta A$. Suppose first that $w \in A \setminus U$. By definition of $U$, $S_w$ is infinite. 
Let $\vec{E}$ be defined by $E_\mu = F_\mu$ if $\mu < k$ and $E_k = F_\nu$,
and consider the extension $d = (\vec{E}, S_w, \Ccal[S_w])$. We claim that $I(d) \subseteq I(c) \setminus \{\nu\}$. 
Fix a part $\mu \in \{\nu, k\}$ of $S_w$. By definition of $S_w$,
for every $\sigma \in S_w$, every $u < |\sigma|$ and every $R$-transitive set $F' \subseteq \dom(S_w) \cap \set_\mu(\sigma)$,
$\varphi(E_\mu \cup F', w, u)$ does not hold. Therefore, by clause~2 of Definition~\ref{def:em-forcing-precondition},
$d \Vdash_\mu (\forall x)\neg \psi(G, w, x)$, hence $d \Vdash_\mu \neg \varphi(G, w)$, and this for some $w \in A$.
Thus $I(d) \subseteq I(c) \setminus \{\nu\}$. 

Suppose now that $w \in U \setminus A$, so $S_w$ is finite.
Fix an $\ell \in \omega$ such that $(\forall \sigma \in S)|\sigma| < \ell$
and a $\tau \in T \cap k^\ell$ such that $T^{[\tau]}$ is infinite.
Consider the 2-partition $E_0 \cup E_1$ of $\set_\nu(\tau) \cap \dom(T)$
defined by $E_0 = \{i \geq t : \tau(i) = \nu \wedge (\exists n)(\forall s > n) R(i, s) \mbox{ holds}\}$
and $E_0 = \{i \geq t : \tau(i) = \nu \wedge (\exists n)(\forall s > n) R(s, i) \mbox{ holds}\}$.
As there exists no $\sigma \in S_w$ which $f$-refines $\tau$, there exists a $u < \ell$
and an $R$-transitive set $F' \subseteq E_0$ or $F' \subseteq E_1$ such that $\varphi(F_\nu \cup F', w, u)$ holds.
By choice of the partition, there exists a $t' > t$ such that $F' \to_R [t', \infty)$
or $[t', \infty) \to_R F'$. 
By Lemma~\ref{lem:emo-cond-valid}, $(F_\nu \cup F', [t', \infty))$ is a valid EM extension of $(F_\nu, [t, \infty))$.
As $T^{[\tau]}$ is infinite, $T^{[\nu, F']}$ is also infinite.
Let $\vec{E}$ be defined by $E_\mu = F_\mu$ if $\mu \neq \nu$ and $E_\mu = F_\nu \cup F'$ otherwise.
Let $S$ be the $k$-partition tree $(k, t', T^{[\nu, F']})$.
The condition $d = (\vec{E}, S, \Ccal[S])$ is a valid extension of $c$.
By clause~1 of Definition~\ref{def:em-forcing-precondition}, $d \Vdash_\mu \varphi(G, w)$ with $w \not \in A$. .
Therefore $I(d) \subseteq I(c) \setminus \{\nu\}$.
\end{proof}

\begin{lemma}\label{lem:em-diagonalization}
If $A \not \in \Sigma^0_{n+2}$ and $\varphi(G, x)$ is $\Sigma^0_{n+2}$,
then the set of $c \in \Pb$ satisfying the following property is dense:
$$
[(\exists w \in A)c \Vdash \neg \varphi(G, w)] \vee [(\exists w \not \in A)c \Vdash \varphi(G, w)]
$$
\end{lemma}
\begin{proof}
Fix a condition $c = (\vec{F}, T, \Ccal)$.
\begin{itemize}
	\item In case $n = 0$, $\varphi(G, w)$ can be expressed as $(\exists x)\psi(G, w, x)$ where $\psi \in \Pi^0_1$.
	Let $U$ be the set of integers $w$ such that there exists an infinite p.r. $k'$-partition tree $S$
	for some $k' \in \omega$, a function~$f : \parts(S) \to \parts(T)$ and a $k'$-tuple of finite sets $\vec{E}$ such that
	\begin{itemize}
		\item[i)] $(E_\nu, [\ell, \infty))$ EM extends $(F_{f(\nu)}, \dom(T))$ for each $\nu < \parts(S)$.
		\item[ii)] $S$ $f$-refines $\bigcap_{\nu < \parts(S)} T^{[f(\nu), E_\nu]}$
		\item[iii)] for each non-empty part $\nu$ of $S$ such that $(\nu, S) \in \Ccal$, 
		$(\vec{E}, S, \Ccal[S]) \Vdash_\nu \psi(G, w, u)$ for some $u < \#S$
	\end{itemize}
	By Lemma~\ref{lem:em-complexity-forcing} and Lemma~\ref{lem:em-refinement-complexity},
	$U \in \Sigma^0_2$, thus $U \neq A$. Let $w \in U \Delta A$.
	Suppose that $w \in U \setminus A$.
	We can choose $\dom(S)$  so that $(E_\nu, \dom(S))$ EM extends $(F_{f(\nu)}, \dom(T))$ for each $\nu < \parts(S)$.
	By Lemma~\ref{lem:em-forcing-extension-level1} and Lemma~\ref{lem:em-promise-no-effect-first-level}, properties i-ii) remain true.
	Let $\Dcal = \Ccal[S] \setminus \{(\nu, S') \in \Ccal : \mbox{ part } \nu \mbox{ of } S' \mbox{ is empty} \}$. 
	As $\Ccal$ is an $\emptyset'$-p.r. promise for $T$,
	$\Ccal[S]$ is an $\emptyset'$-p.r. promise for $S$.
	As $\Dcal$ is obtained from $\Ccal[S]$ by removing only empty parts, $\Dcal$ is also an $\emptyset'$-p.r. promise for $S$.
	By clause~1 of Definition~\ref{def:em-forcing-condition},
	$d = (\vec{E}, S, \Dcal) \Vdash (\exists x)\psi(G, w, x)$ hence $d \Vdash \varphi(G, w)$ for some $w \not \in A$.

	We may choose a coding of the p.r. trees such that
	the code of $S$ is sufficiently large to witness $u$ and $\vec{E}$.
	So suppose now that $w \in A \setminus U$. Then for every infinite p.r. $k'$-partition tree $S$,
	every $\ell$ and $\vec{E}$ smaller than the code of $S$ such that properties i-ii) hold,
	there exists a non-empty part $\nu$ of $S$ such that $(\nu, S) \in \Ccal$
	and $(\vec{E}, S, \Ccal) \not \Vdash_\nu \psi(G, w, u)$ for every $u < \ell$.
	Let $\Dcal$ be the collection of all such $(\nu, S)$. The set $\Dcal$ is $\emptyset'$-p.r.
	By Lemma~\ref{lem:em-forcing-extension-level1} and since~$\#S \geq \#T$
	whenever~$S \leq_f T$, $\Dcal$ is upward-closed under the refinement relation,
	hence it is a promise for~$T$. By clause~2.
	of Definition~\ref{def:em-forcing-condition}, $d = (\vec{F}, T, \Dcal) \Vdash (\forall x) \neg \psi(G, w, x)$,
	hence $d \Vdash \neg \varphi(G, w)$ for some $w \in A$.

	\item In case $n > 0$, let $U = \{ w \in \omega : (\exists d \in \Ext(c)) d \Vdash \varphi(G, w) \}$.
	By Lemma~\ref{lem:em-extension-complexity} and Lemma~\ref{lem:em-complexity-forcing},
	$U \in \Sigma^0_{n+2}$, thus $U \neq A$.
	Fix some $w \in U \Delta A$. If $w \in U \setminus A$ then by definition of~$U$,
	there exists a condition $d$ extending $c$ such that $d \Vdash \varphi(G, w)$.
	If $w \in A \setminus U$, then for every $d \in \Ext(c)$, $d \not \Vdash \varphi(G, w)$
	so by clause~4 of Definition~\ref{def:em-forcing-condition}, $c \Vdash \neg \varphi(G, w)$. 
\end{itemize}
\end{proof}

We are now ready to prove Theorem~\ref{thm:em-preserves-arithmetic}.
It follows from the preservation of the arithmetic hierarchy
for cohesiveness and the stable Erd\H{o}s-Moser theorem.

\begin{proof}[Proof of Theorem~\ref{thm:em-preserves-arithmetic}]
Since $\rca \vdash \coh \wedge \semo \imp \emo$, then by Theorem~\ref{thm:coh-preservation-arithmetic-hierarchy}
it suffices to prove that $\semo$ admits preservation of the arithmetic hierarchy.
Fix some set~$C$ and a $C$-computable stable infinite tournament~$R$.
Let $\Ccal_0$ be the $C'$-p.r. set of all $(\nu, T) \in \TPb$ such that $(\nu, T) \leq (0, 1^{<\omega})$.
Let~$\Fcal$ be a sufficiently generic filter containing~$c_0 = (\{\emptyset\}, 1^{<\omega}, \Ccal_0)$.
Let~$P$ and $G$ be the corresponding generic path and generic real, respectively.
By definition of a condition, the set~$G$ is $R$-transitive.
By Lemma~\ref{lem:em-forcing-infinite}, $G$ is infinite. By Lemma~\ref{lem:em-diagonalization-level1}
and Lemma~\ref{lem:em-complexity-forcing}, $G$ preserves non-$\Sigma^0_1$ definitions relative to~$C$.
By Lemma~\ref{lem:em-diagonalization} and Lemma~\ref{lem:em-complexity-forcing},
$G$ preserves non-$\Sigma^0_{n+2}$ definitions relative to~$C$ for every~$n \in \omega$.
Therefore, by Proposition 2.2 of~\cite{Wang2014Definability}, $G$
preserves the arithmetic hierarchy relative to~$C$.
\end{proof}

\section{$\mathsf{D}^2_2$ preserves higher definitions}

Among the Ramsey-type hierarchies, the $\mathsf{D}$ hierarchy
is conceptually the simplest one. It is therefore natural
to study it in order to understand better the control of iterated jumps
and focus on the core combinatorics without the technicalities specific to
another hierarchy.

\begin{definition}
For every~$n, k \geq 1$, 
$\mathsf{D}^n_k$ is the statement ``Every~$\Delta^0_n$ $k$-partition of the integers has an infinite subset in one of its parts''.
\end{definition}

In particular, $\mathsf{D}^1_k$ is nothing but $\rt^1_k$ for computable colorings.
Cholak et al.~\cite{Cholak2001strength} proved that~$\mathsf{D}^2_k$ and stable Ramsey's theorem
for pairs and~$k$ colors ($\srt^2_k$) are computably equivalent and that the proof is formalizable over~$\rca+\bst$.
Later, Chong et al.~\cite{Chong2010role} proved that~$\mathsf{D}^2_2$ implies~$\bst$ over~$\rca$,
showing therefore that~$\rca \vdash \mathsf{D}^2_k \biimp \srt^2_\ell$ for every~$k, \ell \geq 2$.
Wang~\cite{Wang2014Definability} studied Ramsey's theorem within his framework of preservation of definitions
and proved that $\mathsf{D}^2_2$ admits preservation of $\Xi$ definitions 
simultaneously for all $\Xi$ in $\{\Sigma^0_{n+2}, \allowbreak \Pi^0_{n+2}, \allowbreak \Delta^0_{n+2} : n \in \omega \}$, 
but not~$\Delta^0_2$ definitions. More precisely, he prove that $\sads$, which is a consequence of $\mathsf{D}^2_2$, does not admit preservation of $\Delta^0_2$ definitions. He used for this a combination of the first jump control
of Cholak, Jockusch and Slaman~\cite{Cholak2001strength} and a relativization
of the preservation of the arithmetic hierarchy by~$\wkl$.

In this section, we design a notion of forcing for~$\mathsf{D}^2_2$
with a forcing relation which has the same definitional complexity
as the formula it forces. It enables us to reprove that $\mathsf{D}^2_2$
admits preservation of $\Xi$ definitions 
simultaneously for all $\Xi$ in $\{\Sigma^0_{n+2}, \allowbreak \Pi^0_{n+2}, \allowbreak \Delta^0_{n+3} : n \in \omega \}$.
The proof is significantly more involved than the previous proofs
of preservation of the arithmetic hierarchy. 

\subsection{Sides of a sequence of sets}

A main feature in the construction of a solution to an instance $R_0, R_1$ of $\mathsf{D}^2_2$
is the parallel construction of a subset of $R_0$ and a subset of $R_1$.
The intrinsic disjunction in the forcing argument prevents us from applying
the same strategy as for the Erd\H{o}s-Moser theorem and obtain a preservation of the arithmetic hierarchy.
Given some~$\alpha < 2$, we shall refer to $R_\alpha$ or simply $\alpha$ as a \emph{side} of $\vec{R}$.
We also need to define a relative notion of acceptation and emptiness of a part.

\begin{definition}Fix a $k$-partition tree $T$ of $[t, \infty)$ and a set $X$.
We say that part $\nu$ of $T$ is \emph{$X$-acceptable} if there exists a path $P$ through $T$
such that $\set_\nu(P) \cap X$ is infinite.
We say that part $\nu$ of $T$ is \emph{$X$-empty} if 
$(\forall \sigma \in T)[\dom(T) \cap \set_\nu(\sigma) \cap X = \emptyset]$.
\end{definition}

The intended uses of those notions will be $R_\alpha$-acceptation
and $R_\alpha$-emptiness. Every partition tree has an $R_\alpha$-acceptable part for some $\alpha < 2$.
The notion of $X$-emptiness is $\Pi^{0,X}_1$, and therefore $\Pi^0_2$ if $X$ is $\Delta^0_2$,
which raises new problems for obtaining a forcing relation of weak definitional complexity.
We would like to define a stronger notion of ``witnessing its acceptable parts'' and prove
that for every infinite p.r.\ partition tree~$T$, there is a p.r.\ refined tree~$S$
such that for each side $\alpha$ and each part~$\nu$ of~$S$, either~$\nu$ is $R_\alpha$-empty
in~$S$, or~$\nu$ is $R_\alpha$-acceptable. However, the resulting tree~$S$ would be $\emptyset'$-p.r.\
since~$R_\alpha$ is $\emptyset'$-computable. Thankfully, we will be able to circumvent this problem
in Lemma~\ref{lem:cohzp-validity-exists}.

\subsection{Forcing conditions}

Fix a $\Delta^0_2$ 2-partition~$R_0 \cup R_1 = \omega$.
We now describe the notion of forcing to build an infinite subset
of~$R_0$ or of~$R_1$.

\begin{definition}
We denote by~$\Pb$ the forcing notion whose conditions are tuples
 $((F_\nu^\alpha : \alpha < 2, \nu < k), T, \Ccal)$ where
\begin{itemize}
	\item[(a)] $T$ is an infinite, p.r.\ $k$-partition tree
	\item[(b)] $\Ccal$ is a $\emptyset'$-p.r.\ promise for $T$
	\item[(c)] $(F^\alpha_\nu, \dom(T))$ is a Mathias condition for each $\nu < k$ and $\alpha < 2$
\end{itemize}
A condition $d = (\vec{E}, S, \Dcal)$ \emph{extends} $c = (\vec{F}, T, \Ccal)$
(written $d \leq c$) if there exists a function $f : \parts(S) \to \parts(T)$ such that
$\Dcal \subseteq \Ccal$ and the followings hold
\begin{itemize}
	\item[(i)] $(E^\alpha_\nu, \dom(S) \cap R_\alpha)$ Mathias extends $(F^\alpha_{f(\nu)}, \dom(T) \cap R_\alpha)$ 
	for each $\nu < \parts(S)$ and $\alpha < 2$ 
	\item[(ii)] $S$ $f$-refines $\bigcap_{\nu < \parts(S), \alpha < 2} T^{[f(\nu), E^\alpha_\nu]}$
\end{itemize}
\end{definition}

In the whole construction, the index $\alpha$ indicates that we are constructing a set which is almost included in $R_\alpha$.
Given a condition $c = (\vec{F}, T, \Ccal)$, we write again $\parts(c)$ for $\parts(T)$.
The following lemma shows that we can force our constructed set to be infinite
if we choose it among the acceptable parts.

\begin{lemma}\label{lem:cohzp-forcing-infinite}
For every condition $c = (\vec{F}, T, \Ccal)$ and every $n \in \omega$, there exists an extension $d = (\vec{E}, S, \Dcal)$
such that $|E^\alpha_\nu| \geq n$ on each $R_\alpha$-acceptable part $\nu$ of~$S$ for each $\alpha < 2$.
\end{lemma}
\begin{proof}
It suffices to prove that for every condition $c = (\vec{F}, T, \Ccal)$, every side $\alpha < 2$
and every $R_\alpha$-acceptable part $\nu$ of $T$, 
there exists an extension $d = (\vec{E}, S, \Dcal)$ such that $S \leq_{id} T$ 
and $|E^\alpha_\nu| \geq n$. Iterating the process at most $\parts(T) \times 2$ times completes the proof.
Fix an $R_\alpha$-acceptable part $\nu$ of $T$ and a path $P$ through $T$
such that $\set_\nu(P) \cap R_\alpha$ is infinite. Let $F'$ be a subset of $\set_\nu(P) \cap \dom(T) \cap R_\alpha$
of size $n$. Let $\vec{E}$ be defined by $E^\beta_\mu = F^\beta_\mu$
if $\mu \neq \nu \vee \beta \neq \alpha$ and $E^\alpha_\nu = F^\alpha_\nu \cup F'$ otherwise.
Let $S$ be the p.r.\ partition tree obtained from $T^{[\nu, E^\alpha_\nu]}$ by restricting its domain so that
$(E^\alpha_\nu, \dom(S) \cap R_\alpha)$ Mathias extends $(F^\alpha_\nu, \dom(T) \cap R_\alpha)$.
The condition $(\vec{E}, S, \Ccal[S])$ is the desired extension.
\end{proof}

Given a condition~$c$, we denote by~$\Ext(c)$ the set of all its extensions.

\subsection{Forcing relation}

We need to define two forcing relations at the first level:
the ``true'' forcing relation, i.e., the one having the good density properties
but whose decision requires too much computational power, and a ``weak'' forcing relation
having better computational properties, but which does not behave well with respect to the forcing.
We start with the definition of the true forcing relation.

\begin{definition}[True forcing relation]\label{def:cohzp-true-forcing-precondition}
Fix a condition $c = (\vec{F}, T, \Ccal)$, a $\Sigma^0_0$ formula $\varphi(G, x)$,
a part $\nu < \parts(T)$, and a side $\alpha < 2$.
\begin{itemize}
	\item[1.] $c \Vvdash^\alpha_\nu (\exists x)\varphi(G, x)$ iff there exists a $w \in \omega$ such that $\varphi(F^\alpha_\nu, w)$ holds.
	\item[2.] $c \Vvdash^\alpha_\nu (\forall x)\varphi(G, x)$ iff 
	for every $\sigma \in T$ such that $T^{[\sigma]}$ is infinite, 
	every $w < |\sigma|$ and every set $F' \subseteq \dom(T) \cap \set_\nu(\sigma) \cap R_\alpha$,
	$\varphi(F^\alpha_\nu \cup F', w)$ holds.
\end{itemize}
\end{definition}

Given a condition $c$, a side $\alpha < 2$, a part $\nu$ of $c$ and a $\Pi^0_1$ formula $\varphi$,
the relation $c \Vvdash^\alpha_\nu \varphi(G)$ is $\Pi^{0, \emptyset' \oplus R_\alpha}_1$,
hence $\Pi^0_2$ as $R_\alpha$ is $\Delta^0_2$. This relation
enjoys the good properties of a forcing relation, that is, it is downward-closed
under the refinement relation (Lemma~\ref{lem:cohzp-true-forcing-extension-level1}), 
and the set of the conditions forcing either a $\Sigma^0_1$
formula or its negation is dense (Lemma~\ref{lem:cohzp-true-forcing-dense-level1}).

\begin{lemma}\label{lem:cohzp-true-forcing-extension-level1}
Fix a condition $c = (\vec{F}, T, \Ccal)$ and a $\Sigma^0_1$ ($\Pi^0_1$) formula $\varphi(G)$.
If $c \Vvdash^\alpha_\nu \varphi(G)$ for some $\nu < \parts(T)$ and $\alpha < 2$,
then for every $d = (\vec{E}, S, \Dcal) \leq c$ and
every part $\mu$ of $S$ refining part $\nu$ of $T$, $d \Vvdash^\alpha_\mu \varphi(G)$.
\end{lemma}
\begin{proof}\ 
\begin{itemize}
	\item If $\varphi \in \Sigma^0_1$ then it can be expressed as $(\exists x)\psi(G, x)$
	where $\psi \in \Sigma^0_0$. By clause~1 of Definition~\ref{def:cohzp-true-forcing-precondition},
	there exists a $w \in \omega$ such that
	$\psi(F^\alpha_\nu, w)$ holds. By property (i) of the definition of an extension, $E^\alpha_\mu \supseteq F^\alpha_\nu$
	and $(E^\alpha_\mu \setminus F^\alpha_\nu) \subset \dom(T) \cap R_\alpha$, 
	therefore by continuity $\psi(E^\alpha_\mu, w)$ holds,
	so by clause~1 of Definition~\ref{def:cohzp-true-forcing-precondition}, 
	$d \Vvdash^\alpha_\mu (\exists x)\psi(G, x)$.

	\item If $\varphi \in \Pi^0_1$ then it can be expressed as $(\forall x)\psi(G, x)$
	where $\psi \in \Sigma^0_0$.
	Fix a $\tau \in S$ such that $S^{[\tau]}$ is infinite, a $w < |\tau|$ and a set
	$F' \subseteq \dom(S) \cap \set_\mu(\tau) \cap R_\alpha$. Let~$f$ be the function witnesing~$d \leq c$. 
	By property~(ii) of the definition of an extension, $\tau$ $f$-refines a $\sigma \in T^{[\nu, E^\alpha_\mu]}$.
	We claim that we can even choose~$\sigma$ to be extendible in $T^{[\nu, E^\alpha_\mu]}$.
	Indeed, since~$\tau$ is extendible in~$S$, let~$P$ be a path through~$S$ extending~$\tau$
	and let~$U$ be the set of~$\sigma$'s in~$T$ such that~$P \uh s$ $f$-refines~$\sigma$ for some~$s$.
	The set~$U$ is an infinite subtree of~$T$. Let~$\sigma$ be a string of length~$|\tau|$ and extendible in~$U$, hence in~$T$.
	By definition of~$U$, $\tau$ $f$-refines $\sigma$. By definition of a refinement, 
	such that $|\sigma| = |\tau|$ and $\set_\mu(\tau) \subseteq \set_\nu(\sigma)$. As $w < |\tau|$
	and $\dom(S) \subseteq \dom(T)$, $F' \subseteq \dom(T) \cap \set_\nu(\sigma) \cap R_\alpha$.
	As $\sigma \in T^{[\nu, E^\alpha_\mu]}$, $E^\alpha_\mu \subseteq \set_\nu(\sigma)$
	and by property~(i) of the definition of an extension, $E^\alpha_\mu \subseteq \dom(T) \cap R_\alpha$
	so $E^\alpha_\mu \subseteq \dom(T) \cap R_\alpha$.
	Therefore $E^\alpha_\mu \cup F' \subseteq \dom(T) \cap \set_\nu(\sigma) \cap R_\alpha$.
	By clause~2 of Definition~\ref{def:cohzp-true-forcing-precondition} applied to $c \Vvdash^\alpha_\nu (\forall x)\psi(G, x)$, 
	$\psi(F^\alpha_\nu \cup (E^\alpha_\mu \setminus F^\alpha_\nu) \cup F', w)$ holds, hence $\psi(E^\alpha_\mu \cup F', w)$ holds
	and still by clause~2 of Definition~\ref{def:cohzp-true-forcing-precondition}, $d \Vvdash_\mu (\forall x)\psi(G, x)$.
\end{itemize}
\end{proof}

\begin{lemma}\label{lem:cohzp-true-forcing-dense-level1}
For every $\Sigma^0_1$ ($\Pi^0_1$) formula $\varphi$, the following set is dense in $\Pb$:
$$
\{c \in \Pb : (\forall \nu < \parts(c))(\forall \alpha < 2)
	[c \Vvdash^\alpha_\nu \varphi(G) \mbox{ or } c \Vvdash^\alpha_\nu \neg \varphi(G)] \}
$$
\end{lemma}
\begin{proof}
It suffices to prove the statement for the case where $\varphi$ is a $\Sigma^0_1$ formula,
as the case where $\varphi$ is a $\Pi^0_1$ formula is symmetric. Fix a condition $c = (\vec{F}, T, \Ccal)$
and let $I(c)$ be the set of pairs $(\nu, \alpha) \in \parts(T) \times 2$ such that $c \not \Vvdash^\alpha_\nu \varphi(G)$
and $c \not \Vvdash^\alpha_\nu \neg \varphi(G)$. If $I(c) = \emptyset$ we are done, so suppose $I(c) \neq \emptyset$.
Fix some $(\alpha, \nu) \in I(c)$. We will construct an extension $d$ such that $I(d) \subseteq I(c) \setminus \{(\alpha, \nu)\}$.
Iterating the operation completes the proof.

The formula $\varphi$ is of the form $(\exists x)\psi(G, x)$ where $\psi \in \Sigma^0_0$.
Suppose there exists a $\sigma \in T$ such that $T^{[\sigma]}$ is infinite,
a $w < |\sigma|$ and a set $F' \subseteq \dom(T) \cap \set_\nu(\sigma) \cap R_\alpha$
such that $\psi(F_\nu^\alpha \cup F', w)$ holds.
In this case, letting $\vec{E}$ be defined by $E_\mu^\beta = F_\mu^\beta$ if $\mu \neq \nu \vee \beta \neq \alpha$
and $E_\nu^\alpha = F_\nu^\alpha \cup F'$, and letting~$S$ be the tree~$T^{[\sigma]}$
where the domain is restricted so that $(E_\nu^\alpha, \dom(S))$ Mathias extends $(F_\nu^\alpha, \dom(T))$,
by clause 1 of Definition~\ref{def:cohzp-true-forcing-precondition}, 
the condition $d = (\vec{E}, S, \Ccal[S])$ is a valid extension of $c$
such that $d^\alpha_\nu \Vvdash \varphi(G)$.

Suppose now that for every $\sigma \in T$ such that $T^{[\sigma]}$
is infinite, every $w < |\sigma|$ and every set $F' \subseteq \dom(T) \cap \set_\nu(\sigma) \cap R_\alpha$,
$\psi(F_\nu^\alpha \cup F', w)$ does not hold.
In this case, by clause 2 of Definition~\ref{def:cohzp-true-forcing-precondition},
$c \Vvdash^\alpha_\nu \neg \varphi(G)$.
\end{proof}

We now define the weak forcing relation which is almost the same 
as the true one, except that the set~$F'$ is not required to be a subset of~$R_\alpha$
and that $T^{[\sigma]}$ might be finite. Because of this, whenever a condition 
forces a $\Pi^0_1$ formula by the weak forcing relation, so does it by the strong forcing relation.

\begin{definition}[Weak forcing relation]\label{def:cohzp-forcing-precondition}
Fix a condition $c = (\vec{F}, T, \Ccal)$, a $\Sigma^0_0$ formula $\varphi(G, x)$,
a part $\nu < \parts(T)$ and a side $\alpha < 2$.
\begin{itemize}
	\item[1.] $c \Vdash^\alpha_\nu (\exists x)\varphi(G, x)$ iff there exists a $w \in \omega$ such that $\varphi(F^\alpha_\nu, w)$ holds.
	\item[2.] $c \Vdash^\alpha_\nu (\forall x)\varphi(G, x)$ iff 
	for every $\sigma \in T$, every $w < |\sigma|$ and every set $F' \subseteq \dom(T) \cap \set_\nu(\sigma)$,
	$\varphi(F^\alpha_\nu \cup F', w)$ holds.
\end{itemize}
\end{definition}

As one may expect, the weak forcing relation at the first level
is also closed under the refinement relation.

\begin{lemma}\label{lem:cohzp-forcing-extension-level1}
Fix a condition $c = (\vec{F}, T, \Ccal)$ and a $\Sigma^0_1$ ($\Pi^0_1$) formula $\varphi(G)$.
If $c \Vdash^\alpha_\nu \varphi(G)$ for some $\nu < \parts(T)$ and $\alpha < 2$,
then for every $d = (\vec{E}, S, \Dcal) \leq c$ and
every part $\mu$ of $S$ refining part $\nu$ of $T$, $d \Vdash^\alpha_\mu \varphi(G)$.
\end{lemma}
\begin{proof}\ 
\begin{itemize}
	\item If $\varphi \in \Sigma^0_1$ then this is exactly clause 1 of Lemma~\ref{lem:cohzp-true-forcing-extension-level1}
	since the definition of the weak and the true forcing relations coincide for~$\Sigma^0_1$ formulas.

	\item If $\varphi \in \Pi^0_1$ then it can be expressed as $(\forall x)\psi(G, x)$
	where $\psi \in \Sigma^0_0$. Fix a $\tau \in S$, a $w < |\tau|$ and a set
	$F' \subseteq \dom(S) \cap \set_\mu(\tau)$. By property~(ii) of the definition of an extension,
	there exists a $\sigma \in T^{[\nu, E^\alpha_\mu]}$
	such that $|\sigma| = |\tau|$ and $\set_\mu(\tau) \subseteq \set_\nu(\sigma)$. As $w < |\tau|$
	and $\dom(S) \subseteq \dom(T)$, $F' \subseteq \dom(T) \cap \set_\nu(\sigma)$.
	As $\sigma \in T^{[\nu, E^\alpha_\mu]}$, $E^\alpha_\mu \subseteq \set_\nu(\sigma)$
	and by property~(i) of the definition of an extension, $E^\alpha_\mu \subseteq \dom(T)$.
	Therefore $E^\alpha_\mu \cup F' \subseteq \dom(T) \cap \set_\nu(\sigma)$.
	By clause~2 of Definition~\ref{def:cohzp-forcing-precondition} applied to $c \Vdash^\alpha_\nu (\forall x)\psi(G, x)$, 
	$\psi(F^\alpha_\nu \cup (E^\alpha_\mu \setminus F^\alpha_\nu) \cup F', w)$ holds, hence $\psi(E^\alpha_\mu \cup F', w)$ holds
	and still by clause~2 of Definition~\ref{def:cohzp-forcing-precondition}, $d \Vdash^\alpha_\mu (\forall x)\psi(G, x)$.
\end{itemize}
\end{proof}

The following trivial lemma simply reflects the fact that the promise $\Ccal$ is not part 
of the definition of the weak forcing relation
for $\Sigma^0_1$ or~$\Pi^0_1$ formulas, and therefore has no effect on it.

\begin{lemma}\label{lem:cohzp-promise-no-effect-first-level}
Fix two conditions $c = (\vec{F}, T, \Ccal)$ and $d = (\vec{E}, T, \Dcal)$
and a $\Sigma^0_1$ ($\Pi^0_1$) formula. For every part $\nu$ of $T$ such that $F^\alpha_\nu = E^\alpha_\nu$,
$c \Vdash^\alpha_\nu \varphi(G)$ if and only if $d \Vdash^\alpha_\nu \varphi(G)$.
\end{lemma}
\begin{proof}
If $\varphi \in \Sigma^0_1$ then $\varphi(G)$ can be expressed as $(\exists x)\psi(G, x)$ 
where $\psi \in \Sigma^0_0$.
By clause~1 of Definition~\ref{def:cohzp-forcing-precondition},
$c \Vdash^\alpha_\nu \varphi(G)$ iff
there exists a $w \in \omega$ such that $\psi(F^\alpha_\nu, w)$ holds.
As $F^\alpha_\nu = E^\alpha_\nu$, $c \Vdash^\alpha_\nu \varphi(G)$
iff $d \Vdash^\alpha_\nu \varphi(G)$.
Similarily, if $\varphi \in \Pi^0_1$ then $\varphi(G)$ can be expressed as $(\forall x)\psi(G, x)$
where $\psi \in \Sigma^0_0$.
By clause~2 of Definition~\ref{def:cohzp-forcing-precondition},
$c \Vdash^\alpha_\nu \varphi(G)$ iff
for every $\sigma \in T$, every $w < |\sigma|$ and 
every set $F' \subseteq \dom(T) \cap \set_\nu(\sigma)$, $\psi(F^\alpha_\nu \cup F', w)$ holds.
As $F^\alpha_\nu = E^\alpha_\nu$, $c \Vdash^\alpha_\nu \varphi(G)$
iff $d \Vdash^\alpha_\nu \varphi(G)$.
\end{proof}

We can now define the forcing relation over higher formulas.
It is defined inductively, starting with $\Sigma^0_1$ and~$\Pi^0_1$ formulas.
We extend the weak forcing relation instead of the true one for effectiveness purposes.
We shall see later that the weak forcing relation behaves like the true one
for some parts and some sides of a condition, and therefore that 
it tells us something about the truth of the formula over some carefully defined generic real~$G$.
Note that the forcing relation over higher formulas is still parameterized by the side~$\alpha$ of
the condition.

\begin{definition}\label{def:cohzp-forcing-condition}
Fix a condition $c = (\vec{F}, T, \Ccal)$, a side $\alpha < 2$ and an arithmetic formula $\varphi(G)$.
\begin{itemize}
	\item[1.] If $\varphi(G) = (\exists x)\psi(G, x)$ where $\psi \in \Pi^0_1$ then
	$c \Vdash^\alpha \varphi(G)$ iff for every part $\nu$ of $T$ such that
	$(\nu, T) \in \Ccal$ there exists a $w < \dom(T)$ such that $c \Vdash^\alpha_\nu \psi(G, w)$
	\item[2.] If $\varphi(G) = (\forall x)\psi(G, x)$ where $\psi \in \Sigma^0_1$ then
	$c \Vdash^\alpha \varphi(G)$ iff for every infinite p.r.\ $k'$-partition tree $S$,
	every function $f : \parts(S) \to \parts(T)$,
	every $w$ and $\vec{E}$ smaller than $\#S$ such that the followings hold
	\begin{itemize}
		\item[i)] $E^\beta_\nu = F^\beta_{f(\nu)}$ for each $\nu < \parts(S)$ and $\beta \neq \alpha$
		\item[ii)] $(E^\alpha_\nu, \dom(S) \cap R_\alpha)$ Mathias extends 
			$(F^\alpha_{f(\nu)}, \dom(T) \cap R_\alpha)$ for each $\nu < \parts(S)$
		\item[iii)] $S$ $f$-refines $\bigcap_{\nu < \parts(S)} T^{[f(\nu), E^\alpha_\nu]}$
	\end{itemize}
	for every $(\mu, S) \in \Ccal$, $(\vec{E}, S, \Ccal[S]) \not \Vdash^\alpha_\mu \neg \psi(G, w)$
	\item[3.] If $\varphi(G) = (\exists x)\psi(G, x)$ where $\psi \in \Pi^0_{n+2}$ then
	$c \Vdash^\alpha \varphi(G)$ iff there exists a $w \in \omega$ such that $c \Vdash^\alpha \psi(G, w)$
	\item[4.] If $\varphi(G) = \neg \psi(G)$ where $\psi \in \Sigma^0_{n+3}$
	then $c \Vdash^\alpha \varphi(G)$ iff $d \not \Vdash^\alpha \psi(G)$ for every $d \in \Ext(c)$.
\end{itemize}
\end{definition}

Note that clause 2.ii) of Definition~\ref{def:cohzp-forcing-condition} seems 
to be~$\Pi^0_2$ since~$R_\alpha$ is $\Delta^0_2$. However, in fact, one just needs to ensure
that~$\dom(S) \subseteq \dom(T)$ and~$E^\alpha_\nu \setminus F^\alpha_{f(\nu)} \subseteq \dom(T) \cap R_\alpha$.
This is a $\Delta^0_2$ predicate, and so is its negation, so one can already easily check that
the forcing relation over a $\Pi^0_2$ formula will be also $\Pi^0_2$.
Before proving the usual properties about the forcing relation, we need to discuss the role of the sides
in the forcing relation.
We are now ready to prove that the forcing relation is closed under extension.

\begin{lemma}\label{lem:cohzp-forcing-extension}
Fix a condition $c$, a side $\alpha < 2$ and a $\Sigma^0_{n+2}$ ($\Pi^0_{n+2}$) formula $\varphi(G)$.
If $c \Vdash^\alpha \varphi(G)$ then for every $d \leq c$, $d \Vdash^\alpha \varphi(G)$.
\end{lemma}
\begin{proof}
We prove the statement by induction over the complexity of the formula $\varphi(G)$.
Fix a condition $c = (\vec{F}, T, \Ccal)$ and a side $\alpha < 2$ such that $c \Vdash^\alpha \varphi(G)$.
Fix an extension $d = (\vec{E}, S, \Dcal)$ of $c$.
\begin{itemize}
	\item If $\varphi \in \Sigma^0_2$ then $\varphi(G)$ can be expressed as $(\exists x)\psi(G, x)$ 
  where $\psi \in \Pi^0_1$. By clause~1 of Definition~\ref{def:cohzp-forcing-condition},
	for every part $\nu$ of $T$ such that $(\nu, T) \in \Ccal$, there exists a $w < \dom(T)$
	such that $c \Vdash^\alpha_\nu \psi(G, w)$. Fix a part $\mu$ of $S$ such that $(\mu, S) \in \Dcal$.
	As $\Dcal \subseteq \Ccal$, $(\mu, S) \in \Ccal$. By upward-closure of $\Ccal$,
	part $\mu$ of $S$ refines some part $\nu$ of $\Ccal$ such that $(\nu, T) \in \Ccal$.
	Therefore by Lemma~\ref{lem:cohzp-forcing-extension-level1}, $d \Vdash^\alpha_\mu \psi(G, w)$,
	with $w < \dom(T) \leq \dom(S)$. Applying again clause~1 of Definition~\ref{def:cohzp-forcing-condition},
	we deduce that $d \Vdash (\forall x)\psi(G, x)$, hence $d \Vdash^\alpha \varphi(G)$.

  \item If $\varphi \in \Pi^0_2$ then $\varphi(G)$ can be expressed as $(\forall x)\psi(G, x)$ where $\psi \in \Sigma^0_1$.
	Suppose by way of contradiction that $d \not \Vdash^\alpha (\forall x)\psi(G, x)$.
	Let $f : \parts(S) \to \parts(T)$ witness the refinement $S \leq T$.
	By clause~2 of Definition~\ref{def:cohzp-forcing-condition},
	there exist an infinite p.r.\ $k'$-partition tree $S'$, a function~$g : \parts(S') \to \parts(S)$, 
	a $w \in \omega$, and a $2k'$-tuple of finite sets $\vec{H}$ smaller than the code of $S'$ such that
	\begin{itemize}
		\item[i)] $H^\beta_\nu = E^\beta_{g(\nu)}$ for each $\nu < \parts(S')$ and $\beta \neq \alpha$
		\item[ii)] $(H^\alpha_\nu, \dom(S') \cap R_\alpha)$ Mathias extends $(E^\alpha_{g(\nu)}, \dom(S) \cap R_\alpha)$ 
		for each $\nu < \parts(S')$
		\item[iii)] $S'$ $g$-refines $\bigcap_{\nu < \parts(S')} S^{[g(\nu), H^\alpha_\nu]}$
		\item[iv)] there exists a $(\mu, S') \in \Dcal$ such that 
		$(\vec{H}, S', \Dcal[S']) \Vdash^\alpha_\mu \neg \psi(G, w)$.
	\end{itemize}
	To deduce by clause~2 of Definition~\ref{def:cohzp-forcing-condition} that
	$c \not \Vdash^\alpha (\forall x)\psi(G, x)$ and derive a contradiction, it suffices to prove
	that the same properties hold with respect to~$T$. Let $\vec{H}'$ be defined by
	$H^{'\beta}_\nu = F^\beta_{f(g(\nu))}$ for each $\nu < \parts(S')$ and $\beta \neq \alpha$
	and $H^{'\alpha}_\nu = H^\alpha_\nu$.
	\begin{itemize}
		\item[i)] It trivially holds by choice of $\vec{H}'$.
		\item[ii)] By property (i) of the definition of an extension,
		$(E^\alpha_{g(\nu)}, \dom(S))$ Mathias extends $(F^\alpha_{f(g(\nu))}, \dom(T))$.
		Moreover $(H^\alpha_\nu, \dom(S')$ Mathias extends $(E^\alpha_{g(\nu)}, \dom(S))$,
		so $(H^{'\alpha}_\nu, \dom(S')) = (H^\alpha_\nu, \dom(S'))$ Mathias extends $(F_{f(g(\nu))}, \dom(T))$.
		\item[iii)] As by property (ii) of the definition of an extension,\\
		$S$ $f$-refines $\bigcap_{\nu < \parts(S')} T^{[f(g(\nu)), E^\alpha_{g(\nu)}]}$, and \\
		$S'$ $g$-refines $\bigcap_{\nu < \parts(S')} S^{[g(\nu), H^\alpha_\nu]}$ then \\
		$S'$ $(g \circ f)$-refines $\bigcap_{\nu < \parts(S')} T^{[g(\nu), H^{'\alpha}_\nu]}$.
		\item[iv)] As $\Dcal \subseteq \Ccal$, there exists a part $(\mu, S') \in \Ccal$
		such that $(\vec{H}, S', \Dcal[S']) \Vdash^\alpha_\mu \neg \psi(G, w)$. 
		By Lemma~\ref{lem:cohzp-promise-no-effect-first-level}, $(\vec{H}', S', \Ccal[S']) \Vdash^\alpha_\mu \neg \psi(G, w)$.
	\end{itemize}

	\item If $\varphi \in \Sigma^0_{n+3}$ then $\varphi(G)$ can be expressed as $(\exists x)\psi(G, x)$ 
  where $\psi \in \Pi^0_{n+2}$.
  By clause~3 of Definition~\ref{def:cohzp-forcing-condition}, there exists a $w \in \omega$
  such that $c \Vdash^\alpha \psi(G, w)$. By induction hypothesis, $d \Vdash^\alpha \psi(G, w)$
  so by clause~3 of Definition~\ref{def:cohzp-forcing-condition}, $d \Vdash^\alpha \varphi(G)$.

  \item If $\varphi \in \Pi^0_{n+3}$ then $\varphi(G)$ can be expressed as $\neg \psi(G)$ where $\psi \in \Sigma^0_{n+3}$.
	Suppose by way of contradiction that $d \not \Vdash^\alpha \varphi(G)$.
	By clause~4 of Definition~\ref{def:cohzp-forcing-condition}, 
	there exists an $e \in \Ext(d)$ such that $e \Vdash^\alpha \psi(G)$.
	In particular, $e \in \Ext(c) $, so  by clause~4 of Definition~\ref{def:cohzp-forcing-condition}, $e \not \Vdash^\alpha \psi(G)$
	since $c \Vdash^\alpha \varphi(G)$.
	Contradiction.
\end{itemize}
\end{proof}

Although the weak forcing relation does not satisfy the density property,
the forcing relation over higher formulas does. The reason is that
the extended forcing relation does not involve the weak forcing relation
over~$\Sigma^0_1$ formulas in the clause 2 of Definition~\ref{def:cohzp-forcing-condition}, 
but uses instead the weaker statement ``$c$ does not force the negation of the~$\Sigma^0_1$ formula''.
The link between this statement and the statement ``$c$ has an extension which forces the $\Sigma^0_1$ formula''
is used when proving that $\varphi(G)$ holds iff $c \Vdash \varphi(G)$ for some condition
belonging to a sufficiently generic filter. We now prove the density of the forcing relation for higher formulas.

\begin{lemma}\label{lem:cohzp-forcing-dense}
For every $\Sigma^0_{n+2}$ ($\Pi^0_{n+2}$) formula $\varphi$, 
the following set is dense in $\Pb$:
$$
\{c \in \Pb : (\forall \alpha < 2)[c \Vdash^\alpha \varphi(G) \mbox{ or } c \Vdash^\alpha \neg \varphi(G)] \}
$$
\end{lemma}
\begin{proof}
We prove the statement by induction over~$n$.
It suffices to treat the case where $\varphi$ is a $\Sigma^0_{n+2}$ formula,
as the case where $\varphi$ is a $\Pi^0_{n+2}$ formula is symmetric. 
Moreover, it is enough to prove that for every condition $c$ and every $\alpha < 2$,
there exists an extension $d \leq c$ such that
$d \Vdash^\alpha \varphi(G) \mbox{ or } d \Vdash^\alpha \neg \varphi(G)$.
Iterating the process at most twice completes the proof.
Fix a condition $c = (\vec{F}, T, \Ccal)$ and a part $\alpha < 2$.
\begin{itemize}
	\item In case $n = 0$, the formula $\varphi$ is of the form $(\exists x)\psi(G, x)$ where
	$\psi \in \Pi^0_1$.
	Suppose there exists an infinite p.r.\ $k'$-partition tree $S$
	for some $k' \in \omega$, a function~$f : \parts(S) \to \parts(T)$, 
	and a $2k'$-tuple of finite sets $\vec{E}$ such that
	\begin{itemize}
		\item[i)] $E^\beta_\nu = F^\beta_{f(\nu)}$ for each $\nu < \parts(S)$ and $\beta \neq \alpha$
		\item[ii)] $(E^\alpha_\nu, \dom(S) \cap R_\alpha)$ Mathias extends 
		$(F^\alpha_{f(\nu)}, \dom(T) \cap R_\alpha)$ for each $\nu < \parts(S)$.
		\item[iii)] $S$ $f$-refines $\bigcap_{\nu < \parts(S)} T^{[f(\nu), E^\alpha_\nu]}$
		\item[iv)] for each non-empty part $\nu$ of $S$ such that $(\nu, S) \in \Ccal$, 
		$(\vec{E}, S, \Ccal[S]) \Vdash^\alpha_\nu \psi(G, w)$ for some~$w < \#S$
	\end{itemize}
	Let $\Dcal = \Ccal[S] \setminus \{(\nu, S') \in \Ccal : \mbox{ part } \nu \mbox{ of } S' \mbox{ is empty} \}$. 
	As $\Ccal$ is an $\emptyset'$-p.r.\ promise for $T$,
	$\Ccal[S]$ is an $\emptyset'$-p.r.\ promise for $S$.
	As $\Dcal$ is obtained from $\Ccal[S]$ by removing only empty parts, $\Dcal$ is also an $\emptyset'$-p.r.\ promise for $S$.
	By clause~1 of Definition~\ref{def:cohzp-forcing-condition},
	$d = (\vec{E}, S, \Dcal) \Vdash^\alpha (\exists x)\psi(G, x)$ hence $d \Vdash^\alpha \varphi(G)$.

	We may choose a coding of the p.r.\ trees such that
	the code of $S$ is sufficiently large to witness $f$, $\ell$ and $\vec{E}$.
	So suppose now that for every infinite p.r.\ $k'$-partition tree $S$,
	every function~$f : \parts(S) \to \parts(T)$, $\ell \in \omega$ and $\vec{E}$ 
	smaller than the code of $S$ such that properties i-iii) hold,
	there exists a non-empty part $\nu$ of $S$ such that $(\nu, S) \in \Ccal$
	and $(\vec{E}, S, \Ccal[S]) \not \Vdash^\alpha_\nu \psi(G, w)$ for every $w < \ell$.
	Let $\Dcal$ be the collection of all such $(\nu, S)$. $\Dcal$ is $\emptyset'$-p.r.
	By Lemma~\ref{lem:cohzp-forcing-extension-level1} and since we require that~$\#S \geq \#T$
	in the definition of~$S \leq T$, $\Dcal$ is upward-closed, hence is a promise for~$T$. By clause~2
	of Definition~\ref{def:cohzp-forcing-condition}, $d = (\vec{F}, T, \Dcal) \Vdash^\alpha (\forall x) \neg \psi(G, x)$,
	hence $d \Vdash^\alpha \neg \varphi(G)$.

	\item In case $n > 0$, density follows from clause~4 of Definition~\ref{def:cohzp-forcing-condition}.
\end{itemize}
\end{proof}

We now prove that the weak forcing relation extended to any arithmetic formula
enjoys the desired definability properties. For this, we start with a lemma
showing that the extension relation is $\Pi^0_2$. Therefore, only
the first two levels have to be treated independently, since 
the extension relation does not add some extra complexity to the forcing relation for higher formulas.

\begin{lemma}\label{lem:cohzp-extension-complexity}
For every condition $c$, $\Ext(c)$ is $\Pi^0_2$ uniformly in~$c$.
\end{lemma}
\begin{proof}
Recall from Lemma~\ref{lem:em-refinement-complexity} that
given $k, t \in \omega$, the set $PartTree(k, t)$ denotes 
the $\Pi^0_1$ set of all the infinite p.r.\ $k$-partition trees of $[t, \infty)$,
and given a $k$-partition tree $S$ and a part $\nu$ of $S$,
the predicate $Empty(S, \nu)$ denotes the $\Pi^0_1$ formula ``part $\nu$ of $S$ is empty'',
that is, the formula $(\forall \sigma \in S)[\set_\nu(\sigma) \cap \dom(S) = \emptyset$].
If $T$ is p.r.\ then so is $T^{[\nu, H]}$ for some finite set $H$.

Fix a condition $c = (\vec{F}, (k, t, T), \Ccal)$.
$(\vec{H}, (k', t', S), \Dcal) \in \Ext(c)$ iff the following formula holds:
$$
\begin{array}{l@{\hskip 0.5in}r}
(\exists f : k' \to k)\\
(\forall \nu < k')(\forall \alpha < 2)
(H^\alpha_\nu, [t', \infty) \cap R_\alpha) \mbox{ Mathias extends } (F^\alpha_{f(\nu)}, [t, \infty) \cap R_\alpha) & (\Pi^0_2)\\
\wedge S \in PartTree(k', t') \wedge S \leq_f \bigwedge_{\nu < k', \alpha < 2} T^{[f(\nu), H^\alpha_{\nu}]} & (\Pi^0_1)\\
\wedge \Dcal \mbox{ is a promise for } S \wedge \Dcal \subseteq \Ccal & (\Pi^0_2)\\
\end{array}
$$
The formula $(H^\alpha_\nu, [t', \infty) \cap R_\alpha) \mbox{ Mathias extends } 
(F^\alpha_{f(\nu)}, [t, \infty) \cap R_\alpha)$ can be written
$(\forall x < t)[x \in H^\alpha_\nu \biimp x \in F^\alpha_{f(\nu)}]
\wedge t' \geq t \wedge (\forall x \in H^\alpha_\nu \setminus F^\alpha_{f(\nu)}) 
x \in R_\alpha$ and therefore is $\Pi^0_2$.
By Lemma~\ref{lem:em-refinement-complexity}
and the fact that $\bigwedge_{\nu < k', \alpha < 2} T^{[f(\nu), H^\alpha_{\nu}]}$
is p.r.\ uniformly in $T$, $f$, $\vec{H}$ and $k'$,
the above formula is~$\Pi^0_2$.
\end{proof}

\begin{lemma}\label{lem:cohzp-complexity-forcing}
Fix an arithmetic formula $\varphi(G)$, a condition $c = (\vec{F}, T, \Ccal)$,
a side $\alpha < 2$ and a part $\nu$ of $T$.
\begin{itemize}
	\item[a)] If $\varphi(G)$ is a $\Sigma^0_1$ ($\Pi^0_1$) formula 
	then so is the predicate $c \Vdash^\alpha_\nu \varphi(G)$.
	\item[b)] If $\varphi(G)$ is a $\Sigma^0_{n+2}$ ($\Pi^0_{n+2}$) formula 
	then so is the predicate $c \Vdash^\alpha \varphi(G)$.
\end{itemize}
\end{lemma}
\begin{proof}
We prove our lemma by induction over the complexity of the formula $\varphi(G)$.
\begin{itemize}
	\item If $\varphi(G) \in \Sigma^0_1$ then it can be expressed as $(\exists x)\psi(G, x)$ where $\psi \in \Sigma^0_0$.
	By clause~1 of Definition~\ref{def:cohzp-forcing-precondition}, $c \Vdash^\alpha_\nu \varphi(G)$ if and only if 
	the formula $(\exists w \in \omega)\psi(F^\alpha_\nu, w)$ holds. This is a $\Sigma^0_1$ predicate.
	
	\item If $\varphi(G) \in \Pi^0_1$ then it can be expressed as $(\forall x)\psi(G, x)$ where $\psi \in \Sigma^0_0$.
	By clause~2 of Definition~\ref{def:cohzp-forcing-precondition}, $c \Vdash^\alpha_\nu \varphi(G)$ if and only if 
	the formula $(\forall \sigma \in T)(\forall w < |\sigma|)(\forall F' \subseteq \dom(T) \cap \set_\nu(\sigma))
	\psi(F^\alpha_\nu \cup F', w)$ holds. This is a $\Pi^0_1$ predicate.

	\item If $\varphi(G) \in \Sigma^0_2$ then it can be expressed as $(\exists x)\psi(G, x)$ where $\psi \in \Pi^0_1$.
	By clause~1 of Definition~\ref{def:cohzp-forcing-condition}, $c \Vdash^\alpha \varphi(G)$ if and only if 
	the formula $(\forall \nu < \parts(T)(\exists w < \dom(T))[(\nu, T) \in \Ccal \imp c \Vdash^\alpha_\nu \psi(G, w)]$ holds.
	This is a $\Sigma^0_2$ predicate by induction hypothesis and the fact that $\Ccal$ is $\emptyset'$-computable.

	\item If $\varphi(G) \in \Pi^0_2$ then it can be expressed as $(\forall x)\psi(G, x)$ where $\psi \in \Sigma^0_1$.
	By clause~2 of Definition~\ref{def:cohzp-forcing-condition}, $c \Vdash \varphi(G)$ if and only if 
	for every infinite $k'$-partition tree $S$, every function $f : \parts(S) \to \parts(T)$, 
	every $w$ and $\vec{E}$ smaller than the code of $S$ such that the followings hold
	\begin{itemize}
		\item[i)] $(E_\nu, \dom(S) \cap R_\alpha)$ Mathias extends $(F_{f(\nu)}, \dom(T) \cap R_\alpha)$ for each $\nu < \parts(S)$
		\item[ii)] $S$ $f$-refines $\bigcap_{\nu < \parts(S)} T^{[f(\nu), E_\nu]}$
	\end{itemize}
	for every $(\mu, S) \in \Ccal$, $(\vec{E}, S, \Ccal[S]) \not \Vdash^\alpha_\mu \neg \psi(G, w)$.
	By Lemma~\ref{lem:em-refinement-complexity}, Properties i-ii) are $\Delta^0_2$.
	Moreover the predicate $(\mu, S) \in \Ccal$ is $\Delta^0_2$ since~$\Ccal$ is $\emptyset'$-p.r.
	By induction hypothesis, $(\vec{E}, S, \Ccal[S]) \not \Vdash^\alpha_\mu \neg \psi(G, w)$ is $\Sigma^0_1$.
	Therefore $c \Vdash^\alpha \varphi(G)$ is a $\Pi^0_2$ predicate.

	\item If $\varphi(G) \in \Sigma^0_{n+3}$ then it can be expressed as $(\exists x)\psi(G, x)$ where $\psi \in \Pi^0_{n+2}$.
	By clause~3 of Definition~\ref{def:cohzp-forcing-condition}, $c \Vdash^\alpha \varphi(G)$ if and only if 
	the formula $(\exists w \in \omega)c \Vdash^\alpha \psi(G, w)$ holds. This is a $\Sigma^0_{n+3}$ predicate
	by induction hypothesis.

	\item If $\varphi(G) \in \Pi^0_{n+3}$ then it can be expressed as $\neg \psi(G)$ where $\psi \in \Sigma^0_{n+3}$. 
	By clause~4 of Definition~\ref{def:cohzp-forcing-condition}, $c \Vdash^\alpha \varphi(G)$ if and only if 
	the formula $(\forall d)(d \not \in \Ext(c) \vee d \not \Vdash^\alpha \psi(G))$ holds.
	By induction hypothesis, $d \not \Vdash^\alpha \psi(G)$ is a $\Pi^0_{n+3}$ predicate.
	By Lemma~\ref{lem:cohzp-extension-complexity}, the set $\Ext(c)$ is $\Pi^0_2$-computable uniformly in $c$,
	thus $c \Vdash^\alpha \varphi(G)$ is a $\Pi^0_{n+3}$ predicate.
\end{itemize}
\end{proof}

\subsection{Validity}

As we already saw, we have two candidate forcing relations for $\Sigma^0_1$ and $\Pi^0_1$ formulas:
\begin{itemize}
	\item[1.] The ``true'' forcing relation $c \Vvdash^\alpha \varphi(G)$.
	This relation has been shown to have the expected density properties through Lemma~\ref{lem:cohzp-true-forcing-dense-level1}.
	However deciding such a relation requires too much computational power.
	\item[2.] The ``weak'' forcing relation $c \Vdash^\alpha \varphi(G)$.
	Deciding such a relation requires the same definitional power as the formula it forces.
	It provides a sufficient condition for forcing the formula $\varphi(G)$
	as $c \Vdash^\alpha \varphi(G)$ implies $c \Vvdash^\alpha \varphi(G)$,
	but the converse does not hold and we cannot prove the density property in the general case.
\end{itemize}

Thankfully, there exist some sides and parts of any condition
on which those two forcing relations coincide. This leads to the notion of validity.

\begin{definition}[Validity]
Fix an enumeration $\varphi_0(G), \varphi_1(G), \dots$ of all $\Pi^0_1$ formulas.
Fix a condition $c = (\vec{F}, T, \Ccal)$, a side $\alpha < 2$, and a part $\nu$ of $T$.
We say that \emph{side $\alpha$ is $n$-valid in part $\nu$ of $T$} for some~$n \in \omega$
if part $\nu$ of $T$ is $R_\alpha$-acceptable and for every~$i < n$,
$c \Vvdash^\alpha_\nu \varphi_i(G)$ iff $c \Vdash^\alpha_\nu \varphi_i(G)$.
\end{definition}

The following lemma shows that given some~$n \in \omega$,
we can restrict $\Ccal$ so that it ``witnesses its $n$-valid parts''.

\begin{lemma}\label{lem:cohzp-validity-exists}
For every $n \in \omega$, the following set is dense in~$\Pb$:
$$
\{ (\vec{F}, T, \Ccal) \in \Pb : (\forall \nu)(\exists \alpha < 2)
	[(\nu, T) \in \Ccal \imp \mbox{side } \alpha \mbox{ is } n\mbox{-valid in part } \nu \mbox{ of } T] \} 
$$
\end{lemma}
\begin{proof}
Given a condition $c = (\vec{F}, T, \Ccal)$,
let $I(c)$ be the set of the parts $\nu$ of $T$ such that
$(\nu, T) \in \Ccal$ and no $\alpha < 2$ is valid for $\varphi$ in part $\nu$ of $T$.
Fix a condition $c = (\vec{F}, T, \Ccal) \in \Pb$.
By iterating Lemma~\ref{lem:cohzp-forcing-dense}, we can assume without loss of generality that
for each~$i < n$,
$$
(\forall \alpha < 2)[c \Vdash^\alpha (\exists x)\varphi_i(G) \mbox{ or } c \Vdash^\alpha (\forall x)\neg \varphi_i(G)]
$$
The dummy variable~$x$ ensures that the forcing relation for~$\Sigma^0_2$ and~$\Pi^0_2$ is applied.
It suffices to prove that for every $\nu \in I(c)$,
there exists an extension $d = (\vec{E}, S, \Dcal)$
such that $I(d) \subseteq I(c) \setminus \{\nu\}$. Iterating the process at most $|\parts(T)|$
times completes the proof.

Fix a part $\nu \in I(c)$ and let $\Dcal$ be the set of $(\mu, S) \in \Ccal$
such that part $\mu$ of $S$ does not refine part $\nu$ of $T$.
The set $\Dcal$ is a $\emptyset'$-p.r.\ upward-closed subset of $\Ccal$.
It suffices to prove that for every infinite p.r.\ partition tree $S \leq T$,
there exists a non-empty part $\mu$ of $S$ such that $(\mu, S) \in \Dcal$
to deduce that $\Dcal$ is a promise for $T$ and obtain an extension
$d = (\vec{E}, T, \Dcal)$ of $c$ such that $I(d) \subseteq I(c) \setminus \{\nu\}$.

Fix an infinite p.r.\ partition tree $S \leq_g T$ for some~$g$ and let~$\mu$ be a part of~$S$
$g$-refining part $\nu$ of $T$.
By choice of $\nu$, for every $\alpha < 2$, either $\mu$ is not $R_\alpha$-acceptable in $S$,
or $c \Vvdash^\alpha_\nu \varphi_{i_\alpha}(G)$ but $c \not \Vdash^\alpha_\nu \varphi_{i_\alpha}(G)$
for some~$i_\alpha < n$.
In the latter case, by choice of $c$, $c \Vdash^\alpha_\nu (\forall x)\neg \varphi_{i_\alpha}(G)$.

We now assume that $S$ has $k$ parts, among which $m$ parts $g$-refine~$\nu$.
Let~$f : k+m \to k$ be the function
such that~$f(\mu) = \mu$ for each part~$\mu$ of~$S$ not $g$-refining part~$\nu$ of~$T$,
and such that $f(\mu_\alpha) = \mu$ for each part~$\mu$ of~$S$ $g$-refining part~$\nu$ of~$T$
and each $\alpha < 2$. In other words, $f$ forks each part~$\mu$ of~$S$ $g$-refining the part~$\mu$ of~$T$
into $2$ parts~$\mu_0$ and~$\mu_1$.
Let~$P$ be a path through~$S$, and let~$t \in \omega$
be large enough to ``witness the non~$R_\alpha$-acceptable sides''. More formally,
let~$t$ be such that for every~$\alpha < 2$, either~$\set_\mu(P) \cap R_\alpha \cap [t, \infty) = \emptyset$
for each part~$\mu$ of~$S$ $g$-refining part~$\nu$ of~$T$,
or $c \Vvdash^\alpha_\nu \varphi_{i_\alpha}(G)$  but $c \not \Vdash^\alpha_\nu \varphi_{i_\alpha}(G)$.

Let~$S'$ be the p.r.\ tree of all the $\tau$'s $f$-refining some~$\sigma \in S$
and such that for each~$\alpha < 2$ and each part $\mu$ of $S$ $g$-refining part~$\nu$ of~$T$, 
either~$\set_{\mu_\alpha}(\tau) \cap [t, \infty) = \emptyset$,
or $\varphi_{i_\alpha}(F^\alpha_\nu \cup F')$ holds for each~$F' \subseteq \dom(S) \cap \set_{\mu_\alpha}(\tau)$.
The tree~$S'$ is a $(k+m)$-partition tree of~$[t, \infty)$ $f$-refining~$S$. We claim that~$S'$ is infinite. 
Fix some~$s \in \omega$, we will prove that~$\tau \in S'$
for some string~$\tau$ of length~$s$. Let~$\sigma = P \uh s$.
In particular, $S^{[\sigma]}$ is infinite, so for every~$\alpha < 2$
and every part $\mu$ of~$S$ $g$-refining part~$\nu$ of~$T$,
either~$\set_\mu(P) \cap R_\alpha \cap [t, \infty) = \emptyset$ by definition of~$t$,
or, unfolding clause 2 of Definition~\ref{def:cohzp-true-forcing-precondition}
for $c \Vvdash^\alpha_\nu \varphi_{i_\alpha}(G)$ and since $\sigma$ $g$-refines some extendible node in~$T$, 
for every set~$F' \subseteq \dom(S) \cap \set_\mu(\sigma) \cap R_\alpha$,
$\varphi(F^\alpha_\nu \cup F')$ holds. Let~$\tau$ be the string refining~$\sigma$
such that $\set_{\mu_\alpha}(\tau) = \set_\mu(\sigma) \cap R_\alpha$ for each~$\alpha < 2$
and each part~$\mu$ of $S$ $g$-refining part~$\nu$ of~$T$.
By definition of~$S'$, $\tau \in S'$. Therefore~$S'$ is infinite.
Moreover, by definition of~$S$', for each~$\alpha < 2$,
either $\mu_\alpha$ is empty in $S'$ or~$(\vec{E}, S', \Ccal[S']) \Vdash^\alpha_{\mu_\alpha} \varphi_{i_\alpha}(G)$,
where $\vec{E}$ is obtained by duplicating the sets in~$\vec{F}$ according the forks of~$g \circ f$.

By definition of $c \Vdash^\alpha_\nu (\forall x)\neg \varphi_{i_\alpha}(G)$,
$(\vec{E}, S', \Ccal[S']) \not \Vdash^\alpha_{\mu_\alpha} \varphi_{i_\alpha}(G)$ for each~$\alpha < 2$ 
and each part $\mu$ of $S$ $g$-refining part $\nu$ of~$T$ such that $(\mu_\alpha, S') \in \Ccal$.
Then, for each~$\alpha < 2$, either~$\mu_\alpha$ is empty in~$S'$, or $(\mu_\alpha, S') \not \in \Ccal$, 
as otherwise it would contradict $(\vec{E}, S', \Ccal[S']) \Vdash^\alpha_{\mu_\alpha} \varphi(G)$.
So there must exists a non-empty part $\mu$ of $S'$ not refining part $\nu$ of $T$
such that $(\mu, S') \in \Ccal$, and by upward closure of a promise, there exists a non-empty part $\mu$
of $S$ not refining part $\nu$ of $T$ such that $(\mu, S) \in \Ccal$. By definition of $\Dcal$,
$(\mu, S) \in \Dcal$. Therefore $\Dcal$ is a promise for $T$ and we conclude.
\end{proof}

Given any filter~$\Fcal = \{c_0, c_1, \dots \}$ with $c_s = (\vec{F}_s, T_s, \Ccal_s)$
the set of pairs $(\alpha, \nu_s)$ such that $(\nu_s, T_s) \in \Ccal_s$ forms again
an infinite, directed acyclic graph~$\Gcal(\Fcal)$. By Lemma~\ref{lem:cohzp-validity-exists},
whenever~$\Fcal$ is sufficiently generic,
the graph~$\Gcal(\Fcal)$ yields a sequence of parts $P$ such that for every~$s$ if~$c_s$ refines $c_t$, then part~$P(s)$ of~$c_s$
refines part~$P(t)$ of~$c_t$, and such that for every $n$, there is some~$s$ and some side~$\alpha < 2$ such that the side~$\alpha$
is $n$-valid in part $P(s)$ of~$c_s$. The path~$P$ induces an infinite set~$G = \bigcup \{ F^\alpha_{P(s), s} : s \in \omega \}$. 
Since whenever $\alpha$ is $n$-valid in part~$P(s)$ of~$c_s$, then it is $m$-valid in part~$P(s)$ of~$c_s$ for every~$m < n$,
we can fix an $\alpha < 2$ such that for every $n$, there is some~$s$ such that the side~$\alpha$
is $n$-valid in part $P(s)$ of~$c_s$.
We call $\alpha$ the \emph{generic side}, $P$ the \emph{generic path} and $G$ the \emph{generic real}.

By choosing a generic path that goes through valid sides and parts of the conditions, 
we recovered the density property for the weak forcing relation
and can therefore prove that a property holds over the generic real
if and only if it can be forced by some condition belonging to the generic filter.

\begin{lemma}\label{lem:cohzp-generic-level1}
Suppose that $\Fcal$ is sufficiently generic and let $\alpha$,~$P$ and~$G$ be the generic side, the generic path 
and the generic real, respectively.
For every $\Sigma^0_1$ ($\Pi^0_1$) formula $\varphi(G)$,
$\varphi(G)$ holds iff $c_s \Vdash^\alpha_{P(s)} \varphi(G)$ for some $c_s \in \Fcal$.
\end{lemma}
\begin{proof}
Thanks to validity, it suffices to prove that if $c_s \Vdash^\alpha_\nu \varphi(G)$ for some~$c_s \in \Fcal$,
then $\varphi(G)$ holds. Indeed, if $\varphi(G)$ holds, then by genericity of $\Fcal$, $c_s \Vvdash^\alpha_{P(s)} \varphi(G)$
or $c_s \Vvdash^\alpha_{P(s)} \neg \varphi(G)$ for some $c_s \in \Fcal$.
By validity of side $\alpha$ in part $P(s)$ of $c_s$, $c_s \Vdash^\alpha_{P(s)} \varphi(G)$ or $c_s \Vdash^\alpha_{P(s)} \neg \varphi(G)$.
If $c_s \Vdash^\alpha_{P(s)} \neg \varphi(G)$ then $\neg \varphi(G)$ holds, contradicting 
the hypothesis. So $c_s \Vdash^\alpha_{P(s)} \varphi(G)$.
Fix a condition $c_s = (\vec{F}, T, \Ccal) \in \Fcal$ such that $c_s \Vdash^\alpha_\nu \varphi(G)$, where~$\nu = P(s)$.
\begin{itemize}
	\item If $\varphi \in \Sigma^0_1$ then $\varphi(G)$ can be expressed as $(\exists x)\psi(G, x)$ where $\psi \in \Sigma^0_0$.
	By clause~1 of Definition~\ref{def:cohzp-forcing-precondition}, there exists a $w \in \omega$ such that
	$\psi(F^\alpha_\nu, w)$ holds. As $\nu = P(s)$, $F^\alpha_\nu = F^\alpha_{P(s)} \subseteq G$
	and~$G \setminus F^\alpha_\nu \subseteq (\max F^\alpha_\nu, \infty)$, so $\psi(G, w)$ holds by continuity, hence $\varphi(G)$ holds.
	
	\item If $\varphi \in \Pi^0_1$ then $\varphi(G)$ can be expressed as $(\forall x)\psi(G, x)$ where $\psi \in \Sigma^0_0$.
	By clause~2 of Definition~\ref{def:cohzp-forcing-precondition}, for every $\sigma \in T$, every $w < |\sigma|$
	and every set $F' \subseteq \dom(T) \cap \set_\nu(\sigma)$, $\psi(F^\alpha_\nu \cup F', w)$ holds. 
	For every $F' \subseteq G \setminus F^\alpha_\nu$, and $w \in \omega$ there exists a $\sigma \in T$
	such that $w < |\sigma|$ and $F' \subseteq \dom(T) \cap \set_\nu(\sigma)$. Hence $\psi(F^\alpha_\nu \cup F', w)$ holds.
	Therefore, for every $w \in \omega$, $\psi(G, w)$ holds, so $\varphi(G)$ holds.
\end{itemize}
\end{proof}

\begin{lemma}
Suppose that $\Fcal$ is sufficiently generic and let $\alpha$ and~$G$ be 
the generic side and the generic real, respectively.
For every $\Sigma^0_{n+2}$ ($\Pi^0_{n+2}$) formula $\varphi(G)$,
$\varphi(G)$ holds iff $c_s \Vdash^\alpha \varphi(G)$ for some $c_s \in \Fcal$.
\end{lemma}
\begin{proof}
This lemma uses validity implicitly by calling Lemma~\ref{lem:cohzp-generic-level1}, where it was used explicitly.
Emulating the proof of Lemma~\ref{lem:em-generic-higher-levels}, it suffices to prove that if $c_s \Vdash^\alpha \varphi(G)$ for some
$c_s \in \Fcal$ then $\varphi(G)$ holds. Let~$P$ be the generic path induced by the generic filter~$\Fcal$.
Fix a condition $c_s = (\vec{F}, T, \Ccal) \in \Fcal$ such that $c_s \Vdash^\alpha \varphi(G)$. 
We proceed by case analysis on $\varphi$. 
\begin{itemize}
	\item If $\varphi \in \Sigma^0_2$ then $\varphi(G)$ can be expressed as $(\exists x)\psi(G, x)$ 
  where $\psi \in \Pi^0_1$.
  By clause~1 of Definition~\ref{def:cohzp-forcing-condition}, for every part $\nu$ of $T$
	such that $(\nu, T) \in \Ccal$, there exists a $w < \dom(T)$
  such that $c_s \Vdash^\alpha_\nu \psi(G, w)$. Since $(P(s), T) \in \Ccal$,
	$c_s \Vdash^\alpha_{P(s)} \psi(G, w)$. By Lemma~\ref{lem:cohzp-generic-level1}, $\psi(G, w)$ holds, hence $\varphi(G)$ holds.

	\item If $\varphi \in \Pi^0_2$ then $\varphi(G)$ can be expressed as $(\forall x)\psi(G, x)$ where $\psi \in \Sigma^0_1$.
	By clause~2 of Definition~\ref{def:cohzp-forcing-condition}, for every infinite $k'$-partition
	tree $S$, every function~$f : \parts(S) \to \parts(T)$, 
	every $w$ and $\vec{E}$ smaller than the code of $S$ such that the followings hold
	\begin{itemize}
		\item[i)] $(E_\nu, \dom(S) \cap R_\alpha)$ Mathias extends $(F_{f(\nu)}, \dom(T) \cap R_\alpha)$ for each $\nu < \parts(S)$
		\item[ii)] $S$ $f$-refines $\bigcap_{\nu < \parts(S)} T^{[f(\nu), E_\nu]}$
	\end{itemize}
	for every $(\mu, S) \in \Ccal$, $(\vec{E}, S, \Ccal[S]) \not \Vdash^\alpha_\mu \neg \psi(G, w)$.
	Suppose by way of contradiction that $\psi(G, w)$ does not hold for some $w \in \omega$.
  Then by Lemma~\ref{lem:cohzp-generic-level1}, there exists a $c_t \in \Fcal$
	such that $c_t \Vdash^\alpha_{P(t)} \neg \psi(G, w)$. Since~$\Fcal$ is a filter,
	there is a condition~$c_e = (\vec{E}, S, \Dcal) \in \Fcal$ extending~$c_s$ and~$c_t$.
	By choice of~$P$, $(P(e), S) \in \Ccal$, so by clause ii), $(\vec{E}, S, \Ccal[S]) \not \Vdash^\alpha_{P(e)} \psi(G, w)$,
	hence by Lemma~\ref{lem:cohzp-promise-no-effect-first-level}, $c_e \not \Vdash^\alpha_{P(e)} \psi(G, w)$.
	However, since part~$P(e)$ of~$c_e$ refines part~$P(t)$ of~$c_t$,
	then by Lemma~\ref{lem:cohzp-forcing-extension-level1}, $c_e \Vdash^\alpha_{P(e)} \psi(G, w)$.
	Contradiction. Hence, for every~$w \in \omega$, $\psi(G, w)$ holds, so~$\varphi(G)$ holds.

	\item If $\varphi \in \Sigma^0_{n+3}$ then $\varphi(G)$ can be expressed as $(\exists x)\psi(G, x)$ 
  where $\psi \in \Pi^0_{n+2}$.
  By clause~3 of Definition~\ref{def:cohzp-forcing-condition}, there exists a $w \in \omega$
  such that $c_s \Vdash^\alpha \psi(G, w)$. By induction hypothesis, $\psi(G, w)$ holds, hence $\varphi(G)$ holds.

  Conversely, if $\varphi(G)$ holds, then there exists a $w \in \omega$ such that $\psi(G, w)$ holds,
  so by induction hypothesis $c_s \Vdash^\alpha \psi(G, w)$ for some $c_s \in \Fcal$,
  so by clause~3 of Definition~\ref{def:cohzp-forcing-condition}, $c_s \Vdash^\alpha \varphi(G)$.

	\item If $\varphi \in \Pi^0_{n+3}$ then $\varphi(G)$ can be expressed as $\neg \psi(G)$ where $\psi \in \Sigma^0_{n+3}$.
  By clause~4 of Definition~\ref{def:cohzp-forcing-condition}, for every $d \in \Ext(c_s)$, $d \not \Vdash^\alpha \psi(G)$.
	By Lemma~\ref{lem:cohzp-forcing-extension}, $d \not \Vdash^\alpha \psi(G)$ for every~$d \in \Fcal$,
  and by a previous case, $\psi(G)$ does not hold, so $\varphi(G)$ holds.
\end{itemize}
\end{proof}

\subsection{Preserving definitions}

The following (and last) lemma shows that every sufficiently generic real
preserves higher definitions. This preservation property cannot be proved
in the case of non-$\Sigma^0_1$ sets since the weak forcing relation
does not have the good density property in general.

\begin{lemma}\label{lem:cohzp-diagonalization}
If $A \not \in \Sigma^0_{n+2}$ and $\varphi(G, x)$ is $\Sigma^0_{n+2}$,
then the set of $c \in \Pb$ satisfying the following property is dense:
$$
(\forall \alpha < 2)[(\exists w \in A)c \Vdash^\alpha \neg \varphi(G, w)] 
\vee [(\exists w \not \in A)c \Vdash^\alpha \varphi(G, w)]
$$
\end{lemma}
\begin{proof}
It is sufficient to find, given a condition $c$ and a side $\alpha < 2$,
an extension $d$ of~$c$ such that the following holds:
$$
[(\exists w \in A)c \Vdash^\alpha \neg \varphi(G, w)] 
\vee [(\exists w \not \in A)c \Vdash^\alpha \varphi(G, w)]
$$
Fix a condition $c = (\vec{F}, T, \Ccal)$ and a side $\alpha < 2$.
\begin{itemize}
	\item In case $n = 0$, $\varphi(G, w)$ can be expressed as $(\exists x)\psi(G, w, x)$ where $\psi \in \Pi^0_1$.
	Let $U$ be the set of integers $w$ such that there exists an infinite p.r.\ $k'$-partition tree $S$
	for some $k' \in \omega$, a function~$f : \parts(S) \to \parts(T)$ and a $2k'$-tuple of finite sets $\vec{E}$ such that
	\begin{itemize}
		\item[i)] $E^\beta_\nu = F^\beta_{f(\nu)}$ for each $\nu < \parts(S)$ and $\beta \neq \alpha$
		\item[ii)] $(E^\alpha_\nu, \dom(S) \cap R_\alpha)$ Mathias extends $(F^\alpha_{f(\nu)}, \dom(T))$ for each $\nu < \parts(S)$.
		\item[iii)] $S$ $f$-refines $\bigcap_{\nu < \parts(S)} T^{[f(\nu), E^\alpha_\nu]}$
		\item[iv)] for each non-empty part $\nu$ of $S$ such that $(\nu, S) \in \Ccal$, 
		$(\vec{E}, S, \Ccal[S]) \Vdash^\alpha_\nu \psi(G, w, u)$ for some~$u < \#S$
	\end{itemize}
	By Lemma~\ref{lem:cohzp-complexity-forcing} and Lemma~\ref{lem:em-refinement-complexity},
	$U \in \Sigma^0_2$, thus $U \neq A$. Let $w \in U \Delta A$.

	Suppose that $w \in U \setminus A$.
	Let $\Dcal = \Ccal[S] \setminus \{(\nu, S') \in \Ccal : \mbox{ part } \nu \mbox{ of } S' \mbox{ is empty} \}$. 
	As $\Ccal$ is an $\emptyset'$-p.r.\ promise for $T$,
	$\Ccal[S]$ is an $\emptyset'$-p.r.\ promise for $S$.
	As $\Dcal$ is obtained from $\Ccal[S]$ by removing only empty parts, $\Dcal$ is also an $\emptyset'$-p.r.\ promise for $S$.
	By Lemma~\ref{lem:cohzp-promise-no-effect-first-level}, for every part~$\nu$ of~$S$
	such that~$(\nu, S) \in \Dcal \subseteq \Ccal$, $(\vec{E}, S, \Dcal) \Vdash^\alpha_\nu \psi(G, w, u)$ for some~$u < \dom(S)$,
	hence by clause~1 of Definition~\ref{def:cohzp-forcing-condition},
	$d = (\vec{E}, S, \Dcal) \Vdash^\alpha (\exists x)\psi(G, w, x)$.
	In other words, $d \Vdash^\alpha \varphi(G)$ for some $w \not \in A$.

	We may choose a coding of the p.r.\ trees such that
	the code of $S$ is sufficiently large to witness $w$ and $\vec{E}$.
	So suppose now that $w \in A \setminus U$. Then for every infinite p.r.\ $k'$-partition tree $S$,
	every function~$f : \parts(S) \to \parts(T)$
	and every $\vec{E}$ smaller than the code of $S$ such that properties i-iii) hold,
	there exists a non-empty part $\nu$ of $S$ such that $(\nu, S) \in \Ccal$
	and $(\vec{E}, S, \Ccal[S]) \not \Vdash^\alpha_\nu \psi(G, w, u)$ for every $u < \#S$.
	Let $\Dcal$ be the collection of all such $(\nu, S)$. $\Dcal$ is $\emptyset'$-p.r.
	By Lemma~\ref{lem:cohzp-forcing-extension-level1} and since~$\#S \geq \#T$
	whenever~$S \leq T$, $\Dcal$ is upward-closed under the refinement relation, hence is a promise for~$T$. By clause~2
	of Definition~\ref{def:cohzp-forcing-condition}, $d = (\vec{F}, T, \Dcal) \Vdash^\alpha (\forall x) \neg \psi(G, w, x)$,
	hence $d \Vdash^\alpha \neg \varphi(G, w)$ for some $w \in A$.

	\item 
	In case $n > 0$, let $U = \{ w \in \omega : (\exists d \in \Ext(c)) d \Vdash^\alpha \varphi(G, w) \}$.
	By Lemma~\ref{lem:em-extension-complexity} and Lemma~\ref{lem:em-complexity-forcing},
	$U \in \Sigma^0_{n+2}$, thus $U \neq A$.
	Fix $w \in U \Delta A$. If $w \in U \setminus A$ then by definition of~$U$,
	there exists a condition $d$ extending $c$ such that $d \Vdash^\alpha \varphi(G, w)$.
	If $w \in A \setminus U$, then for every $d \in \Ext(c)$, $d \not \Vdash^\alpha \varphi(G, w)$
	so by clause~4 of Definition~\ref{def:cohzp-forcing-condition}, $c \Vdash^\alpha \neg \varphi(G, w)$. 
\end{itemize}
\end{proof}

We are now ready to reprove Corollary~3.29 from Wang~\cite{Wang2014Definability}.

\begin{theorem}[Wang~\cite{Wang2014Definability}]
$\rt^2_2$ admits preservation of $\Xi$ definitions 
simultaneously for all $\Xi$ in $\{\Sigma^0_{n+2}, \allowbreak \Pi^0_{n+2}, \allowbreak \Delta^0_{n+3} : n \in \omega \}$.
\end{theorem}
\begin{proof}
Since $\rca \vdash \coh \wedge \mathsf{D}^2_2 \imp \rt^2_2$,
and $\coh$ admits preservation of the arithmetic hierarchy, it suffices to prove
that~$\mathsf{D}^2_2$ admits preservation of $\Xi$ definitions 
simultaneously for all $\Xi$ in $\{\Sigma^0_{n+2}, \allowbreak \Pi^0_{n+2}, \allowbreak \Delta^0_{n+3} : n \in \omega \}$.
Fix some set~$C$ and a $\Delta^{0,C}_2$ 2-partition $R_0 \cup R_1 = \omega$.
Let $\Ccal_0$ be the $C'$-p.r.\ set of all $(\nu, T) \in \TPb$ such that $(\nu, T) \leq (0, 1^{<\omega})$.
Let~$\Fcal$ be a sufficiently generic filter containing~$c_0 = (\{\emptyset, \emptyset\}, 1^{<\omega}, \Ccal_0)$.
Let $G$ be the corresponding generic real.
By definition of a condition, the set~$G$ is $\vec{R}$-cohesive.
By Lemma~\ref{lem:cohzp-diagonalization} and Lemma~\ref{lem:cohzp-complexity-forcing},
$G$ preserves non-$\Sigma^0_{n+2}$ definitions relative to~$C$ for every~$n \in \omega$.
Therefore, by Proposition 2.2 of~\cite{Wang2014Definability}, $G$
preserves $\Xi$ definitions relative to~$C$
simultaneously for all $\Xi$ in $\{\Sigma^0_{n+2}, \Pi^0_{n+2}, \Delta^0_{n+3} : n \in \omega \}$.
\end{proof}

\vspace{0.5cm}

\noindent \textbf{Acknowledgements}. The author is thankful to Wei Wang for
interesting comments and discussions.
The author is funded by the John Templeton Foundation (`Structure and Randomness in the Theory of Computation' project). The opinions expressed in this publication are those of the author(s) and do not necessarily reflect the views of the John Templeton Foundation.

\end{document}